\documentclass[a4paper,11pt]{article}
\usepackage[utf8]{inputenc}
\usepackage{fullpage}
\usepackage{amsmath, amssymb, amstext, amsfonts, amsthm, bbm, array, enumerate,scalerel}
\usepackage{mathrsfs, dsfont,mathabx}%\mathscr schoen, \mathcal.. nicht so schoen
\usepackage{nicefrac}
\usepackage[none]{hyphenat}
\usepackage[dvipsnames]{xcolor}
\usepackage{tikz}
\usepackage{stmaryrd}
\usepackage{tikz-3dplot}
\usepackage{float}
\usepackage{appendix}
\usepackage{longtable}
\usepackage[round]{natbib}
\usetikzlibrary{patterns}
\usetikzlibrary{plotmarks}
\usepackage{shuffle}
\usepackage{ stmaryrd }
 \usepackage{bbm}
 \usepackage{bm}
\usepackage{xcolor}
\usepackage{graphicx}
\usepackage{comment}
\usepackage{framed}

\newtheoremstyle{boldremark}
    {\dimexpr\topsep/2\relax} % space above
    {\dimexpr\topsep/2\relax} % space below
    {}          % body font
    {}          % indent amount
    {\bfseries} % theorem head font
    {.}         % punctuation after theorem head
    {.5em}      % space after theorem head
    {}          % theorem hed spec. (empty = "normal")

\RequirePackage[colorlinks,citecolor=blue, urlcolor=blue, linkcolor=blue ]{hyperref}

\theoremstyle{plain}
\newtheorem{theorem}{Theorem}[section]
\newtheorem{lemma}[theorem]{Lemma}
\newtheorem{corollary}[theorem]{Corollary}
\newtheorem{proposition}[theorem]{Proposition}

\theoremstyle{definition}
\newtheorem{definition}[theorem]{Definition}
\newtheorem{example}[theorem]{Example}

\newtheorem{notation}[theorem]{Notation}

\newtheorem*{assumption*}{\assumptionnumber}

\providecommand{\assumptionnumber}{}
\makeatletter

\theoremstyle{boldremark}
\newtheorem{remark}[theorem]{Remark}

\numberwithin{equation}{section}

\def\PP{\mathcal{P}_R^d(0)^*}
\def\P{P_R^d(0)}

\def\SS{\mathcal{S}^*}
\def\S{\mathcal{S}}

\def\C{\mathbb C}
\def\Rset{\mathbb R}

\def\a{\mathbf{a}}

\def\j{\mathbf{j}}
\def\u{\mathbf{u}}
\def\c{\mathbf{c}}
\def\v{\mathbf{v}}

\def\b{\mathbf{b}}
\def\m{\mathbf{m}}

\def\Nset{\mathbb{N}}
\def\N{\mathbb{N}}

\def\la{\bm{\lambda}}
\def\sg{\bm{\sigma}}
\def\ppsi{\bm{\psi}}
\newcommand{\cadlag}{c\`adl\`ag~}

\newcommand{\R}{\mathbb R}
\newcommand{\E}{\mathbb E}

\newcommand{\A}{\mathcal A}
\newcommand{\Bcal}{\mathcal B}
\newcommand{\D}{\mathcal D}
\newcommand{\Acal}{\mathcal A}
\newcommand{\Dcal}{\mathcal D}
\newcommand{\Ocal}{\mathcal O}
\newcommand{\Scal}{\mathcal S}
\newcommand{\Vcal}{\mathcal V}
\newcommand{\Wcal}{\mathcal W}

\newcommand{\e}{\epsilon}
\newcommand{\vare}{\varepsilon}

\title{Holomorphic jump-diffusions}
\author{Christa Cuchiero \thanks{University of Vienna, Department of Statistics and Operations Research, Data Science @ Uni Vienna, Kolingasse 14-16, 1090 Vienna, Austria, christa.cuchiero@univie.ac.at
} 
\and Francesca Primavera  \thanks{University of California, Berkeley, Department of Industrial Engineering and Operations Research, Etcheverry Hall 4141, 94720, Berkeley, CA, USA, francesca.primavera@berkeley.edu}
\and Sara Svaluto-Ferro \thanks{University of Verona, Department of Economics, Via Cantarane 24, 37129 Verona, Italy, sara.svalutoferro@univr.it \newline
This research was funded in whole or in part by the Austrian Science Fund (FWF)
Y 1235.
}}
\date{}

\begin{document}
\maketitle

\begin{abstract}
 We introduce a class of jump-diffusions, called \emph{holomorphic}, of which the well-known classes of affine and polynomial processes are particular instances. The defining property concerns the extended generator, which is required to map a (subset of) holomorphic functions to themselves. This leads to a representation of the expectation of  power series of the process' marginals via a potentially infinite dimensional linear ODE.  We apply the same procedure by considering exponentials of holomorphic functions, leading to a class of processes named \emph{affine-holomorphic} for which a representation for quantities as the characteristic function of power series is provided.  Relying on powerful results from complex analysis, we obtain sufficient conditions on the process' characteristics which guarantee the holomorphic and affine-holomorphic properties and provide applications to several classes of jump-diffusions.

\end{abstract}

\tableofcontents

\section{Introduction}

The goal of this article is to introduce a class of jump-diffusion processes, called \emph{holomorphic}, for which the calculation of expected values of power series of the process' marginals, 
reduces to solving a sequence-valued linear ordinary differential equation (ODE). 
This thus constitutes an important extension of \emph{polynomial processes,} introduced in \cite{CKT:12, FL:16}, for which merely moments can be computed by solving a (finite dimensional) linear ODE.

At the core of our analysis lie duality considerations which
are a key concept in many areas of mathematics and have also played an important role in the analysis of stochastic processes. Indeed, duality theory for Markov processes with respect to a duality function goes back to several contributions in the early fifties, e.g., \cite{KM:57}, where classifications of birth and death processes are considered. Since then it  has been  extended in several directions (see e.g., \cite{HS:79, EK}) and applied in the context of interacting particle systems, queuing theory and population genetics.

The concept of dual processes can be formalized as follows: let $T >0$ be some finite time-horizon and consider two time-homogeneous Markov processes $(X_t)_{t \in [0,T]}$ and $(U_t)_{t \in [0,T]}$ with respective state spaces $S$ and $U$. Then $X$ and $U$ are \emph{dual} with respect to some measurable function $H: S \times U \to \mathbb{R}$ if for all $x \in S$, $u \in U$ and $t \in [0,T]$
\begin{align}\label{eq:dual}
\mathbb{E}_x[H(X_t,u)]=\mathbb{E}_u[H(x,U_t)],
\end{align}
holds and both, the left and right hand side are are well-defined. Here, $\mathbb{E}_x$ denotes the expected value for the Markov process $X$ starting at $X_0=x$ and similarly for $U$. Modulo technical conditions, the dual relation \eqref{eq:dual} holds if and only if the  generators denoted by $\mathcal{A}$ and $\mathcal{B}$ satisfy 
\begin{align} \label{eq:generator}
\mathcal{A} H(\cdot, u)(x) = \mathcal{B} H(x, \cdot)(u), \quad  \text{for all } x \in S, u \in U,
\end{align}
(see e.g.,~\cite{JK:14}). One prominent example is the Wright-Fisher diffusion, also called Jacobi process and denoted now by $X$, whose dual process with respect to $H(x,u)= x^u$ with $u \in \mathbb{N}$ is the Kingman coalescent $U$, i.e., we have 
\[
\mathbb{E}_x[X_t^u]=\mathbb{E}_u[x^{U_t}].
\]
 The Wright-Fisher diffusion is also an example of a polynomial process and for all functions $H(x,u): [0,1] \times \mathbb{R}^n \to \mathbb{R}$ given by $H(x,u)=\sum_{i=0}^{n-1} u_i x^i$ it  holds by the so-called \emph{moment formula} (see e.g., \cite{FL:16})
that 
 \[
\mathbb{E}_x[H(X_t,u)]=\mathbb{E}_x\left[\sum_{i=0}^{n-1} u_i X_t^i\right]= \sum_{i=0}^{n-1} c(t)_i x^i=H(x,c(t))=\mathbb{E}_u[H(x,c(t))],
\]
where $c(t)$ is the solution of a linear ordinary differential equation (ODE). 
This means  that in this case the dual process is given by the  coefficients $U_t=c(t)$ (with $U_0=c(0)=u$) of the polynomial $x \mapsto H(x,c(t))$, which are as solution of a linear ODE deterministic. This property actually defines polynomial processes in the sense that the expected value of a polynomial of $X_t$ for $t \in [0,T]$ is a polynomial in the initial value $X_0=x$. Another class of stochastic processes which admit a deterministic dual process with respect to certain functions $H$ are \emph{affine processes} (see \cite{DFS:03, CT:13}). In this case the duality functions  $H: S  \times U \to \mathbb{R}$ are given by $H(x,u)=\exp(\langle u,x \rangle)$ where $S \subseteq \mathbb{R}^d $, $U \subseteq \mathbb{C}^d$ and $\langle \cdot,\cdot \rangle$ denotes the scalar product on $\mathbb{R}^d$ (extended to $\mathbb{C}^d$). The corresponding dual process is a solution to a generalized Riccati ODE. 

These deterministic dual processes for $H$ ranging in the important and law determining\footnote{In the polynomial case the law is determined by the moments under certain exponential moment conditions, while the characteristic function in the affine case always determines the law.} function classes of polynomials and exponential functions (when viewed as functions in $x$) are the crucial property for the popularity of affine and polynomial processes for all kind of applications including finance,  population genetics and physics. Indeed, the computation of expected values for these functions $H$ reduces to solving simple deterministic ODEs.

It is therefore natural  to search of other stochastic processes and functions which admit a deterministic dual process.

The most natural extension of both polynomial and exponential functions are \emph{entire functions}, or more generally the class \emph{holomorphic functions}, which always admit a power series representation on $\mathbb{C}^d$, whose radius of convergence is not necessarily the whole space.  
The goal of the present paper is thus to define and specify a class of processes $X$ such that
\begin{align}\label{eq:holomorpicformula_intro}
\mathbb{E}_x[H(X_t,\u)]=H(x,\c(t)),
\end{align}
where $H$ is (the restriction to $\mathbb{R}^d$ of) a holomorphic function on $\mathbb{C}^d$,  given by the power series $H(x,\u)=\sum_{\alpha \in \mathbb{N}_0^d} \u_{\alpha} x^{\alpha}$ where $\alpha \in \mathbb{N}_0^d$ is a multi-index and $\u_{\alpha} \in \mathbb{C}$ are the coefficients. The reason why we consider holomorphic functions on $\mathbb{C}^d$ and not only power-series on $\mathbb{R}^d$ is that the uniform limit of a sequence of holomorphic functions is again holomorphic, a key property which is needed to guarantee property \eqref{eq:generator} in presence of jumps.

This rather wide extension of polynomial processes 
is twofold:
first, it allows go beyond them and second, we can establish an analog of the moment formula, which we call \emph{holomorphic formula} given by \eqref{eq:holomorpicformula_intro},  also for affine and certain polynomial processes (see Section~\ref{sec: examples}), including for instance the Wright-Fisher diffusion (see Section~6.2 of \cite{CST:23}). This means that for these well-known processes' classes, expectation of many more functions than just exponentials or polynomials can be computed analytically.
 
 In contrast to the polynomial case, $(\c(t))_{t \in [0,T]}$ in \eqref{eq:holomorpicformula_intro} is  a sequence-valued infinite dimensional deterministic process, given as a solution of an infinite dimensional linear ODE. 
 This already indicates that the analysis  becomes much more delicate, as we have to deal with existence of such solutions, convergence of the corresponding power series and many other subtleties from  complex analysis. 
 
 The fact that we consider \emph{jump-diffusions} instead of just continuous processes makes things even further involved. Indeed, already establishing the necessary condition~\ref{eq:generator}, which in the current setup means that the generator of $X$ maps holomorphic functions to holomorphic ones, is due to the appearance of jumps not at all straightforward (see Theorem~\ref{prop: holo generalkernel}). Indeed, one has to define specific subsets  of holomorphic functions, denoted by $\mathcal{V}$, where this holds true, see e.g., the set defined in \eqref{eq:Vgeneralkernel}. This is the very reason why we  consider  \emph{$\mathcal{V}$-holomorphic processes} which we define as solutions to the martingale problem specified by the (extended) generator $\mathcal{A}$ which maps $\mathcal{V}$ to general holomorphic functions. The structure of the compensator of the jump measure will allow different sets $\mathcal{V}$. The corresponding analysis in this sense  requires to apply a number of results of complex analysis and is subject of Section~\ref{sec: section 2.1} (see in particular Theorem~\ref{ps to ps}). 
For such $\mathcal{V}$-holomorphic processes we then establish the holomorphic formula \eqref{eq:holomorpicformula_intro}, under technical conditions involving in particular the existence of solutions to the sequence-valued linear ODEs and certain moment conditions. This can be viewed as a verification result (see Theorem~\ref{th: moment formula holomorphic}). To apply this result we establish further sufficient conditions which are then verified for L\'evy processes, affine processes and certain non-polynomial jump diffusions on compact sets (see Section ~\ref{sec: examples}), which is one of the main contributions of this paper.
 
In analogy to affine processes we can of course also consider processes  where the expected value of the exponential of a holomorphic function is given as the exponential of a holomorphic function whose coefficients solve a sequence-valued Riccati ODE. 
This is subject of Section~\ref{sec: affine holomorphic} where we introduce the class of \emph{affine-holomorphic processes},
explain their relation to holomorphic ones and elaborate on their characteristics as well as appropriate subsets of holomorphic functions that are mapped to (general) holomorphic functions by the Riccati operator. Here, the jumps are again 
the essential part which makes the analysis intricate (see Theorem~\ref{theorem: new}). Under technical conditions, involving in particular the existence of solutions of the infinite dimensional Riccati ODEs, we then prove the so-called affine-holomorphic formula (see Theorem~\ref{th: affine formula}).

In terms of applications our results can be used for pricing and hedging real analytic claims in finance  and new duality relations in population genetics as e.g., in \cite{CSP:18, BCGKW:14}. Reading the duality relation backwards, the (affine)-holomorphic formula also allows to develop numerical schemes for solving infinite dimensional linear and Riccati ODEs based on the stochastic representations.

 The remainder of the paper is organized as follows. In the subsections below we clarify the relations to the literature and fix the terminology related to holomorphic functions. Section~\ref{sec:preliminaries} introduces all necessary notation, general jump-diffusion processes and convergent power series for given state spaces. Section~\ref{sec:holomorphic_jd} and Section~\ref{sec: affine holomorphic}
are then dedicated to the analysis of holomorphic and affine-holomorphic processes respectively. The paper concludes with several appendices, containing in particular the proofs and important results from complex analysis (see Appendix~\ref{appendix:holofunctions}).
 
\subsection{Relation to the literature}

Our expressions for the expected value of holomorphic functions either in terms of the holomorphic formula or the affine-holomorphic formula are related to several recent works in the literature. Indeed, in the one-dimensional case similar results have been obtained in \cite{CST:23}, however only for continuous processes and  real analytic functions, which is a significantly simpler setting. Note that the main focus of \cite{CST:23} are actually continuous signature SDEs  which reduce in the one dimensional case to continuous holomorphic processes. We also refer to Remark~\ref{rem: poly processes tensor algebra} for another relation with signature SDEs. Developing signature jump-diffusions and in turn the theory of 
holomorphic jump-diffusions on the extended tensor algebra is subject of future work and already partly realized in Chapter~3 in \cite{P:24}.

Our expressions for the
expected values of power series in terms of power series are also related to similar expansions, obtained e.g., in \cite{FHT:22, FGR:22, FG:22, FM:21,  AGR:20}. For applications to Fourier pricing in finance relying on infinite dimensional Riccati equations  we refer to \cite{AG:24, ALL:24}.

In terms of polynomial processes which -- as already mentioned -- constitute a  special case of holomorphic processes, let us in particular mention the recent paper \cite{BDK:20} where an abstract far reaching approach to polynomial processes is proposed. This covers also infinite dimensional polynomial processes considered  in \cite{BDK:18, CS:21, CDGS:24} as well. The latter can also be related to lifts of polynomial Volterra processes which were recently  analyzed in \cite{ACPPS:24}.

\subsection{Terminology}

We here aim to recall several notions related to holomorphic functions and fix some terminology which we shall use throughout.

A holomorphic function with values in $\C$ is a map defined on (some subset of) the complex plane that is complex differentiable at every point within its domain (see Definition~\ref{def: holo functions}, where more generally $\C^m$-valued holomorphic functions are introduced). Despite the resemblance of this class with the class of just differentiable functions
of one real variable, the former satisfy much stronger properties. A holomorphic function is actually infinitely many times complex differentiable, that is, the existence
of the first derivative guarantees the existence of the derivatives of any
order. In fact, more
is true: every holomorphic function is \emph{complex analytic}, in the sense that it has a
power series expansion near every point. This is also the reason why the term (complex) analytic is frequently used as a synonym for holomorphic. Moreover, any real analytic function on some open set on the real line can be extended to a complex analytic, and thus holomorphic, function on some open set of the complex plane. (However, not every real analytic function defined on the whole real line can be extended to a complex function defined on the whole complex plane). In particular, any real convergent power series can be extended to a complex holomorphic function on some open disc of the complex plane.

\section{Preliminaries}\label{sec:preliminaries}

This section is primarily dedicated to introduce  necessary notation and the notion of jump-diffusion processes with which we shall work.

\subsection{Notation}\label{sec:notation}
\subsubsection{Multi-index notation} \label{sec: multi index notation}
Fix $d\in \N$. For $z\in \C^d$  we write $z_1,\ldots, z_d$ for the components of $z$ and set 
$$|z|:=(|z_1|,\ldots,|z_d|)\qquad\text{and}\qquad\|z\|:=\sqrt{|z_1|^2+\ldots+|z_d|^2},$$
where $|z_j|$ denotes denote the modulus of $z_j$.
For a multi-index $\alpha=(\alpha_1,\dots,\alpha_d)\in \N^d_0$ we use the following notation: $|\alpha|:=\alpha_1+\dots+\alpha_d$, $\alpha!:=\alpha_1!\dots\alpha_d!$, $z^\alpha:=z_1^{\alpha_1}\dots z_d^{\alpha_d}$ for  $z\in \C^d$, and for $d=1$, we disregard the brackets and simply write $\alpha\in \N_0$. We set
\begin{align*}
      h^{(\alpha)}(z):=D^\alpha h (z)=\frac{\partial D^{|\alpha|}h}{\partial ^{\alpha_1}z_1\dots\partial ^{\alpha _d}z_d } (z),
\end{align*}
for sufficiently regular maps $h:U\rightarrow \C$ and some open set $U\subseteq \C^d$ (or $U\subseteq \Rset^d$), with $z\in U$, and  $\|h\|_\infty:=\sup_{z\in U}\|h(z)\|$.
Moreover, we write $\nabla h(z)$ and $\nabla^2 h(z)$ to denote the gradient and Hessian matrix of $h$ at $z$, respectively. For $d=1$ we also write $h'$ and $h''$, respectively.

\subsubsection{The space of sequences}\label{sec212}
We introduce the notation for representing power series in $d$ variables. Throughout, we use bold letters to denote sequences $\u:=(\u_\alpha)_{\alpha\in\N_0^d}$ with $\u_
\alpha\in\C$ indexed by multi-indices and set $|\u|:=(|\u_\alpha|)_{\alpha\in\N_0^d}$. To denote vectors and matrices of sequences we then make use of underlining. Specifically, we write
$$\underline \u:=(\underline\u^1,\ldots,\underline\u^d)\qquad\text{and}\qquad \underline{\underline \u}:=\begin{pmatrix}
\underline{\underline \u}^{11} & \cdots & \underline{\underline \u}^{1d}\\
\vdots&\ddots&\vdots\\
\underline{\underline \u}^{d1} & \cdots & \underline{\underline \u}^{dd}
\end{pmatrix},$$ 
where $\underline\u^i=(\underline\u^{i}_{\alpha})_{\alpha\in\N_0^d}$ and $\underline{\underline\u}^{ij}=(\underline{\underline\u}^{ij}_{\alpha})_{\alpha\in\N_0^d}$ for $\underline\u^{i}_{\alpha},\underline{\underline\u}^{ij}_{\alpha}\in \C$ . We let $\mathbf{1}:=(\mathbf{1}_{\alpha})_{\alpha\in \N^d_0}$ denote the sequence such that $\mathbf{1}_{\alpha}=1$ if $\alpha=(0,\dots,0)$ and $\mathbf{1}_{\alpha}=0$ otherwise. To simplify the notation, we denote by $(\epsilon_i)_{i\in\{1,\dots,d\}}$ the canonical basis of $\Rset^d$ and write
$$\underline \u^{\epsilon_i}:=\underline\u^{i},\qquad \text{and}\quad \underline {\underline\u}^{2\epsilon_i}:=\underline {\underline\u}^{ii},\qquad \underline {\underline\u}^{\epsilon_i+\epsilon_j}:=\underline {\underline\u}^{ij}.$$
This is useful to access also the components of $\underline \u$ and  $\underline{\underline \u}$ via the multi-index-notation, namely $\underline \u^\beta$ and $\underline{\underline \u}^\beta$ for $|\beta|=1,2$, respectively. 

Finally, in this paper integrals of sequence-valued maps are always computed componentwise. Precisely, given a measure $F$ on a measurable space $E$ and a map $( (\mathbf{f}(\cdot))_\alpha)_{\alpha\in \N_0^d}$ on $E$ such that $ (\mathbf{f}(y))_\alpha\in\C$,  we define 
$$\Big(\int_{E}{\mathbf{f}}(y) \ F(dy)\Big)_\alpha:= \int_{E}({\mathbf{f}}(y))_\alpha \ F(dy),$$
whenever the involved quantities are well defined. The same extends to vector and matrices of sequence valued maps.

\subsubsection{The set of holomorphic functions on polydiscs}\label{sec213}
Let $d\in \N$, $R=(R_1,\dots, R_d)\in (0,\infty] ^d$, $a\in \C^d$ and denote by $P^d_R(a)$ the complex \emph{polydisc} centered at $a$ with \emph{polyradius} $R$:
\begin{align*}
&P_R^d(a):=\{z\in \C^d \colon |z_i-a_i|< R_i \ \text{for all }i\in\{1,\dots,d\}\},
\end{align*}
where $|\cdot|$ denotes the complex modulus. This in particular implies that
\begin{align*}
    P_R^d(a)=P_{R_1}^1(a_1)\times \dots\times P_{R_d}^1(a_d)\subseteq \C^d, 
\end{align*}
where $P_{R_j}^1(a_j)$ denotes the complex disc centered at $a_j$ and radius $R_j$. Observe that for $R_j=\infty$ we get $P_{R_j}^1(a_j)=\C$.
If $R\in (0,\infty)^d$ we let $T_R^d(a)$ denote the boundary of the polydisc $P_R^d(a)$, also called \emph{polytorus}, defined via
\begin{equation}\label{eqn5}
    T_R^d(a):=\{z\in \C^d \colon |z_i-a_i|= R_i \ \text{for all }i\in\{1,\dots,d\}\}.
    \end{equation}
    The closure $\overline{P_R^d(a)}$ of the polydisc is  given by $\overline{P_R^d(a)}=\{z\in \C^d \colon |z_i-a_i|\leq R_i \ \text{for all }i\}$.
    Moreover, for any $R\in (0,\infty]$ we denote the polyradius $(R,\ldots,R)$ again by $R$. Given two polyradii  $M,N\in (0,\infty]^d$, we write $M>N$ if for all $j=1,\dots,d$ $M_j>N_j$.
  
We denote by $H(P_R^d(0), \C^m)$, $m\in \N$, the set of holomorphic functions on $P_R^d(0)$ with values in $\C^m$ (see Definition~\ref{def: holo functions}). When $m=1$, we simply write  $H(P_R^d(0))$. When a holomorphic function on a polydisc has a power series representation, it can be identified through the corresponding coefficients (see Proposition~\ref{propclass}\ref{class:it3}).
Setting
\begin{equation}\label{eqn2}
    |\u|_z:=\sum_{\alpha\in \N_0^d}\Big|\u_\alpha \frac{z^\alpha}{\alpha!}\Big|,
\end{equation}
 we define
\begin{align*}
    \mathcal{P}_R^d(0)^*&:=\{\u=(\u_\alpha)_{\alpha\in\N_0^d}\ :\ \u_\alpha \in \mathbb{C} \text{ and }|\u|_z <\infty  \text{  for all } z\in P_R^d(0)\}.
\end{align*}
For $\u\in \mathcal{P}_R^d(0)^*$, we let $h_\u\in H(P^d_R(0))$ denote the holomorphic function on $P_R^d(0)$ determined by the corresponding convergent power series
\begin{align*}
    h_\u(z):=\sum_{\alpha\in \N_0^d}\u_\alpha \frac{z^\alpha}{\alpha!}, \qquad z\in P_R^d(0).
\end{align*}
Similarly, we write $\underline{\u}\in   (\PP)^d$, $\underline{\underline {\u}}\in (\PP)^{d \times d}$ if each $\underline{\underline {\u}}^{ij},\underline{\u}^i\in \PP$ and denote by $h_{\underline{\u}}\in H(P^d_R(0),\C^d)$ and $h_{\underline{\underline {\u}}}\in H(P^d_R(0),\C^{d\times d})$ the corresponding holomorphic functions.

\subsubsection{Operations on holomorphic functions and sequences}\label{sec 214}
Now we introduce the operations on holomorphic functions that will be used throughout the paper and describe how these can be translated into operations on the corresponding sequence of coefficients, and vice-versa. The relevant operations in our setting are algebraic operations, differentiation and exponentiation.
\begin{itemize}
\item {\textbf{Linear operations}:} Through componentwise linear operations we obtain the relation
$$\lambda h_{\u}(z)+\mu h_{\v}(z)=h_{\lambda \u+\mu \v}(z)$$
for each $z\in \P$, $\u,\v\in \PP$ and $\lambda,\mu\in \C$.
\item{\textbf{Product}:} We consider the symmetric bilinear map on $\PP$ given by
\begin{align*}
\PP &\times \PP \longrightarrow \PP\\
&(\u,\v)\longmapsto\u\ast\v
\end{align*}
where $\u\ast\v:=((\u\ast\v)_\alpha)_{\alpha\in \N_0^d}$ for 
\begin{align*}(\u\ast\v)_\alpha=\sum_{\gamma+\beta=\alpha} \frac{\alpha!}{\beta!\gamma!}\u_{\beta}\v_{\gamma}, \quad \alpha\in\N_0^d.
\end{align*}
Notice that for all $z\in \P$ and $k\in \N$
$$h_{\u}(z)h_{\v}(z)=h_{\u\ast\v}(z)\qquad\text{and}\qquad  (h_{\u}(z))^k=h_{\u^{\ast k}}(z),$$
where $\u^{\ast 1}=\u$ and  $\u^{\ast k}:=\u^{\ast (k-1)}\ast\u$ for each $k>1$. We also set $\u^{\ast 0}:=(1,0,0,\ldots)$.

Through the multi-index notation, these representations extend to vector-valued functions. More precisely, setting $\underline \u:=(\underline\u^1,\ldots,\underline\u^d) \in (\PP)^{d}$ we define $\underline{\u}^{\ast \beta}:=(\underline\u^1)^{\ast \beta_1}\ast\cdots\ast (\underline\u^d)^{\ast \beta_d}$ and then get
$$(h_{\underline{\u}}(z))^\beta=h_{\underline{\u}^{\ast \beta}}(z).$$

     \item {\textbf{Differentiation}:} For $\u \in \PP$ and $\beta\in \N^d_0$ we set  $\textbf{u}^{(\beta)}:=(\u_{\alpha+\beta})_{\alpha\in \N_0^d}$. Observe that $\u^{(\beta)}\in\PP$ and 
 \begin{align*}
    h_\textbf{u}^{(\beta)}(z)=\sum_{\alpha\in \N^d_0} \u_{\alpha+\beta }\frac{z^\alpha}{\alpha!} =h_{\textbf{u}^{(\beta)}}(z),
\end{align*}
for all $z\in \P$.

\item {\textbf{Exponentiation}:}  For $\u\in\PP$, we denote by $\textbf{exp*}(\u)\in\PP$ the coefficients determining the power series representation on $P^d_R(0)$ of $\exp(h_\u)$, namely $$\exp(h_\u(z))=h_{\textbf{exp*}(\u)}(z),$$
for all $z\in \P$.
\item {\textbf{Composition}:} 
Fix $\underline{\v}\in ( {\mathcal{P}_R^d(0)}^*)^d$ and $\u\in  {\mathcal{P}^d_N(0)}^{*}$ for $N$ such that
$$N_j>\sup_{z\in \mathcal{P}_R^d(0)}\big(|z_j|+|h_{\underline \v}(z)_j|\big).$$
Then $ h_\u(\cdot+h_{\underline{\v}}(\cdot))\in H(\P)$. If we additionally assume that 
$$N_j>\sup_{z\in \mathcal{P}_R^d(0)}\big(|z_j|+| \underline\v^j|_z\big),$$
setting
$$(\u\circ^s\underline{\v})_\alpha:= \sum_{\beta \in \N^d_0}\frac{1}{\beta!}  (\u^{(\beta)} \ast \underline{\v}^{\ast \beta})_\alpha$$
it holds $\u\circ^s\underline{\v}\in\PP$ and
$$h_{\u\circ^s\underline{\v}}(z)=h_\u(z+h_{\underline{\v}}(z))$$
for all $z\in \P$. The superscript $\ ^s$ is mnemonic for ``shift''.

\begin{proof}
The first claim follows since the composition of holomorphic functions is holomorphic (see Proposition 1.2.2. in \cite{SC:05}). Next, observe that since for all $z\in \P$ we have $|z_j|+h_{|\underline{\v}^j|}(|z|)<N_j$, we get $h_{|\u|}(|z|+h_{|\underline{\v}|}(|z|))<\infty$. Moreover,  by Proposition~\ref{propclass}\ref{class:it3} for all $z\in \P$ we get
 \begin{align*}
    h_{|\u|}(|z|+h_{|\underline{\v}|}(|z|))=&\sum_{\beta\in \N^d_0} \frac{1}{\beta!}  h^{(\beta)}_{|\u|}(|z|)h_{|\v|}(|z|)^{\beta}\\
    =&\sum_{\beta\in \N^d_0} \frac{1}{\beta!}\sum_{\alpha\in\N^d_0} (|\u|^{(\beta)}\ast |\underline{\v}|^{\ast \beta} )_\alpha \ \frac{|z|^\alpha}{\alpha!}\\
    =&\sum_{\alpha\in\N^d_0}\Big( \sum_{\beta\in \N^d_0} \frac{1}{\beta!}(|\u|^{(\beta)}\ast |\underline{\v}|^{\ast \beta} )_\alpha\Big)  \frac{|z|^\alpha}{\alpha!}.
\end{align*}
 Thus, $|\u|\circ^s|\underline{\v}|\in \PP$. Since   $|(\u\circ^s\underline{\v})_\alpha|\leq (|\u|\circ^s|\underline{\v}|)_\alpha$ the claim follows.
\end{proof}
\end{itemize}

\subsection{Jump-diffusion processes}\label{sec: jump-diff}
Let $S\subseteq\Rset^d$, $b:S \rightarrow\Rset^d$, $a:S \rightarrow\mathbb{S}^d_+$ measurable functions and $K(\cdot,d\xi)$ a transition kernel from $S$ to $\Rset^d$ which satisfies $K(x,\{0\})=0$,  and $\int_{\Rset^d} \|\xi\|\wedge\|\xi\|^2K(x,d\xi)<\infty,$ for all $x\in S$. 
Let $M(S;\C)$ denote the set of measurable maps on $S$ with values in $\C$, and consider the operator $\Acal:\Dcal(\Acal)\to M(S;\C)$ such that 
\begin{align*}
  \mathcal{A}f(x)=\nabla f(x)^\top b(x)+\frac{1}{2}\text{Tr}(a(x)\nabla ^2f(x))+\int_{\Rset^d}f(x+\xi)-f(x)-\nabla f(x)^\top\xi  \  K(x,d\xi), 
\end{align*}
for each $ x\in S$ and $f\in \mathcal{D}(\mathcal{A})$, where 
\begin{align*}
    \mathcal{D}(\mathcal{A}):=\{f\in C^2(\Rset^d;\mathbb{C})  \ : \forall x\in S \ \int_{\Rset^d }|f(x+\xi)-f(x)-\nabla f(x)^\top \xi |\ K(x,d\xi)<\infty \}.
\end{align*}
In particular, observe that $\D(\A)$  includes all bounded $f\in C^2(\Rset^d;\mathbb{C})$. 

Fix then $T>0$. We say that $X=(X_t)_{t\in[0,T]}$  is an \emph{$S$-valued jump-diffusion} with characteristics
\begin{align*}
    \Bigg(\int_0^\cdot b(X_{s^-})ds,\int_0^\cdot a(X_{s^-})ds,K(X_{s^-},d\xi)ds\Bigg)
\end{align*}
if $X$ is a special \cadlag semimartingale on some filtered probability space $(\Omega, \mathcal{F},(\mathcal{F}_t)_{t\in[0,T]},\mathbb{P})$ such that for all bounded $f\in C^2(\Rset^d;\mathbb{C})$ the process $N^f:=(N^f_t)_{t\in[0,T]}$ given by
\begin{align}\label{eq:localmart}
  N_t^f:=  f(X_t)-f(X_0)-\int_0^t \A f(X_s)ds, \quad t\in [0,T]
\end{align}
defines a local martingale\footnote{See  Theorem II.2.42 in \cite{JS:87}.}. We refer to $\A$ as the \emph{extended generator} of $X$. In this paper, we stick to the truncation function $\chi (\xi)=\xi$ and for simplicity we refer to the coefficients $(b,a,K)$ as characteristics of the the jump-diffusion $X$.

Next, we introduce a class of functions that will form a subset of the domain of the extended generator of the process under consideration and discuss some key properties of the transition kernels.

\subsubsection{Convergent power series on a given set}\label{sec: convergent power series on a given set}
Here we introduce the space of holomorphic functions defined on polydiscs containing a given subset $ S\subseteq  \Rset^d$, which will serve as the state space of a jump-diffusion process. For  $i\in \{1,\dots,d\}$ set
$$R_i(S):=\sup\{|x_i|\ :  x\in S\}\qquad \text{and}\qquad R(S):=(R_1(S),\dots,R_d(S)).$$
Note that  
 $P^d_{R(S)+{\varepsilon}}(0)$ includes  $S$ for each $\varepsilon>0$, and define
 $$\Ocal(S):=\bigcup_{\vare>0}\Ocal_\vare(S)$$
 for
\begin{align*}
    \mathcal{O}_\varepsilon(S):=\Big\{f:S\to \mathbb{C}\colon f=h|_S \text{ for some } h\in H(P^d_{R(S)+{\varepsilon}}(0))\Big\}.
\end{align*}
For a function $f:A\to \C$ for some $S\subseteq A\subseteq \C^d$, we write $f\in \Ocal(S)$ if $f|_S\in \Ocal(S)$. Observe that given $f_1,f_2\in\mathcal{O}(S)$ it holds $\lambda f_1+\mu f_2\in \Ocal(S)$ for each $\lambda,\mu\in \R$, showing that $\Ocal(S)$ is a linear space. Maps in $\Ocal(S)$ have two important properties:
\begin{enumerate}
    \item they admit a power series representation: setting
\begin{align}\label{eq: set Sstar}
    \S^*:=\{\u\in \mathcal{P}^d_{R(S)}(0)^*  \colon h_\u=h|_{P^d_{R(S)}(0)}  \text{ for some } h\in H(P^d_{R(S)+{\varepsilon}}(0))\text{ and } \varepsilon>0\},
\end{align}
it holds that for all $f\in \mathcal{O}(S)$, $f=h_\u|_S$, for some $\u\in \SS$  (see Proposition~\ref{propclass}\ref{class:it3}). Notice that the power series representation of each $f\in \mathcal{O}(S)$ might not be unique. For the purposes of this paper, this is however not relevant. We only need such a representation to exist (see Section~\ref{sec: the holomorphic formula} and in particular Remark~\ref{rem2}\ref{rem2_2});
\item\label{it:drift0} on $S$ they coincide with the restriction of a complex-valued smooth function on $\R^d$, making them eligible to be elements of $\Dcal(\Acal)$. It is important to observe that if $\Acal$ is the extended generator of $X$, then $\Acal f|_S$ just depends on $f|_S$ for all $f\in \mathcal{D}(\A)$. This can be proved using that for each $f\in \mathcal{D}(\A)$ such that $f|_S\equiv0$ the process $(f(X_t))_{t\geq 0}$ needs to have vanishing drift. For similar results in different contexts see for instance Theorem~2.8 in \cite{CLS:18} or the discussion after Definition 2.3 in \cite{LS:20}. This observation in particular implies that $K(x,(S-x)^c)=0$, where $(S-x)^c$ denotes the complement of the set $S-x$.
\end{enumerate}
\subsubsection{Kernels extension}
Similarly as in the previous section where we were interested in the restriction to $S$ of holomorphic functions defined on polydiscs containing $S$, we are now interested in kernels with the same type of property. 

For a transition kernel $K(x,d\xi)$ from $S$ to $\R^d$ we write $K\in\Ocal_\vare(S)$ if for all $x\in S$ , $K(x,d\xi)=K_\vare(x,d\xi)$, for some transition kernel $K_\vare(z,d\xi)$ from $S_\vare$ to $\C^d$ where $S_\vare$ denotes an open set in $\mathbb{C}^d$ such that $S\subseteq S_\vare \subseteq  P^d_{R(S)+{\vare}}(0)$, for some $\vare>0$. Moreover, we require that for all $|\beta|\geq 2$
\begin{equation}\label{eqn1}
\int_{\C^d}\xi^\beta K_\vare(\cdot,d\xi)\in \Big\{f:S_\vare \to \C\colon f=h|_{S_\vare} \text{ for some } h\in H(P^d_{R(S)+{\vare}}(0))\Big\}.
\end{equation}
Observe in particular that for each $K\in \Ocal_\vare(S)$ it holds
$$\int_{\R^d}\xi^\beta K(\cdot,d\xi)\in \Ocal_\vare(S),
$$
for each $|\beta|\geq2$. 
Also in this context we write $\Ocal(S):=\bigcup_{\vare>0}\Ocal_\vare(S)$.

\section{Holomorphic jump-diffusions}\label{sec:holomorphic_jd}
Fix now $S\subseteq \R^d$ and $T>0$. 
\begin{definition}\label{def:holomorphic process}
    Let $X=(X_t)_{t\in[0,T]}$  be an $S$-valued jump-diffusion with extended generator $\mathcal{A}$ and fix a linear subset $\mathcal{V}\subseteq \mathcal{O}(S)$.  We say that $X$ is an $S$-valued $\mathcal{V}$-\textit{holomorphic process} if for each $f\in\Vcal$ it holds $f\in\Dcal(\Acal)$,  $\mathcal{A}f\in \mathcal{O}(S)$, and  the process $N^f$ introduced in equation \eqref{eq:localmart} defines a local martingale.
\end{definition}
Let $X=(X_t)_{t\in [0,T]}$ be an $S$-valued $\mathcal{V}$-holomorphic process, for some linear subset $\mathcal{V}\subseteq \mathcal{O}(S)$ and set 
\begin{align}\label{eq: V^*}
    \mathcal{V}^*:=\{\u\in \SS \colon h_\u=h|_{{P}^d_{R(S)}(0)},  \ h\in \mathcal{V}\},
\end{align}
for $\SS$ introduced in equation \eqref{eq: set Sstar}. Following the discussion on the properties of the functions in $\mathcal{O}(S)$ in Section~\ref{sec: convergent power series on a given set}, one can see that there always exists a linear map $ L: \mathcal{V}^*\rightarrow \mathcal{O}(S),$
such that for all $\u\in\mathcal{V}^*$,
\begin{align}\label{eq: Ah_u=h_Lu}
   \mathcal{A}h_\u=h_{L(\u)}|_S. 
\end{align} Observe that since by the identity theorem for holomorphic functions (see Proposition~\ref{propclass}\ref{class:it identity}) the power series representation of the maps in $\mathcal{O}(S)$ is not unique, there might be more than one map $L$ (corresponding to different representations of the functions $\mathcal{A}h$ on $S$, $h\in \mathcal{V}$) for which \eqref{eq: Ah_u=h_Lu} holds true. However, we only need such a map to exist (see Section~\ref{sec: the holomorphic formula} and in particular Remark~\ref{rem2}\ref{rem2_2}).

\begin{remark}\label{rem: poly processes tensor algebra}
    Recall the notion of polynomial jump-diffusions given in Definition 1 in \cite{FL:20} and note that the class of holomorphic processes is a considerable extension of the class of polynomial jump-diffusions. It is also interesting to observe that for an $S$-valued $\mathcal{V}$-holomorphic process $X=(X_t)_{t\in[0,T}$, the infinite-dimensional process 
    \begin{align}\label{eq: boldX infinite}
        \mathbb{X}:=((1, \mathbb{X}_t^{1},\mathbb{X}^{2}_t,\dots,\mathbb{X}^{i}_t,\dots))_{t\in [0,T]},
    \end{align}
    where, for $i\in \N$ and all $t\in[0,T]$, $$\mathbb{X}^{i}_t:=\Big(\frac{1}{\alpha^1!}X_t^{\alpha^1},\dots,\frac{1}{\alpha^{k_i}!}X_t^{\alpha^{k_i}}\Big)\in \Rset^{k_i},\qquad \alpha^j\in \N^d_0\text{ with }|\alpha^j|=i$$ 
    where $k_i$ is the number of multi-indices $\alpha\in \N^d_0$ such that $ |\alpha|=i$,
    can be interpreted as a $\mathcal{V}^*$-polynomial jump-diffusion on the extended tensor algebra of $\Rset^d$, in the sense of Definition~3.17 in \cite{CST:23}, for $\mathcal{V}^*$ defined in \eqref{eq: V^*}.
\end{remark}

    \subsection{Characteristics of holomorphic jump-diffusions}\label{sec: section 2.1}
In this section, we provide a discussion on some sufficient and necessary conditions for a jump-diffusion to be a holomorphic process. We do not elaborate on existence results here.  We assume instead the existence of a jump-diffusion process on a certain state space $S$ and study the conditions on its characteristics such that the corresponding process is a holomorphic one. We start with a simple result establishing necessary conditions.
\begin{lemma}\label{lemma:finitepoments}
    Consider a subset $\Vcal$ such that
    \begin{align}\label{eq:poly_subset_V}
        \{p|_S\text{  for some polynomial }p:\R^d\to \R\}\subseteq \Vcal\subseteq \Ocal(S)
    \end{align}
    and let $X=(X_t)_{t\in [0,T]}$ be an $S$-valued $\mathcal{V}$-holomorphic process. Then $\int_{\Rset^d}\|\xi\|^k \ K(x,d\xi)<\infty$ for all  for all $k\geq 2$ and $x\in S$, and the characteristics $(b,a,K)$ of $X$ satisfy
 \begin{align}\label{eq:character_holo}
        &b_i(\cdot)\in \mathcal{O}(S), \quad \text{for all }i\in\{1,\dots,d\},\\
        &a_{ij}(\cdot)+\int_{\Rset^d}\xi_i\xi_j K(\cdot,d\xi)\in\mathcal{O}(S), \quad \text{for all }i,j\in\{1,\dots,d\}^2,\nonumber \\
        &\int_{\Rset^d}\xi^\beta  K(\cdot,d\xi)\in\mathcal{O}(S), \quad  \text{for all }\beta\in \N^d_0,\ |\beta|\geq3. \nonumber 
    \end{align}
\end{lemma}
\begin{proof}
    The proof follows the proof of Lemma~1 in \cite{FL:20} by applying $\mathcal{A}$ to all monomials.
\end{proof}

\begin{remark}\label{rem: poly processes are holomorphic}
Each $S$-valued polynomial jump-diffusion is an $S$-valued $\Vcal$-holomorphic process for $\Vcal=   \{p|_S\text{  for some polynomial }p:\R^d\to \R\}$. However, holomorphic processes allow for more general drift and  diffusion coefficients (e.g.~higher order polynomials and holomorphic functions) as well as more general kernels. 
For example, excluding   higher order polynomials from $\Vcal$ would permit to include in the class of holomorphic processes jump-diffusions with kernels that do not admit all moments (see Corollary~\ref{prop:unboundedjs} and the discussion thereafter). An extreme example in this sense is given by $\Vcal$ consisting entirely of bounded holomorphic functions. This shows how holomorphic processes substantially extend the class of polynomial jump-diffusions.
\end{remark}
In the classical polynomial case, if the characteristics satisfiy the type of conditions expressed by equation \eqref{eq:character_holo}, the corresponding jump-diffusion process is automatically polynomial (see Lemma 1 in \cite{FL:20}).
For holomorphic processes, this is not the case. Indeed, in addition to the holomorphic dependence of the characteristics  on the state variables  further assumptions have to be made. Before illustrating this in the following theorem,
we introduce some notation that will be used throughout the paper.

\begin{notation}\label{not1}
    Throughout, given an $S$-valued jump-diffusion $X$ with  drift and diffusion $(b,a)$ such that $b_i\in \Ocal(S)$ and $a_{ij}\in \Ocal(S)$, we let 
    $$
     \underline{\b}\in (\S^*)^d, \qquad \underline{\underline{\a}}\in (\S^*)^{d\times d}
    $$ denote some coefficients determining the power series representation of $b$ and $a$, respectively. This in particular implies that
    $$b_i=h_{\underline \b^i}|_S\qquad\text{and}\qquad a_{ij}=h_{\underline {\underline \a}^{ij}}|_S.$$
    For each $|\beta|\geq2$ and  transition kernel $K\in\mathcal O_\vare(S)$ we then denote by $\m^\beta\in \mathcal{P}^d_{R(S)+\vare}(0)^*$ 
     the coefficients determining the power series representation  of   $\int_{\C^d}\xi^\beta K_\vare(\cdot,d\xi)$ and thus satisfying
    $$\int_{\C^d}\xi^\beta K_\vare(\cdot,d\xi)
    = h_{\m^\beta}|_{S_\vare}.$$
    Similarly, for a  transition kernel $K$  with  $ \int_{\Rset^d}\xi^\beta K(\cdot,d\xi)\in \mathcal{O}(S),$ we let
    $   \m^\beta\in \S^*$
    be some coefficients such that $\int_{\Rset^d}\xi^\beta K(\cdot,d\xi)
    =
    h_{\m^\beta}|_S$.
\end{notation}

Notice that by the assumptions in Section~\ref{sec: jump-diff}, for all $\alpha\in \N^d_0$ it holds $\underline{\underline \a}^{ij}_\alpha$, ${\underline \b}^i_\alpha, \m^\beta_\alpha \in \mathbb{R}$.
The proof of the following theorem can be found in Appendix~\ref{proofpropgkernel}. 
\begin{theorem}\label{prop: holo generalkernel}
Let $X=(X_t)_{t\in [0,T]}$ be an $S$-valued jump-diffusion with characteristics $(b,a,K)$ and extended generator $\Acal$. Fix $\vare>0$ and assume that $b_j,a_{ij},K\in \Ocal_\vare(S)$.
Suppose that the following conditions hold true.
\begin{enumerate}
\item \label{iitheorem1} for all $z\in {S_\vare}$, $ \int_{\|\xi\|>1} \exp(|\xi_1|+\dots+|\xi_d|) K_\vare(z,d\xi)<\infty;$
    \item \label{iiitheorem1}the real valued map
$\sum_{|\beta|\geq2}\frac{1}{\beta!}|h_{\m^\beta}(\cdot)|
$
is locally bounded on $P^d_{R(S)+\vare}(0)$.
\end{enumerate}
Fix $G=(G_1,\ldots,G_d)$ such that 
\begin{equation}\label{eqGj}
    G_j>\sup_{z\in {S_\vare}}\sup_{\xi\in \textup{supp}(K_\vare(z,\cdot))}|z_j|+|\xi_j|,
\end{equation}
and set
\begin{equation}\label{eq:Vgeneralkernel}
\begin{aligned}
    \mathcal{V}&\subseteq\{h\in H(P^d_G(0))\ \colon
      (h^{(\beta)})_{|\beta|\geq2} \text{ is locally uniformly bounded on }P^d_G(0)\},
      \\
      \mathcal{V}^*&:=\{\u\in \SS \colon h_\u=h|_{{P}^d_{R(S)}(0)},  \ h\in \mathcal{V}\}.\\
\end{aligned}
\end{equation}
Then $X$ is an $S$-valued $\mathcal{V}$-holomorphic process and
for all $\u\in\mathcal{V}^*$  
$$\mathcal{A}h_\u=h_{L(\u)}|_S,$$
where $L:\mathcal{V}^*\rightarrow \SS$  is given by
\begin{align}\label{eq:Lgeneralkernel}
    L(\u)_\alpha=\sum_{|\beta|=1}(\u^{(\beta)}\ast \underline{\b}^{\beta})_\alpha+\sum_{|\beta|=2}\frac{1}{\beta!}(\u^{(\beta)}\ast (\underline{\underline{\a}}^{\beta}+\m^{\beta}))_\alpha+\sum_{|\beta|\geq 3}\frac{1}{\beta!}(\u^{(\beta)}\ast \m^{\beta})_\alpha,
\end{align}
for all $\u\in\mathcal{V}^*$, $\alpha\in\N^d_0$.
\end{theorem}
\begin{remark}\phantomsection \label{remark_continuityexp}
\begin{enumerate}
\item   The assumption in equation \eqref{eqGj} is essential to derive the  representation of the operator $L$ as in equation \eqref{eq:Lgeneralkernel} and for the application of the Vitali-Porter theorem, on which the proof of Theorem \ref{prop: holo generalkernel} is based (see Remark \ref{rem1}\ref{rem1i} for a more detailed discussion).  
    \item \label{i__remark_continuityexp}Observe that by Proposition~\ref{propclass}\ref{class:it3}, for each $h\in H(P^d_G(0))$ with $
      (h^{(\beta)})_{|\beta|\geq2}$ locally uniformly bounded on  $P^d_G(0)$ it holds
    \begin{align*}
        |h(z)|=\big|\sum_{|\beta|\geq0}\frac 1 {\beta!}h^{(\beta)} (0)z^\beta\Big|
        \leq C\sum_{|\beta|\geq0}\frac 1 {\beta!}|z|^\beta = C\exp(|z_1|+\dots+|z_d|),
    \end{align*}
    for some $C>0$ and all $z\in {P^d_{G-\delta}(0)}$, with $\delta>0$. This shows how condition~\ref{iitheorem1} in Theorem~\ref{prop: holo generalkernel} is strictly related to the defining property of $\Vcal$. As one can imagine, this implies that one can relax the integrability condition on $K_\vare$ by choosing a more restrictive growth condition on the derivatives of the elements of $\Vcal$. Vice versa, one can obtain the result for a larger set of functions $\Vcal$ by imposing a stronger integrability condition in~\ref{iitheorem1}. This trade-off appears in many results of the paper and is exploited for instance in Corollary~\ref{prop:unboundedjs} (see also Remark~\ref{remark:nofinitemoments}).
    \item\label{itii} Suppose that $S_\vare=P^d_{R(S)+\vare}(0)$. By the monotone convergence theorem it holds
    \begin{align*}
        \sum_{|\beta|\geq2}\frac 1 {\beta!}|h_{\m^\beta}(z)|
    &\leq\sum_{|\beta|\geq2}\frac 1 {\beta!}\int_{\C^d}|\xi|^\beta K_\vare(z,d\xi)\\
    &=\int_{\C^d}\exp(|\xi_1|+\dots+|\xi_d|)-1-(|\xi_1|+\dots+|\xi_d|)K_\vare(z,d\xi).
    \end{align*}
    Using that
    $0\leq\exp(y)-1-y\leq y^21_{\{y\leq 1\}}+\exp(y)1_{\{y>1\}}$ for each $y\in \R_+$, by continuity of $\int_{\C^d}(|\xi_1|^2+\dots+|\xi_d|^2)K_\vare(\cdot,d\xi)$ we can see that condition~\ref{iiitheorem1} of Theorem~\ref{prop: holo generalkernel} is automatically satisfied  if the map $\int_{\|\xi\|>1}\exp(|\xi_1|+\dots+|\xi_d|)K_\vare(\cdot,d\xi)$ is locally bounded on $S_\vare$. 

    \end{enumerate}

\end{remark}
It is important to highlight that when $d=1$, the properties of holomorphic functions in one variable allow us to considerably relax the conditions on the kernel $K$ needed to deduce the result of Theorem~\ref{prop: holo generalkernel}.  This is a consequence of the simplification of the ``Identity theorem'' for holomorphic functions on $\C$ (see Proposition~\ref{propclass}\ref{class:it identity}), which is used in the proof of Vitali-Porter theorem, and on which  the proof of Theorem~\ref{prop: holo generalkernel} 
relies. 
 We illustrate the precise statement in the following corollary, whose proof is analogous to the one of Theorem~\ref{prop: holo generalkernel}. For more details on the differences in this one-dimensional setting, we refer to Remark~\ref{rem1}\ref{rem1ii}. 

\begin{corollary}\label{coro1dimensional1}
Let $S\subseteq \Rset$ be a subset with an accumulation point in $\R$ and $X=(X_t)_{t\in [0,T]}$ be an $S$-valued jump-diffusion with characteristics $(b,a,K)$ and extended generator $\Acal$. Fix $\vare>0$ and assume that $b_j,a_{ij}\in\Ocal_\vare(S)$ and 
\begin{align}\label{eq:Vitali_1dim_moments}
    \int_{\Rset}\xi^\beta K(\cdot,d\xi) \in \mathcal{O}_\vare(S),\quad \text{for all }  \beta\geq2. 
\end{align}
Suppose furthermore that the following conditions hold true.
\begin{enumerate}
\item for all $x\in {S}$, $ \int_{|\xi|>1} \exp(|\xi|) K(x,d\xi)<\infty;$
    \item the real valued map
$\sum_{\beta\geq2}\frac{1}{\beta!}|h_{\m^\beta}(\cdot)|
$
is locally bounded on $P^1_{R(S)+\vare}(0)$.
\end{enumerate}
Fix 
\begin{equation}\label{eqG1dim}
    G>\sup_{x\in {S}}\sup_{\xi\in \textup{supp}(K(x,\cdot))}|x|+|\xi|
\end{equation}
and set $\mathcal{V}$, $\mathcal{V}^*$ as in equation \eqref{eq:Vgeneralkernel}.  Then $X$ is an $S$-valued $\mathcal{V}$-holomorphic process and
for all $\u\in\mathcal{V}^*$  
$$\mathcal{A}h_\u=h_{L(\u)}|_S,$$
where $L:\mathcal{V}^*\rightarrow \SS$  is the operator defined in equation \eqref{eq:Lgeneralkernel}, whose explicit form for $d=1$ reads
\begin{align}\label{eq:Lgeneralkernel_1dim}
    L(\u)_\alpha=(\u^{(1)}\ast \b)_\alpha+\frac{1}{2}(\u^{(2)}\ast (\a+\m^2))_\alpha+\sum_{\beta=3}^\infty\frac{1}{\beta!}(\u^{(\beta)}\ast \m^\beta)_\alpha.
\end{align}
for all $\u\in\mathcal{V}^*$, $\alpha\in\N_0$.
\end{corollary}

Next, we shall make stronger assumptions on the parametrization of the jump kernel $K$ of an $S$-valued jump diffusion $X$. Indeed, we only consider kernels with holomorphic jump size, defined as follows.

\begin{definition}\label{def:kernelholojs}
    The jump kernel $K$ is said to have holomorphic jump size if it is of the form
   \begin{align*}
    K(\cdot,A)=\lambda(\cdot)\int_E 1_{A\setminus\{0\}} ( j(\cdot,y))F(dy),
\end{align*}
where,
\begin{enumerate}
\item $F$ is a non-negative measure on a measurable space $E$;
    \item $\lambda\in \mathcal{O}(S)$ and its restriction to $S$ is positive real valued;
    \item\label{it3} $j:S\times E \rightarrow \mathbb{R}^d$ such that $ j=j_\vare|_{S\times E}$ for some $\vare>0$ and some measurable function $$j_\vare:P^d_{R(S)+\vare}(0)\times E \rightarrow \mathbb{C}^d$$ with $j_\vare(\cdot,y)\in H(P^d_{R(S)+\vare}(0),\C^d)$ for each $y \in E$.
\end{enumerate}  
\end{definition}
Before developing the analysis of these kernels further, we introduce another notation that will be employed throughout the paper.
 \begin{notation}
    Consider a jump kernel $K$ with holomorphic jump size $j$ and let $j_\vare$ be the corresponding extension as in Definition~\ref{def:kernelholojs}. For each $y\in E$ we let $$\underline{\bm{j}}(y)\in (\S^* )^d,$$ denote the coefficients determining the power series representation of $j_\vare(\cdot,y)$. This  implies that
    \begin{align*}
        h_{\underline{\bm{j}}(y)}|_S=j(\cdot,y).
    \end{align*}
    Similarly, we choose $\la\in \S^* $ such that $h_{\la}|_S=\lambda$.

    Finally, for each $|\beta|\geq 2 $ and $|\alpha|\geq 0$ recall that
    \begin{align*}
\left(\int_E\underline{\bm{j}}( y)^{\ast \beta}\ F(d y)\right)_\alpha=\int_E(\underline{\bm{j}}( y)^{\ast \beta})_\alpha\ F(d y), 
    \end{align*} 
    whenever the involved quantities are well defined.
\end{notation}
     Notice that for all $\alpha\in \N^d_0$, it holds $\underline{\bm{j}}(y)^i_\alpha$, $\la_\alpha \in \mathbb{R}$. Moreover, the map $E\ni y\mapsto \underline{\bm{j}}(y)_\alpha$ is a measurable function as for all $y\in E$, $\underline{\bm{j}}(y)_\alpha=D^\alpha j_\vare(z,y)|_{z=0}$  with $j_\vare$ measurable (see Proposition~\ref{propclass}\ref{class:it2}).

Under the assumption of holomorphic jump sizes, the hypothesis of Theorem~\ref{prop: holo generalkernel} become more explicit.  The obtained result is stated in the next corollary. Recall the notation introduced in \eqref{eqn2}.

\begin{corollary}\label{coro:fubini}
Let $X=(X_t)_{t\in[0,T]}$ be an $S$-valued jump-diffusion with characteristics $(b,a,K)$ and extended generator $\Acal$. Assume that $b_j,a_{ij}\in\Ocal(S)$ and $K$ is a kernel with holomorphic jump size. Suppose that for some $\vare>0$ the following conditions hold true.
        \begin{enumerate}
            \item \label{coro1:i}for all $|\beta|\geq2$ and $z\in P^d_{R(S)+\vare}(0)$, $\int_E|\underline{\bm{j}}( y)^{\ast \beta}|_z \ F(d y)<\infty$;
            \item \label{coro1:ii}the map $ \int_{\|j(\cdot, y)\|>1}\exp(|j_{ \vare,1}(\cdot, y)|+\dots+|j_{ \vare,d}(\cdot, y)|) \ F(d y)$ is locally bounded on $P^d_{R(S)+\vare}(0)$.
        \end{enumerate}
        Let $\mathcal{V}$ and $\mathcal{V}^*$ be defined as in equation \eqref{eq:Vgeneralkernel}. 
        Then, $K\in \Ocal(S)$ and $X$ is an $S$-valued $\mathcal{V}$-holomorphic process. Moreover, for all $\u\in\mathcal{V}^*$,  $\mathcal{A}h_\u={h_{L(\u)}}|_S$, where $L:\mathcal{V}^*\rightarrow \SS$ is the operator introduced in equation \eqref{eq:Lgeneralkernel}, and for all $|\beta|\geq2$, 
        \begin{align}\label{eq:m_beta}
    \m^\beta=\la \ast \int_E\underline{\bm{j}}( y)^{\ast \beta}\ F(d y).
        \end{align}
\end{corollary}
\begin{proof}
Fix $\vare>0$ such that $\lambda,b_j,a_{ij}\in \Ocal_\vare(S)$,  $j_\vare$ satisfies the conditions of Definition~\ref{def:kernelholojs}\ref{it3}, and the conditions~\ref{coro1:i} and~\ref{coro1:ii} of the corollary are satisfied. Let $K_\vare$ denote the transition kernel from $P^d_{R(S)+\vare}(0)$ to $\mathbb{C}^d$ given by
$$K_\vare(\cdot,A)=h_{\la}(\cdot)\int_E 1_{A\setminus\{0\}} (j_\vare(\cdot,y))F(dy).$$
    Notice that by the dominated convergence theorem and condition~\ref{coro1:i}, it holds $\int_E\underline{\bm{j}}( y)^{\ast \beta}\ F(d y)\in \mathcal{P}^d_{R(S){ +\vare}}(0)^* $ and for all $z\in P^d_{R(S)+\vare}(0)$ 
    \begin{align*}
        \int_{\mathbb{C}^d}  \xi^\beta  \ K(z,d \xi)=\lambda(z)\int_Ej_{ \vare}^\beta(z, y) \ F(d y)=\sum_{\alpha\in\N^d_0} \left(\la \ast \int_E\underline{\bm{j}}( y)^{\ast \beta}\ F(d y) \right)_\alpha \frac{z^\alpha}{\alpha!}.
    \end{align*}
   As a consequence, $\left(\la \ast \int_E\underline{\bm{j}}( y)^{\ast \beta}\ F(d y)\right)\in \mathcal{P}^d_{R(S) +\vare}(0)^* $, implying that $K\in \mathcal{O}_\vare(S)$, with $S_\vare=P^d_{R(S)+\vare}(0)$. By Remark~\ref{remark_continuityexp}\ref{itii} condition~\ref{iiitheorem1} of Theorem~\ref{prop: holo generalkernel} follows by condition~\ref{coro1:ii} in the statement. Thus, all the hypotheses of Theorem~\ref{prop: holo generalkernel} are verified and the claim follows.
\end{proof}
Even though Corollary~\ref{coro:fubini} gives explicit conditions for Theorem~\ref{prop: holo generalkernel}, direct verification of such conditions may prove to be rather complicated. Relying on powerful results from complex analysis it is however possible to obtain simpler sufficient conditions. Particularly useful for our purpose is Morera's theorem (see Proposition~\ref{propclass}\ref{class:it5} and Lemma~\ref{prop:Morerageneral} for the adaptation to our needs) providing sufficient conditions for the integral of an holomorphic map to be holomorphic. This property is very useful to show that  the extended generator of an $S$-valued jump-diffusion maps a subset of $\mathcal{O}(S)$ into $\mathcal{O}(S)$. The corresponding result is provided in Theorem~\ref{ps to ps}.

Recall that $\circ^s$ and integration of sequence-valued maps have been introduced in Section~\ref{sec 214} and Section~\ref{sec212}, respectively. In the following, we let $L^1(E,F)$ denote the space of $L^1$ maps on the measurable space $E$ with non-negative measure $F$.
The proof of the following theorem can be found in Section~\ref{sec312}.

\begin{theorem}\label{ps to ps} 
Let $X=(X_t)_{t\in [0,T]}$ be an $S$-valued jump-diffusion with characteristics $(b,a,K)$ and extended generator $\Acal$.  Assume that $b_j,a_{ij}\in\Ocal(S)$, $K$ is a kernel with holomorphic jump size, and let $F$ be the corresponding non-negative measure on the measurable space $E$. 
Fix $G=(G_1,\ldots,G_d)$ for $G_j$ such that for some $\vare>0$
\begin{equation}\label{eqnG}
    G_j>\sup_{z\in P^d_{R(S)+\vare}(0)}\sup_{ y\in \textup{supp}(F)}|z_j|+|\underline{\bm{j}}( y)^j|_z.
\end{equation}
For each $h\in H(P_G^d(0))$, $z\in P^d_{R(S)+\vare}(0)$, and $y\in E$ set 
$$
Jh(z,y):=\lambda(z)\big(h(z+j_\vare(z,y))-h(z)-\nabla h(z)^\top j_\vare(z,y)\big)
$$
and
\begin{equation}\label{eqDJ}
    \Dcal(J):=\{h\in H(P_G^d(0))\colon Jh(z,\cdot)\in L^1(E,F)\text{ for each }z\in P^d_{R(S)+\vare}(0)\}.
\end{equation}
Set 
\begin{align*}
     \mathcal{V}&\subseteq\{h\in \Dcal(J)\colon
     \text{the map $P^d_{R(S)+\vare}\ni z\to Jh(z,\cdot)\in L^1(E,F)$ is continuous}\},
\\  \mathcal{V}^*&:=\{\u\in \SS \colon h_\u=h|_{{P}^d_{R(S)}(0)},  \ h\in \mathcal{V}\}.
 \end{align*}
Then $X$ is an $S$-valued $\Vcal$-holomorphic process and for all $\u\in\mathcal{V}^*$, $$\mathcal{A}h_\u=h_{L(\u)}|_S,$$ where $L:\mathcal{V}^*\rightarrow \SS$  is given by
\begin{align}\label{eq:OpLu}
    L(\u) =\sum_{|\beta|=1}\u^{(\beta)}\ast \underline{\b}^{\beta} &+\sum_{|\beta|=2}\frac{1}{\beta!}\u^{(\beta)}\ast \underline{\underline{\a}}^{\beta}\\
    &+\la\ast \int_{E}\u\circ^s \underline{\bm{j}} ( y)-\u-\sum_{|\beta|=1}\u^{(\beta)}\ast \underline{\bm{j}}( y)^{\ast\beta} \ F(d y).\nonumber 
\end{align}
\end{theorem}
\begin{remark}
\begin{enumerate}
    \item   Similar to Theorem \ref{prop: holo generalkernel}, the assumption in equation \eqref{eqnG} is essential to derive the  representation of the operator $L$ as in Equation \eqref{eq:OpLu} (see Remark \ref{rem: G Morera}).  
    \item  Observe that the operator $L:\mathcal{V}^*\rightarrow \S^*$  defined in Equation \eqref{eq:OpLu} can be expressed in   the following more explicit form when $d=1$
\begin{align}\label{eq:L_1dim}
     L(\u)=\u^{(1)}\ast \b+\frac{1}{2}\u^{(2)}\ast\a +\la\ast \int_{E}(\u\circ^s\j( y))-\u-\u^{(1)}\ast \j( y) \ F(d y),
 \end{align} 
 for all $\u\in \Vcal^*$.
\end{enumerate}
   
\end{remark}

 The next step consists in investigating the sets $\mathcal{V}$ satisfying the condition of Theorem~\ref{ps to ps}.  We start the analysis by considering the case of general jump diffusions with possible unbounded jump sizes. 
To specify the integrability properties of the kernels of such processes, we exploit the notion of a weight function.

\begin{definition}\label{def2}
Given a polydisc $P_G^d(0)$, for some $G>0$, and a nondecreasing map
$$v:\R_+\rightarrow(0,\infty),$$
which we call \emph{weight function}, we denote by
$$H_v(P_G^d(0)):=\{h\in H(P_G^d(0)) \colon \|h\|_v<\infty\},$$ 
the corresponding weighted  spaces of holomorphic functions on $P_G^d(0)$, where 
$$\|h\|_v:= \sup_{z\in P_G^d(0)}|h(z)|v(\|z\|)^{-1}.$$
\end{definition}
The proof of the following corollary can be found in Section~\ref{proofprof:unboundedjs}.  Moreover, we refer to Remark \ref{rem: G Morera corollary}\ref{rem: G Morera corollary i} for further comments on the choice of the polyradius $G$ in this specific case.

\begin{corollary}\label{prop:unboundedjs}
Let $X=(X_t)_{t\in [0,T]}$ be an $S$-valued jump-diffusion with characteristics $(b,a,K)$ and extended generator $\Acal$.  Assume that $b_j,a_{ij}\in\Ocal(S)$, $K$ is a kernel with holomorphic jump size, and let $F$ be the corresponding non-negative measure on the measurable space $E$.  Fix  a weight function $v$, $\vare>0$, and suppose that for each $z\in P^d_{R(S)+\vare}(0)$, there exists $\delta_z>0$ such that
\begin{equation}\label{eq:theorem_unboundejs}
\begin{aligned}
    &\int_E  \ \sup_{w\in P_{{\delta_z}}^d(z)}\|j_\vare(w, y)\|\land\|j_\vare(w, y)\|^2\ F(d y)<\infty;\\
          &\int_E  \sup_{w\in P_{{\delta_z}}^d(z)} 1_{\{\|j_\vare(w, y)\|> 1\}} v(\|w+j_\vare(w, y)\|)\ F(d y)<\infty. 
\end{aligned}
\end{equation}
Let $G\in (0,\infty]^d$ satisfy \eqref{eqnG} and set
\begin{align*}
    \Vcal:=H_v(P_G^d(0)) \qquad \text{and}\qquad  \mathcal{V}^*:=\{\u\in \SS \colon h_\u=h|_{{P}^d_{R(S)}(0)},  \ h\in \mathcal{V}\}
\end{align*}
Then $X$ is an $S$-valued $\mathcal{V}$-holomorphic process and for all $\u\in\mathcal{V}^*$, $\mathcal{A}h_\u=h_{L(\u)}|_S$, where $L:\mathcal{V}^*\rightarrow \S^*$ is the operator defined in equation \eqref{eq:OpLu}. 
\end{corollary}

\begin{remark}\label{remark:nofinitemoments}
Hereafter, we list some interesting examples of weight functions. Observe furthermore that condition~\eqref{eq:theorem_unboundejs} concerns integrability conditions of the extension of the jump kernel.
\begin{enumerate}
    \item \textbf{Weight functions with sublinear growth}:  we say that a weight functions $v$ has sublinear growth 
$$v(t)\leq Ct$$
for some $C>0$ and each $t>1$. In such case, the second condition in \eqref{eq:theorem_unboundejs} is implied by the first one.
\item \textbf{Rapidly increasing weight functions}: we say that a weight functions $v$  is rapidly increasing if satisfies $\lim_{t\to\infty}t^kv(t)^{-1}=0$ for each $k\geq 0$. In this case, one has that
$$\{p|_S\text{  for some polynomial }p:\R^d\to \R\}
\subseteq\{h\in H(P^d_G(0)) \colon
     \sup_{z\in P^d_{G}(0)}|h(z)|v(\|z\|)^{-1}<\infty \},$$
     implying that the conditions of Lemma~\ref{lemma:finitepoments} are satisfied. If $v$ is sub-multiplicative (i.e. $v(t_1+t_2)\leq C' v(t_1)v(t_2)$),  the dominated convergence theorem with the bound
     $$\sup_{w\in P_{{\delta_z}}^d(z)}  \|j_\vare(w, y)\|^{|\beta|}
     \leq C \sup_{w\in P_{{\delta_z}}^d(z)}
     \Big(\|j_\vare(w, y)\|^21_{\{\|j_\vare(w, y)\|\leq 1\}}+
      1_{\{\|j_\vare(w, y)\|> 1\}} v(\|w+j_\vare(w, y)\|)\Big),$$ and an application of Lemma~\ref{prop:Morerageneral} yields that 
the coefficients $\m^\beta$ determining the power series representation of $\int_{\Rset^d} \xi^\beta  \ K_\vare(\cdot ,d\xi)$ are given by Equation \eqref{eq:m_beta}, for all  $|\beta|\geq 2$.

     \item \label{entire functions finite order }\textbf{Entire functions of finite order and type}: if $S$ is given as an unbounded subset of $\R$, an interesting choice for the weight function $v$ is 
    $$v(t):=\exp(\tau t^\eta)$$
    for some $\tau,\eta\in (0,\infty)$. This allows to work with the concept of order and type of an entire function. 
    \begin{itemize}
        \item  An entire function $h\in H(\C)$ is said to be of (finite) \textit{order} $\eta\in \R$ if 
    \begin{equation}\label{eqn8}
        \rho=\inf\{c >0 \colon |h(z)|<\exp (|z|^c) \ \text{for sufficiently large } |z|\}.
    \end{equation}
    \item  An entire function $h\in H(\C)$ of finite order $\eta\in \R$ is then said to be of (finite) \textit{type} $\tau\in \R$ if
\begin{equation}\label{eqn10}
    \tau=\inf\{a >0 \colon |h(z)|<\exp (a|z|^\eta) \ \text{for sufficiently large } |z|\}.
\end{equation}
\item An entire function is said of  \textit{exponential type $\tau$} if its order $\eta<1$, or $\eta=1$ and its type $\tau$ is finite.
    \end{itemize}
    More generally, we say that an entire function is of  \textit{exponential type} if its order $\eta<1$, or $\eta=1$ and its type is finite.

   In particular, setting
    $\Vcal= H_v(\C),$ 
    we obtain that $\Vcal$ contains every entire function of order strictly smaller than $\eta$ and every entire function of order $\eta$ and type strictly smaller than $\tau$. The advantage of working with the class of entire functions of finite order is that conditions for characterizing the sequence determining their power series representation have been studied extensively (see e.g.~Theorem 2 and Theorem 3 in \cite{LE:98} and Proposition~\ref{propE3}).

    \item \label{bounded jump size} \textbf{Locally uniformly bounded jump size}: assume that for
every $z\in P^d_{R(S)+\vare}(0)$ there exists $\delta_z>0$ such that  
\begin{equation}\label{eqn_bdd}
         \sup_{w\in P_{{\delta_z}}^d(z), \  y\in \textup{supp}(F)}\|j(w, y)\|<\infty.
\end{equation}
     In this case the second condition in \eqref{eq:theorem_unboundejs} is implied by the first one for each weight function $v$ and the result of Corollary~\ref{prop:unboundedjs} holds for $\Vcal=H(P_G^d(0))$. This in particular implies that $X$ is an $S$-valued $H(P_G^d(0))$-holomorphic process. This is of particular interest when $S$ is bounded and condition \eqref{eqn_bdd} is automatically satisfied.
     Also in this case by the dominated convergence theorem and Lemma~\ref{prop:Morerageneral} we can conclude that $\m^\beta$ satisfies \eqref{eq:m_beta}, for all  $|\beta|\geq 2$.

\end{enumerate}
It is also interesting to notice that including extra conditions on the kernel $K$ as in Corollary~\ref{prop:unboundedjs} permits to obtain the holomorphic property for a set of functions $\Vcal$ considerably larger than the one defined in Theorem~\ref{prop: holo generalkernel} (see in particular Equation \eqref{eq:Vgeneralkernel}). Consider for instance the case of a locally uniformly bounded jump size~\ref{bounded jump size}, and notice that here the set $\Vcal$ is strictly larger than the corresponding set defined in Equation \eqref{eq:Vgeneralkernel}. The same holds true in the above case~\ref{entire functions finite order }. Indeed, assume that $\eta,{\tau}>1$ and consider the entire function 
\begin{align*}
    f(z):=\sum_{n=0}^\infty \frac{e^n n!}{n^n}\frac{z^n}{n!}, \quad z\in \C.
\end{align*}
By Theorem 2 and Theorem 3 in \cite{LE:98} $f$ is of order and type $1$, implying that $f\in \mathcal{V}=H_v(\mathbb{C})$. However, its derivatives in 0 are not uniformly bounded.
   \end{remark}

\subsection{Examples of holomorphic jump-diffusions}
This section is dedicated to provide an overview of different instances of holomorphic processes. For simplicity we consider the case $d=1$.
We already saw that classical polynomial processes are $\Vcal$-holomorphic if $\Vcal$ consists of polynomials. We show now that this is the case also for larger sets $\Vcal$. 

\begin{example}
Let $B=(B_t)_{t\in[0,T]}$ be a one-dimensional Brownian motion and $N(dy,dt)$ a  Poisson random measure with compensator $F(dy)\times dt$ on $E\times  [0,T]$, for some measurable space $E$. Consider the following stochastic differential equation:
\begin{align}\label{eq: SDE}
dX_t=b(X_t)dt+\sigma(X_t)dB_t+\int_E\delta(X_{t^-},y) (N(dy,dt)-F(dy)dt), \qquad X_0=x_0\in \R.
\end{align}
Suppose that 
\begin{align*}
    b(x):=\b_0+\b_1x, \qquad \sigma(x):=\sg_0+x\sg_1, \quad \text{and}\quad j(x,y):=\j_0(y)+x\j_1(y),
\end{align*}
for some functions $\j_i:E\rightarrow\Rset$, such that $\int_E|\j_i(y)|^kF(dy)<\infty$ for all $k\geq2$, $i=0,1$. Then, there exists a unique strong $\R$-valued solution $X=(X_t)_{t\in[0,T]}$ of the equation \eqref{eq: SDE}, which is in fact a polynomial jump-diffusion (see Example 2.6 in \cite{FL:20}). 
Furthermore, notice that setting $$K(x,A):=\int_E1_{A\setminus \{0\}}(j(x,y))F(dy),$$ it holds that $K$ is a transition kernel with holomorphic jump size in the sense of Definition~\ref{def:kernelholojs}. Moreover, in some cases, $X$ is also a holomorphic process, for other $\mathcal{V}$ different from polynomials.
\begin{enumerate}
    \item If $\sup_{y\in E}|\j_i(y)|<\infty$, for $i=0,1$,  the conditions of Remark~\ref{remark:nofinitemoments}\ref{bounded jump size} are satisfied. Let $\Vcal$ and $\Vcal^*$ be the sets defined in the same remark. Then $X$ is a $\Vcal$-holomorphic process and for all $\u\in \Vcal^*$, $\mathcal{A}h_\u=h_{L(\u)}|_\Rset$, where $L$ is given in equation \eqref{eq:L_1dim} with 
    \begin{align*}
        \b:=(\b_0,\b_1,0,\dots), \quad  \a:=(\sg_0^2,2\sg_0\sg_1,2\sg_1^2,0,\dots)\quad \text{and} \quad \la:=(1,0,0,\dots).
    \end{align*}
    \item Alternatively, if for every $z\in \C$ it holds 
    \begin{align*}
        \int_E  \sup_{w\in P_{{\delta_z}}^1(z)} 1_{\{|\j_0(y)+x\j_1(y)|> 1\}} v(|\j_0(y)+w(1+\j_1(y))|) F(dy)<\infty,
    \end{align*}
    for some $\delta_z>0$ and some weight function $v$, then the hypothesis of Corollary~\ref{prop:unboundedjs} are satisfied. Letting $\Vcal$ and $\Vcal^*$ be the sets defined in the same corollary we get that the process $X$ is $\Vcal$-holomorphic. 
\end{enumerate}
Relying on existence results given in Theorem III.2.32 in \cite{JS:87}, one could also go beyond linear coefficients and consider solutions of stochastic differential equations with entire coefficients.

\begin{example}
    Let $B=(B_t)_{t\in[0,T]}$ be a one-dimensional Brownian motion. Consider the following stochastic differential equation:
\begin{align}\label{eq: SDE2}
dX_t=b(X_t)dt+\sigma(X_t)dB_t, \qquad X_0=x_0\in \R.
\end{align}
Assume that $b$ and $\sigma$ are bounded entire functions with bounded first derivatives. Then, equation \eqref{eq: SDE2} admits a unique solution $X=(X)_{t\in [0,T]}$ which is furthermore a holomorphic $\mathcal{V}$-holomorphic process, for $\mathcal{V}$ being the class of entire functions.

 Notice that more generally solutions of \textit{neural SDEs} as specified e.g., in \cite{GSSSZ:22}, \cite{CKT:20} are holomorphic processes if the considered activation functions are holomorphic ones.  
\end{example}

\end{example}

Alternatively, examples of holomorphic processes can be provided by studying martingale problems. Indeed, recall that
we defined holomorphic processes as solution of a martingale problem, i.e.~as $X=(X_t)_{t\in [0,T]}$ such that the process $N^f$ introduced in \eqref{eq:localmart} is a local martingale for a set of test functions $f$. Given some coefficients $a,b,\lambda$ and $j$, sufficient (and necessary) conditions for the existence of a solution to the corresponding martingale problem can be obtained by verifying the hypothesis of Theorem 4.5.4 in \cite{EK} (modulo explosion). If the state space is $\R$, this reduces to check that the diffusion coefficient $a$ is nonnegative. To guarantee that $X$ does not explode, one can then resource to Theorem 4.3.8 in \cite{EK}, which translates in checking that $\Acal1=0$, when the state space is compact.

We here present some coefficients $a,b,\lambda$ and $j$,  for which the solution of the corresponding martingale problem is a holomorphic problem.

\begin{example}
Fix $S\subseteq \Rset$, $G>0$ as in \eqref{eqG1dim}, set $b,a, \lambda\in \Ocal(S)$, and let $F$ be a non-negative measure on $(\Rset,\mathcal{B}(\Rset))$. We now analyze the form of the jump size $j$ in the compensator of the jumps $K(x,A):=\int_\R1_{A\setminus \{0\}}(j(x,y))F(dy)$ such that a jump-diffusion $X=(X_t)_{t\in[0,T]}$, given as solution to the martingale problem for the triplet $(b,a,K)$, is a holomorphic process.
Consider the following three specifications:
\begin{enumerate}
    \item \label{item j i}Let $j:P^1_G(0)\times \Rset\rightarrow\Rset $ be such that for all $(z,y)\in P^1_G(0)\times \R$, $j(z,y):=\sum_{\alpha\in \N_0}\bm{j}_\alpha(y) \frac{z^\alpha}{\alpha!}$, for $\j_\alpha:\Rset\rightarrow\Rset $ measurable and bounded functions, for all $\alpha\in\N_0$.  Set
    \begin{align*}
        s(y):=\sup_{\alpha\in \N_0} |\j_\alpha(y)|, \quad y\in\Rset.
    \end{align*}
  Assume that $\sup_{y\in \Rset} s(y)<\infty$ and $\int_\Rset s(y)^2 \ F(dy)<\infty$.
Fix $z\in P^1_G(0)$, $\delta_z>0$ and note that 
    \begin{align*}
        &\sup_{y\in \Rset}\sup_{w\in P^1_{\delta_z}(z)}|j(w,y)|\leq \sup_{\alpha\in \N_0}\sup_{y\in\Rset}|\j_\alpha(y)|\exp(r+|z|)<\infty,\\
        & \int_\Rset \sup_{w\in P^1_{\delta_z}(z)}|j(w,y)|^2 \ F(dy)\leq  \exp(2(r+|z|))\int_\Rset s(y)^2 \ F(dy)<\infty.
    \end{align*}
    In particular, a polynomial jump size might be considered: $j(z,y):=\sum_{\alpha=0}^N \bm{j}_\alpha(y) \frac{z^\alpha}{\alpha!}$, where for all $\alpha\leq N$, $\j_\alpha:\Rset\rightarrow\Rset $ with $\sup_{y\in\Rset}|\j(y)_\alpha|<\infty $ and $ \int_\Rset |\j_\alpha(y)|^2 \ F(dy)<\infty.$ 
    \item \label{item j ii} Let $f\in \Ocal(S)$ and assume that $F$ is a measure of bounded support $E:=\textup{supp}(F)$ such that $\int_E |y|^2 F(dy)<\infty$. Set for all $(z,y)\in P^1_G(0)\times E$, $$j(z,y):=f(z+y)=\sum_{\alpha\in\N_0} f^{(\alpha)}(y)\frac{z^\alpha}{\alpha!}.$$ Then for all $z\in P^1_G(0)$ and some $\delta_z>0$,
\begin{align*}
    &\sup_{y\in E }\sup_{w\in P^1_{\delta_z}(z)}|j(w,y)|=\sup_{y\in E }\sup_{w\in P^1_{\delta_z}(z)}|f(w+y)|<\infty,\\
    & \int_E \sup_{w\in P^1_{\delta_z}(z)}|j(w,y)|^2 \ F(dy)\leq  \int_E\sup_{w\in  P^1_{\delta_z}(z)}|f(w+y)|^2F(dy)\leq  \int_E C|y|^2 F(dy)<\infty,
\end{align*}
for some $C>0$.
\item \label{item j iii}Let $j:P^1_G(0)\times \Rset\rightarrow\Rset $ be such that for all $(z,y)\in P^1_G(0)\times \R$, $j(z,y):=\sum_{\alpha\in \N_0} \bm{j}_\alpha(y) \frac{z^\alpha}{\alpha!}$, where for all $\alpha$, $\j_\alpha:\Rset\rightarrow\Rset $ with $|\j(y)_\alpha|<\ell(y)$ 
for some $\ell:\R\to\R_+$ satisfying $\ell(y)\geq |y|$ and
$\int_\Rset\exp(C\ell(y))   F(dy)<\infty$ for each $C>0$. Notice that for all $z\in \C$ and some $\delta_z>0$,
        \begin{align*}
           & \int_\Rset \sup_{w \in P^1_{\delta_z}(z)}1_{\{|j(w,y)|\leq1\}} |j(w,y)|^2\ F(dy)\leq C\int_\Rset \ell(y)^2F(dy)<\infty ,\\
           &\int_\Rset \sup_{w \in P^1_{\delta_z}(z)}1_{\{|j(w,y)|>1\}}\exp(|j(w,y)|) \ F(dy)\leq \int_\Rset \exp(C\ell(y))F(dy)<\infty.
           \end{align*}
Note that the integrability condition on $\ell$ is guaranteed if $\int_\Rset\exp(\ell(y)^2)   F(dy)<\infty$.

\end{enumerate}
If the jump size is specified as in~\ref{item j i} or~\ref{item j ii}, then the conditions of Remark~\ref{remark:nofinitemoments}\ref{bounded jump size} are satisfied and $X$ is a $\Vcal$-holomorphic jump-diffusion, for $\Vcal$ being the set defined in the same remark.  If instead $j$ is given as in~\ref{item j iii}, then the conditions of Corollary~\ref{prop:unboundedjs} hold true and $X$ is $\Vcal$-holomorphic, for $\Vcal$ denoting the set of entire functions of exponential type (see Remark~\ref{remark:nofinitemoments}\ref{entire functions finite order } for the definition of entire function of exponential type).
\end{example}

\subsection{The holomorphic formula}\label{sec: the holomorphic formula}

Fix $\Vcal\subseteq\mathcal{O}(S)$ and let $X=(X_t)_{t\in[0,T]}$ be an $S$-valued $\Vcal$-holomorphic process. Denote by $\mathcal{A}$ its extended generator and recall that by the properties of the functions in $\mathcal{O}(S)$, there exists a  linear map $ L: \mathcal{V}^*\rightarrow \mathcal{O}(S),$
such that for all $\u\in\mathcal{V}^*$,
\begin{align}\label{eq: Ah=  dual point of view}
   \mathcal{A}h_\u=h_{L(\u)}|_S,
\end{align}
where $\Vcal^*$ denotes the set of coefficients determining some power series representation on $P^d_{R(S)}(0)$ of the functions in $\Vcal$. In this section, we rely on the \textit{duality theory}, as presented in Chapter 4 of \cite{EK} (see also Section 2.1 in \cite{CST:23}), to compute expected values of $h_\u(X_t)$ for $t\in [0,T]$ and $h_\u\in \mathcal{V}$. In particular, we show that condition \eqref{eq: Ah=  dual point of view} is in fact the key property that allows to recognize  
a sequence valued-solution of the linear ODE $$\partial_t\c(t)=L(\c(t)), \quad \c(0)=\u$$
as (one possible choice of) a dual process\footnote{We refer to page 188 in \cite{EK} for the precise notion of a dual process.} of $X$.
This then implies that computing expected values of holomorphic functions of holomorphic processes reduces to solving an (infinite-dimensional) system of  linear ODEs.

Recall from Section~\ref{sec212} that integrals of sequences are defined componentwise.
\begin{theorem}\label{th: moment formula holomorphic}
    Set $X_0=x_0\in S$, fix $\u\in \Vcal^*$, and suppose that the following conditions hold true.
\begin{enumerate}\phantomsection
    \item \label{moment formula i}The sequence-valued linear ODE
    \begin{align}\label{eq:ODE sequence space holomorphic}
        \c(t)=\u+\int_0^tL(\c(s))ds, \qquad t\in [0,T],
    \end{align}
    admits a $\Vcal^*$-valued solution $(\c(t))_{t\in[0,T]}$  such that
   \begin{align}\label{eq: ODE pointwise}
       h_{\c(t)}(x)=h_\u(x)+\int_0^th_{L(\c(s))}(x)ds,
   \end{align}
  for all $x\in S$, $t\in[0,T]$.

    \item \label{moment formula ii}The process $(N^{h_{\c(s)}}_t)_{t\in[0,T]}$ given in equation \eqref{eq:localmart} defines a true martingale for each $s\in [0,T].$
    \item \label{moment formula iii}$\int_0^T\int_0^T\E[|\mathcal{A}h_{\c(s)}(X_t)|]dsdt<\infty$.
\end{enumerate}
Then for each $\u\in \Vcal*$ it holds that
\begin{align}\label{eq: holomorphic formula}
    \E[h_\u(X_T)]=h_{\c(T)}(x_0).
\end{align}
\end{theorem}

\begin{proof}
    We verify the conditions of Lemma A.1 in \cite{CST:23}. Set
    \begin{align*}
        Y^1(s,t):=h_{\c(s)}(X_t),\qquad Y^2(s,t):=h_{L(\c(s))}(X_t).
    \end{align*}
    By continuity of $Y^1(\cdot,t)(\omega)$ and measurability (on $[0,T]\times \Omega$) of $Y^1(s,\cdot)$ the two maps $Y^1,Y^2:[0,T]\times [0,T]\times \Omega\longrightarrow\C$ are measurable functions.
    Next, observe that by the assumption~\ref{moment formula ii} it holds that for each $s,t\in[0,T]$, the process $(N^{h_{\c(s)}}_t)_{t\in[0,T]}$, whose explicit form reads 
       \begin{align}\label{eq:true martingale}
           h_{\c(s)}(X_t)-h_{\c(s)}(x_0)-\int_0^t \A h_{\c(s)}(X_u)du, 
       \end{align}
       is a true martingale. Moreover,  by  assumption~\ref{moment formula i} 
       it holds that for every    $s,t\in [0,T]$,
    \begin{align}\label{eq:solution ode}
      h_{\c(s)}(X_t)-h_{\c(0)}(X_t)-\int_0^s \A h_{\c(u)}(X_t)du=0.
    \end{align}
   Finally, taking expectation in both equations \eqref{eq:true martingale}, \eqref{eq:solution ode} and by assumption~\ref{moment formula iii},  all the hypotheses of the above mentioned lemma are satisfied, and thus the claim follows.
\end{proof}

\begin{remark}\phantomsection\label{rem2}
\begin{enumerate}
\item\label{rem2i} An inspection of the proof shows that \eqref{eq:ODE sequence space holomorphic} is not necessary. Specifically, condition~\ref{moment formula i} can be replaced by the assumption of the existence of a $\Vcal^*$-valued map $(\c(t))_{t\in[0,T]}$  such that $$
       h_{\c(t)}(x)=h_\u(x)+\int_0^th_{L(\c(s))}$$
  for all $x\in S$, $t\in[0,T]$.

    On the other hand if  $\int_0^T|L(\c(s))|_xds<\infty$ for all $x\in S$, $t\in [0,T]$ then condition \eqref{eq:ODE sequence space holomorphic} implies \eqref{eq: ODE pointwise}. 
    \item \label{rem2_2}Recall that there might be more than one map $L$ (corresponding to different representations of the functions $\mathcal{A}h$ on $S$ for $h\in \mathcal{V}$) for which \eqref{eq: Ah=  dual point of view} holds true. However, this variability is not an issue for the objectives of this paper. The duality approach requires only the existence of a dual process, which in this context then refers to the existence of a solution to some sequence-valued ODE that satisfies \eqref{eq: Ah= dual point of view}.

      \item Notice that the claim of Theorem~\ref{th: moment formula holomorphic} for the $\mathcal{V}$-holomorphic process $X$ coincides with the assertion of Theorem 3.21 in \cite{CST:23} for the (infinite-dimensional) $\mathcal{V}^*$-polynomial process $\mathbb{X}$ introduced in Equation \eqref{eq: boldX infinite}. 
    
    \item Let $X$ be a $S$-valued polynomial jump-diffusion and recall from Remark~\ref{rem: poly processes are holomorphic} that $X$ is a $S$-valued $\mathcal{V}$-holomorphic process, for  $\Vcal=   \{p|_S\text{  for some polynomial }p:\R^d\to \R\}$.   In this case, the holomorphic formula \eqref{eq: holomorphic formula} coincides with the so-called moment formula in Theorem 1 in \cite{FL:20} and Theorem 2.7 in \cite{CKT:12}. Notice in particular that the sequence-valued linear ODE in \eqref{eq:ODE sequence space holomorphic} reduces to a finite dimensional system of linear ODEs, and~\ref{moment formula i},\ref{moment formula ii},\ref{moment formula iii} are always satisfied (see also Section~4.2 in \cite{CS:21}).
    \item\label{rem2ii} Observe that in principle there could exists two $S$-valued $\Vcal$-holomorphic processes $X$ and $Y$ sharing the same generator $\Acal$. If condition~\ref{moment formula i} of Theorem~\ref{th: moment formula holomorphic} is satisfied for some  $(\c(t))_{t\in[0,T]}$ and  conditions~\ref{moment formula ii} and~\ref{moment formula iii} are satisfied for both $X$ and $Y$ we can conclude that
    $$\E[h_\u(X_T)]=\E[h_\u(Y_T)],$$
    for each $\u\in \Vcal^*$. 
    If instead $\Vcal$ is rich enough, e.g. it contains some exponential functions,  uniqueness hold.
    \end{enumerate}
\end{remark}

    \subsection{Sufficient conditions for the application of the holomorphic formula}
This section is dedicated to the study of sufficient conditions for the application of Theorem~\ref{th: moment formula holomorphic}. To start, we provide sufficient conditions for the existence of a $\Vcal^*$-valued map $(\c(t))_{t\in[0,T]}$ such that condition \eqref{eq: ODE pointwise} holds. Recall that by Remark~\ref{rem2}~\ref{rem2i} this can substitute assumption~\ref{moment formula i} in Theorem~\ref{th: moment formula holomorphic}. Suppose that  $X$ is a time-homogeneous Markov process with semigroup $(P_t)_{t\in [0,T]}$, that is
$$P_tf(x)=\E[f(X_t)|X_0=x],\qquad x\in S,$$
for each measurable function $f:S\to \R$ such that $P_t|f|(x)<\infty$ for each $x\in S$. 
The next lemma is an adaptation of Lemma 2.6 in \cite{CKT:12}.
Observe that the action of the corresponding extended generator (in the sense of Definition 2.3 in \cite{CKT:12}) on functions $f\in \Vcal$ corresponds here to $\Acal f$.

\begin{lemma}\label{lemma:AP=PA}
 Let $f\in \mathcal{V}$ and $\u\in \Vcal^*$ be such that $f=h_\u|_S$. Suppose that:
 \begin{enumerate}
     \item  the process $N^f$ introduced in \eqref{eq:localmart} is a true martingale;
     \item $\int_0^T P_s|\Acal f|(x)ds<\infty$ 
for all $x\in S$;
     \item   $P_tf\in\Vcal$ for all $t\in [0,T]$;
     \item $S\ni x\mapsto P_t\mathcal{A}f(x)$ is continuous.
      \end{enumerate}  
 Then for all $x\in S$, $P_t\Acal f(x)=\Acal P_tf(x)$ and each 
$\Vcal^*$-valued map $(\c(t))_{t\in[0,T]}$ such that
$$    P_tf(x)=h_{\c(t)}(x),$$
satisfies \eqref{eq: ODE pointwise} with initial condition $\u$.
\end{lemma}
\begin{proof}
  Since by assumption $N^f$ is a true martingale and for each $x\in S$, $\int_0^T P_s|\Acal f|(x)ds<\infty$, by Fubini theorem, for each $t\in [0,T]$, 
    \begin{equation}\label{eqn121}
        P_tf(x)-P_0f(x)-\int_0^t P_s\Acal f(x)ds=\E[N^f_t-N^f_0|X_0=x]=0.
    \end{equation}
Fix $s\in [0,T]$ and notice that by definition of $\mathcal{V}$-holomorphic process, since $P_sf\in\Vcal$, the process $N^{P_sf}$ given by
    \begin{align*}
        N^{P_sf}_t=P_sf(X_t)-P_sf(X_0)-\int_0^t\mathcal{A}P_sf(X_r)dr, \qquad t\in[0,T],
    \end{align*}
    is a local martingale. 
    Next, we show that the process
    \begin{equation}\label{eqnnewmart1}
        P_sf(X_t)-P_sf(X_0)-\int_0^tP_s\mathcal{A}f(X_r)dr, \qquad t\in[0,T],
    \end{equation}
is a true martingale. Since for each $t\in [0,T]$ and each $x\in S$ $\int_0^t P_s|\Acal f|(x)ds<\infty$, and $N^f$ is a true martingale,  $P_t|f|(x)<\infty$ and thus
    $\E[|P_sf(X_t)|]\leq P_{s+t}|f|(x)<\infty$ and $$\int_0^t\E[|P_s\mathcal{A}f(X_r)|]dr\leq \int_0^t P_{s+r}|\Acal f|(x)dr<\infty,$$
showing that \eqref{eqnnewmart1} is integrable. By taking conditional expectation we thus obtain
\begin{align*}
&\E[P_sf(X_t)-P_sf(X_u)-\int_u^tP_s\mathcal{A}f(X_r)dr|\mathcal F_u]\\
&\qquad=\E[P_sf(X_{t-u})-P_sf(X_0)-\int_u^tP_s\mathcal{A}f(X_{r-u})dr|X_0=x]|_{x=X_u}\\
&\qquad=P_{s+t-u}f(X_u)-P_sf(X_u)-\int_u^tP_{s+r-u}\mathcal{A}f(X_{u})dr\\
&\qquad=P_{s+t-u}f(X_u)-P_sf(X_u)-\int_s^{s+t-u}P_{r}\mathcal{A}f(X_{u})dr=0,
\end{align*}
and the claim follows. By uniqueness of the decomposition of a special semimartingale, it follows that
$\int_0^t\mathcal{A}P_sf(X_r)dr=\int_0^tP_s\mathcal{A}f(X_r)dr$ and that $N^{P_sf}$ is given by \eqref{eqnnewmart1} and is in fact a true martingale. This implies condition~\ref{moment formula ii} of Theorem~\ref{th: moment formula holomorphic}. Finally, since $P_sf\in \mathcal{V}$ implies $\mathcal{A}P_sf\in \mathcal{O}(S)$, and by assumption $P_s\mathcal{A}f$ is continuous, we can conclude that for each $x\in S$,
$$\mathcal{A}P_sf(x)=P_s\mathcal{A}f(x).$$
Then, \eqref{eq: ODE pointwise} follows by \eqref{eqn121}.
\end{proof}

\begin{remark}\phantomsection\label{rem: lemma AP=PA} 
 An inspection of the proof of Lemma~\ref{lemma:AP=PA} shows that its assumptions imply condition~\ref{moment formula ii} of Theorem~\ref{th: moment formula holomorphic}.

 Furthermore, notice that  since 
$\int_0^T\int_0^TP_t|\Acal P_sf|(x)dsdt\leq\int_0^T\int_0^TP_{t+s}|\Acal f|(x)dsdt,$ if we additionally assume that 
$$\int_0^T\int_0^TP_{t+s}|\Acal h_\u|(x)dsdt<\infty$$
then condition~\ref{moment formula iii} of the same theorem holds. This can be interesting in view of Remark~\ref{rem2}.
\end{remark}

Next, we specify some conditions for the assumptions~\ref{moment formula ii} and~\ref{moment formula iii} in Theorem~\ref{th: moment formula holomorphic} to be satisfied. 
Additionally to the assumption made at the beginning of the section, we suppose that for each $\u\in\Vcal^*$, there exists a $\Vcal^*$-valued solution of the linear ODE \eqref{eq:ODE sequence space holomorphic} which satisfies \eqref{eq: ODE pointwise}, with initial value $\u$, that we denote by $(\c(t))_{t\in[0,T]}$.

The first result pertains to holomorphic processes whose extended generator $\mathcal{A}$ acts between (the restriction on $S$ of) weighted spaces of holomorphic functions (see Definition~\ref{def2}).

 \begin{lemma}\label{lemma: lemma moment formula weights}
Assume that $\Vcal\subseteq H_v(\C^d)$ and that $\mathcal{A}(\mathcal{V})\subseteq H_w(\C^d) $, for some weight functions $v$ and $w$. Suppose furthermore that one of the following conditions hold true for each $s\in [0,T]$.
   \begin{enumerate}
       \item $
       \int_0^T\|h_{L(\c(s))}\|_wds<\infty$
       and  $ \E[\sup_{t\leq T}v(\|X_t\|)],\E[\sup_{t\leq T}w(\|X_t\|)]<\infty.$
       \item \label{it ii weight} $\|h_{\c(s)}^p\|_v,\|h_{L(\c(s))}^p\|_w,\int_0^T\|h_{L(\c(s))}\|_wds<\infty$ for  $p>1$
       and  $ \E[v(\|X_T\|)],\int_0^T\E[w(\|X_t\|)]dt<\infty$.
   \end{enumerate}
Then\footnote{An inspection on the proof shows that the supremum over $\C^d$ defining $\|\cdot\|_v$ can be replaced by a supremum over $\R^d$.} conditions~\ref{moment formula ii} and~\ref{moment formula iii} of Theorem~\ref{th: moment formula holomorphic} are satisfied.
  
\end{lemma}
\begin{proof} Since for each $s\in [0,T]$, $\c(s)\in \mathcal{V}^*$, and $X$ is an $S$-valued $\Vcal$-holomorphic process, it holds that the process $(N^{h_{\c(s)}}_t)_{t\in[0,T]}$ given in equation \eqref{eq:localmart} defines a local martingale for every $s\in [0,T]$. Consider the first set of assumptions and note that since 
\begin{align*}
    \E[\sup_{t\leq T}|N_t^{h_{\c(s)}}|]&\leq 2 \E[\sup_{t\leq T}|h_{\c(s)}(X_t)|]+T\E[\sup_{t\leq T}|\mathcal{A}h_{\c(s)}(X_t)|]\\
    &\leq \|h_{\c(s)}\|_v\E[\sup_{t\leq T}v(\|X_t\|)]+T \|h_{L(\c(s))}\|_w\E[\sup_{t\leq T}w(\|X_t\|)]<\infty,
\end{align*}
we can conclude that $(N^{h_{\c(s)}}_t)_{t\in[0,T]}$ is a true martingale and thus condition~\ref{moment formula ii}  of Theorem~\ref{th: moment formula holomorphic}.
Similarly, since 
        \begin{align*}
\int_0^T\int_0^T\E[|\mathcal{A}h_{\c(s)}(X_t)|]dsdt \leq T\int_0^T\|h_{L(\c(s))}\|_wds\ \E[\sup_{t\leq T}w(\|X_t\|)]<\infty,
        \end{align*}
we can conclude that condition~\ref{moment formula iii} of the same theorem holds too.

With the second set of assumptions by Doob's inequality we have that
\begin{align*}
    \E[\sup_{t\leq T}|N_t^{h_{\c(s)}}|^p]
    &\leq C
    \E[|N_T^{h_{\c(s)}}|^p]
    \leq 2C' \E[|h_{\c(s)}(X_T)|^p]+C'\int_0^T\E[|\mathcal{A}h_{\c(s)}(X_t)|^p]dt\\
    &\leq 2C'\|h_{\c(s)}^p\|_v\E[v(\|X_T\|)]+ C'\|h_{L(\c(s))}^p\|_w\int_0^T\E[w(\|X_t\|)]dt<\infty,
\end{align*}
proving condition~\ref{moment formula ii} and
\begin{align*}
\int_0^T\int_0^T\E[|\mathcal{A}h_{\c(s)}(X_t)|]dsdt \leq \int_0^T\int_0^T\|h_{L(\c(s))}\|_wds\ \E[w(\|X_t\|)]dt<\infty,
        \end{align*}
proving condition~\ref{moment formula iii}   and     concluding the proof.
\end{proof}

In the next result we use a Gronwall-type argument to deduce some integrability conditions needed for proving~\ref{moment formula ii} and~\ref{moment formula iii} of Theorem~\ref{th: moment formula holomorphic}.  A similar argument is used in the classical case (see Theorem 2.10 in \cite{CKT:12}) to prove finiteness of moments of polynomial jump-diffusions and in the infinite dimensional setting (see Definition 3.18 and Lemma 3.19 in \cite{CS:21}) for similar purposes.
\begin{definition}\label{def:g-cyclical}
   Let $\mathcal{A}$ be the extended generator of an $S$-valued jump-diffusion, and fix $g:\Rset^d\rightarrow \Rset_+$, with $ g\in \mathcal{D}(\mathcal{A})$. We say that $\mathcal{A}$ is $g$-\emph{cyclical} if $ |\mathcal{A}g(x)|g(x)^{-1}<\infty$ for all $x\in S$. 
\end{definition}
\begin{remark}
\begin{enumerate}
 \item Notice that from the polynomial property of the extended generator $\mathcal{A}$ of a polynomial jump-diffusion, one can deduce that $\mathcal{A}$ is g-cyclical for $g(x):=1+\|x\|^{2k}$, $x\in \Rset$, for every $k\in \N$. Observe however that, since in the setting of holomorphic processes we deal with the larger class of convergent power series (and not only with polynomials of finite degree), the same cyclical argument does not generally follow directly. 
    \item Following the same reasoning in the proof of Lemma 3.19 in \cite{CS:21} (see also Theorem 2.10  in \cite{CKT:12}), we get that if the extended generator of an $S$-valued jump-diffusion $X=(X_t)_{t\in [0,T]}$ is $g$-cyclical, then 
    \begin{equation}\label{eqn9}
    \E[g(X_t)]\leq g(x_0)\exp(Ct),
    \end{equation}
    for all $t\in [0,T]$. Thus, the map $g$ can play the role of a weight function to deduce sufficient conditions for~\ref{moment formula ii} and~\ref{moment formula iii} in Theorem~\ref{th: moment formula holomorphic}. To simplify the notation, for a map $f$ we set 
    \begin{align}\label{eq: f cyclical}
        \|f\|_g:=\sup_{x\in S}|f(x)|g(x)^{-1}.
    \end{align}
 Note that contrary to the setting in Definition~\ref{def2}, in \eqref{eq: f cyclical}, we consider the supremum only over the set $S$.
\end{enumerate}

\end{remark}
The proof of the next lemma follows the proof of Lemma~\ref{lemma: lemma moment formula weights} combined with $\eqref{eqn9}$.
\begin{lemma}
 Assume that the operator $\mathcal{A}$ is $g$-cyclical, for some function $g:\Rset^d\rightarrow \Rset_+$.
Fix $\u\in \mathcal{V}^*$ and let $(\c(t))_{t\in [0,T]}$ satisfy condition~\ref{moment formula i} of Theorem~\ref{th: moment formula holomorphic} with $\c(0)=\u$.
If for  each $s\in [0,T]$ it holds
$\|h_{\c(s)}^p\|_g,\|h_{L(\c(s))}^p\|_g$, and $\int_0^T\|h_{L(\c(s))}\|_gds<\infty$
 for some $p>1$,
 then conditions~\ref{moment formula ii} and~\ref{moment formula iii} of Theorem~\ref{th: moment formula holomorphic} are satisfied.
\end{lemma}

To conclude, we discuss the case of holomorphic processes with values in a bounded state space $S$. Recall that $R(S)$ denotes the polyradius of the smallest closed polydisc which includes $S$ (see Section~\ref{sec: convergent power series on a given set}) and that for $\u\in\Vcal^*$, $z\in \C^d$, the notation $|\u|_z$ has been introduced in Equation \eqref{eqn2}.

\begin{lemma}\label{lemma: i ii moment formula bounded}
    Assume that $S$ is a bounded set and that 
    \begin{align}\label{eq: integral L s compact}
        \int_0^T |L(\c(s))|_{R(S)}ds<\infty.
    \end{align} Then condition \eqref{eq:ODE sequence space holomorphic} implies \eqref{eq: ODE pointwise} and 
    conditions~\ref{moment formula ii} and~\ref{moment formula iii} of Theorem~\ref{th: moment formula holomorphic} are satisfied.
\end{lemma}
\begin{proof}
By the definition of a holomorphic process, if  $\c(s)\in\mathcal{V}^*$ then the process $(N^{h_{\c(s)}}_t)_{t\in[0,T]}$ is a local martingale, for each $s\in [0,T]$. Since bounded local martingales are martingales, condition~\ref{moment formula ii} is always satisfied. Next, notice that $$\int_0^T\int_0^T\E[|h_{L(\c(s))}(X_t)|]dsdt\leq T\int_0^T |L(\c(s))|_{R(S)}ds<\infty,$$
proving that condition ~\ref{moment formula iii} is satisfied, too.
\end{proof}

\subsection{Applications}\label{sec: examples}
In this section, we discuss some applications of the preceding theory. As a first example, we consider continuous-time Markov chains with a finite-state space. Next, we examine the set of Lévy processes, affine processes, and finally, we present some examples of jump diffusions that are not polynomial for which Theorem~\ref{th: moment formula holomorphic} applies

\subsubsection{Continuous time Markov chains with a finite state space}\label{finspace}

Let $S:=\{x_1,\ldots,x_N\}\subseteq \R^d$ and note that every map $f:S\to \R$ can be seen as the restriction to $S$ of an entire map bounded on $\R$.

Let $X=(X_t)_{t\in [0,T]}$ be  a continuous-time Markov chain with a finite-state space $S$ with $X_0=x_0\in S$ and  generator
        $$\Acal f(x_i)=\sum_{j=1}^N\lambda_{ij}(f(x_j)-f(x_i)),$$
        for some $\lambda_{ij}\geq0$.
    Set $\Vcal:=\{h:S\to \C\}$  and note that since  $1_{\{\cdot =x_i\}}\in \Vcal$ we can fix $\v^i\in \mathcal{S}^*$ such that $h_{\v^i}(x)=1_{\{x=x_i\}}$ for each $i\in\{1,\ldots,N\}$.  
 Observe that 
    for each $k$ and $\ell$,
     $$\Acal h_{\v^k}(x_\ell)=\sum_{j=1}^N \lambda_{\ell j}( 1_{\{k=j\}}- 1_{\{k=\ell \}})
       = \sum_{i=1}^Nh_{\v^i}(x_\ell )\sum_{j=1}^N \lambda_{ij}( 1_{\{k=j\}}- 1_{\{k=i\}})
       =h_{L(\v^k)}(x_\ell ),$$
       where  $L$ denotes the operator given by \eqref{eq:Lgeneralkernel}, which explicitly reads as  $$L(\v^k)=\sum_{i=1}^N \v^i\sum_{j=1}^N\lambda_{ij}(1_{\{k=j\}}-1_{\{k=i\}}).$$
       This 
       shows that $X$ is $\Vcal$-holomorphic.

    We illustrate now how the conditions of Theorem~\ref{th: moment formula holomorphic} can be verified.
        Since for each $h\in\mathcal{V} $ it holds that $h\in\text{span}\{h_{\v^1},\dots,h_{\v^N}\}$,
       the sequence-valued ODE given by \eqref{eq:ODE sequence space holomorphic} can be interpreted as an $N$-dimensional system of linear ODEs, and for $\c(0)=\v^k$ we can explicitly write
       $$\c(t)=\sum_{j=1}^N \v^j \exp(t\widetilde L)_{jk}, \qquad t\in [0,T],$$
        where  $\exp(t\widetilde L)$ denotes the matrix exponential of $t\widetilde L$, where $\widetilde L\in \R^{N\times N}$  is given by $\widetilde L_{ik}=\sum_{j=1}^N\lambda_{ij}(1_{\{k=j\}}-1_{\{k=i\}})$.
       Since conditions~\ref{moment formula ii} and~\ref{moment formula iii} of Theorem~\ref{th: moment formula holomorphic} follow by the finiteness of the state space we can conclude that for each $h\in \Vcal$,
    $$\E[h(X_T)]=h_{\c(T)}(x_0)=(h_{\v^1}(x_0),\ldots,h_{\v^N}(x_0))\exp(T\widetilde L)(h(x_1),\ldots, h(x_N))^\top.$$

       Note that the same approach can be used if the extended generator is mapping the span of a finite number of holomorphic functions to itself, also for non-finite state spaces. This is the case for polynomial processes, where the bases is given e.g.~by monomials. It is also possible to consider an infinite number of basis elements but more conditions need to be verified. A possible approach to guarantee existence of the solution of the ODE is given in Lemma~\ref{lem4new} below.

\subsubsection{Lévy processes}
It is well known that all Lévy processes whose extended generator is well-defined on the space of polynomials are polynomial jump-diffusions (see e.g.~Lemma 1 in \cite{FL:20}). In the next proposition, we prove that they are also $\Vcal$-holomorphic processes for some suitable $\Vcal$ (containing but not being limited to the polynomials) and demonstrate the validity of the holomorphic formula \eqref{eq: holomorphic formula} for a large class of holomorphic functions. For simplicity, we deal with processes with values on (a subset of) $\Rset$. Notice however that most of the analysis could be extended to the more general multidimensional case.

The proof of the next theorem is given in Appendix~\ref{proof: levy processes 1}. Recall the notion of entire function of exponential type from Remark~\ref{remark:nofinitemoments}\ref{entire functions finite order }  and that in this paper we stick to the truncation function $\chi(\xi)=\xi$ 
\begin{theorem}\label{prop: levy processes 1}
    Fix $S\subseteq\Rset$ with and let  $X=(X_t)_{t\in[0,T]}$ be an $S$-valued Lévy process, with characteristics $(b,a,F)$. Assume that $\int_{|y|>1}\exp(\alpha  |y|) F(dy)<\infty$ for some $\alpha>0$.        Set  
\begin{align*}
    \mathcal{V}&\subseteq\{h\in H(\C) \colon h \text{ is of exponential type }\tau<{\alpha}\},\\
       \mathcal{V}^*&:=\{\u\in \SS \colon h_\u=h,  \ h\in \mathcal{V}\}.
\end{align*}
Then,
\begin{enumerate}
\item \label{item: levy1 i}$X$ is an $S$-valued $\mathcal{V}$-holomorphic process. 
    \item \label{item: levy1 ii} For all $h\in \Vcal$, the process $N^h$ given in equation \eqref{eq:localmart} is a true martingale.
    \item \label{item: levy1 iii} For all $\u\in  \mathcal{V}^*$, there exists a $\mathcal{V}^*$-valued solution $(\c(t))_{t\in[0,T]}$ of \eqref{eq: ODE pointwise} with initial condition $\u$ which satisfies condition~\ref{moment formula iii} of Theorem~\ref{th: moment formula holomorphic}. 
 \end{enumerate}
 In particular, for each $\u\in  \mathcal{V}^*$, the holomorphic formula holds true:
    \begin{align*}
    \E[h_\u(X_T)]=h_{\c(T)}(x_0).
\end{align*}
    
\end{theorem}

\begin{remark}\label{rem: Levy 1}
\begin{enumerate}
\item Set $v_\tau(t):=\exp(\tau t)$ and note that $\Vcal = \bigcup_{\tau<\alpha}H_{v_\tau}(\C)$.
Observe moreover that imposing enough integrability on $F$ the result of Proposition~\ref{prop: levy processes 1} can be extended to  all entire functions of finite (but arbitrarily large) exponential type. In particular, if $\int_{|y|>1}\exp(\alpha  |y|) F(dy)<\infty$ for all $\alpha>0$ then the results of the cited proposition hold for
$$\Vcal\subseteq\{h\in H(\C) \colon h \text{ is of finite exponential type }\}=\bigcup_{\tau>0}H_{v_\tau}(\C).$$
This is in particular the case if $F$ has bounded support.

    \item If $\int_{|y|>1}\exp(  |y|) F(dy)<\infty$, with a slight adaptation, the proof of Proposition~\ref{prop: levy processes 1} can be adapted to other sets $\Vcal$. An example is given by the set
    $$\mathcal{V}:=\{h\in H(\C) \colon h=h_\u \text{ for some }\u \text{ such that }\sup_{n\in \N}|\u_n|<\infty\}.$$ 
    Note that $\Vcal=H_v(\C)$ for $v(t)=\exp(t)$. In this case, the only missing piece is given by checking that $h',h''\in \Vcal$ for each $h\in \Vcal$, which can be deduced from the representation of the derivatives introduced in Section~\ref{sec 214}.
\end{enumerate}
\end{remark}

Next we show that if the Lévy measure $F$ has bounded support, 
the results of Proposition~\ref{prop:unboundedjs} hold for a larger set $\Vcal$.  In particular, instead of defining $\Vcal$ through the growth rate of the complex extension of its elements, one can impose conditions on the growth rate of functions and their derivatives only on $\R$. The proof of the Corollary~\ref{coro: levy1 bounded support} can be found in Appendix~\ref{proofcor}.  In the following, given a holomorphic function $h$, we write $h|_\R$ for its restriction on $\R$ and denote by $|h|_\R|$ its absolute value.

 \begin{corollary}\label{coro: levy1 bounded support}
 Fix $S\subseteq\Rset$ and let  $X=(X_t)_{t\in[0,T]}$ be an $S$-valued Lévy process with characteristics $(b,a,F)$. Assume that $F$ has bounded support and set
 \begin{align*}
    \mathcal{V}&:=\{h\in H(\C)\colon |h|_\R|,|h'|_\R|,|h''|_\R|\leq C\exp(a|\cdot|)\text{ for some }a\in \Rset, C>0 \}\\
    \Vcal^*&:=\{\u\in \SS \colon h_\u=h,  \ h\in \mathcal{V}\}.
\end{align*}
The following conditions hold true.
\begin{enumerate}
\item \label{item: levy2 i}$X$ is an $S$-valued $\Vcal$-holomorphic process.
    \item \label{item: levy2 ii} The process $N^h$ given in equation \eqref{eq:localmart} is a true martingale for each $h\in \Vcal$.
    \item \label{item: levy2 iii} 
    Fix $\u\in \Vcal^*$. Assume that 
       \label{b: levy process 2}  
        $\int_\Rset |h_\u(x)|dx<\infty $ or $ \int_\Rset |h_\u(x)|^2 dx<\infty$ and  
     the map $\hat h_\u:\R\to\C$ given by 
        \begin{equation}\label{eqnfur}
    \hat h_\u(u):=\int_\R h_\u(x)e^{-iux}dx
    \end{equation}
 satisfies
\begin{align}\label{eq:condition Fourier}
        \int_\Rset \exp(|ux|)|\hat{h}_\u(u)|du<\infty,
    \end{align}
    for all $x\in \Rset$.  Then there exists a $\mathcal{V}^*$-valued solution $(\c(t))_{t\in[0,T]}$ of \eqref{eq: ODE pointwise} with initial condition $\u$ which satisfies condition~\ref{moment formula ii} and~\ref{moment formula iii} of Theorem~\ref{th: moment formula holomorphic}. In particular, the holomorphic formula holds true:
    \begin{align*}
    \E[h_\u(X_T)]=h_{\c(T)}(x_0).
\end{align*}
\end{enumerate}
    \end{corollary}

     \begin{remark}\phantomsection\label{remarkcarino}
    \begin{enumerate}
    \item The map defined in \eqref{eqnfur} represents the Fourier transform of $h_\u$.

        \item\label{remarkcarinoii} Notice that contrary to the setting of Proposition \eqref{prop: levy processes 1}, functions of the form $h_n(z):=\exp(-z^{2n}),$  are included in the set $\mathcal{V}$ introduced in Corollary~\ref{coro: levy1 bounded support}. Moreover, by equation~(60) in \cite{B:14} (see also Exercise 5 in Chapter 5 in \cite{SS:10}) we know that the tails of  $\hat h_n$ decay at least as $u^{(1-n)/(2n-1)}\exp(-Cu^{2n/(2n-1)})$, for some $C>0$, implying that
          \eqref{eq:condition Fourier} is satisfied too and we can conclude that
        \begin{align*}
            \E[\exp(-X_T^{2n})]=h_{\c(T)}(x_0),
        \end{align*}
        where $(\c(t))_{t\in[0,T]}$ 
 is a $\mathcal{V}^*$-valued solution of \eqref{eq: ODE pointwise} with initial condition $\u$ given as $h_\u(x)=\exp(-x^{2n})$, $x\in \Rset$.  Note then that
  by standard properties of the Fourier transform this class can be further extended. Using that for $a>0$ it holds
 $\widehat {h_\u(a\cdot)}(u)=\frac 1 {a}\hat h_\u(u/a)$ we can indeed include functions of the form $\exp(-ax^{2n})$. Moreover, observe that
 for $|u|>2R$, either $|s|>R$ or $|u-s|>R$ and hence 
 \begin{align*}
     \widehat{h_\u h_\v}(u)&=\int_\R \hat h_\u(u-s)\hat h_\v(s) ds\\
 &\leq \int_{|s|>R} \hat h_\u(u-s)\hat h_\v(s) ds+\int_{|u-s|>R} \hat h_\u(u-s)\hat h_\v(s) ds\\
 &\leq C\sup_{|s|>R}(|\hat h_\u(s)|+|\hat h_\v(s)|),
 \end{align*}
 showing that 
 Corollary ~\ref{coro: levy1 bounded support} also applies to functions of the form $h(z):=\exp(\sum_{k=0}^n a_{2k}z^{2k})$, for $a_{2k}\in \Rset_-$, $n\in \Nset$.
            \item If furthermore the characteristics of the Lévy process $X$ satisfies $\int_{|\xi|\leq1}|\xi|F(d\xi)<\infty$ and $a=0$, Corollary~\ref{coro: levy1 bounded support} holds true by considering the larger set 
            $$ \mathcal{V}:=\{h\in H(\C)\colon |h|_\R|,|h'|_\R||\leq C\exp(a|\cdot|)\text{ for some }a\in \Rset, C>0 \}.$$
    \end{enumerate}
    
\end{remark}

\subsubsection{Affine processes}

It is well known that under some integrability conditions, affine processes are polynomial jump-diffusions (see Corollary 3.3 in \cite{FL:20}). Here, we go further and show that affine processes are holomorphic processes for a larger class of holomorphic functions than polynomials and show the validity of the holomorphic formula \eqref{eq: holomorphic formula}. Notably, the literature encompasses various formulations of affine processes, reflecting subtle differences in their definitions. In this context, we will always consider affine processes as specified in Definition 3.1 in \cite{FL:20}. This is a relaxed definition compared to the definition of an affine process in \cite{DFS:03}, because it is directly given in terms of the point-wise action of the extended generator on exponential-affine functions. The proof of Proposition~\ref{prop: affine} can be found in Appendix~\ref{proof: prop affine}.

\begin{proposition}\label{prop: affine}
      Let $X=(X_t)_{t\in[0,T]}$ be an affine process on $S\subseteq \Rset$.  Suppose that for each $x\in S$, $K(x,d\xi)=\nu_0(d\xi)+x\nu_1(d\xi)$ for some signed measure $\nu_0,\nu_1$ such that both $|\nu_0|$ and $|\nu_1|$ have bounded support.       Set
      \begin{align*}
    \Vcal&:=\{h\in H(\C)\colon h,h',h''\text{ are bounded on }\R\},\\
    \Vcal^*&:=\{\u\in \SS \colon h_\u=h,  \ h\in \mathcal{V}\}.
\end{align*}
The following conditions hold true.
\begin{enumerate}
\item \label{itemaffini i}$X$ is an $S$-valued $\Vcal$-holomorphic process.
    \item \label{itemaffini ii}The process $N^h$ given in equation \eqref{eq:localmart} is a true martingale for each $h\in \Vcal$.
    \item \label{itemaffini iii}
    Fix $\u\in \Vcal^*$. Assume that $h_\u(x)=\int_{-\tau}^\tau \exp((\varepsilon+iu)x)g(u) du$ for some $g:[-\tau,\tau]\to \R$ with $\int_{-\tau}^\tau |g(u)|du<\infty$, $\varepsilon,\tau\in[0,\infty)$ and 
 $$\E[\exp((\varepsilon+iu)X_t)|X_0=x]=\exp(\phi(t,\varepsilon+iu)+\psi(t,\varepsilon+iu)x),$$
 for all $u\in [-\tau,\tau]$ and $t\in [0,T]$. Then there exists a $\mathcal{V}^*$-valued solution $(\c(t))_{t\in[0,T]}$ of \eqref{eq: ODE pointwise} with initial condition $\u$ which satisfies condition~\ref{moment formula ii} and~\ref{moment formula iii} of Theorem~\ref{th: moment formula holomorphic}. 
 
 In particular, the holomorphic formula holds true:
    \begin{align*}
    \E[h_\u(X_T)]=h_{\c(T)}(x_0).
\end{align*}

\end{enumerate}
\end{proposition}

\begin{remark}
\begin{enumerate}
\item Observe that by the Paley--Wiener theorem (see Proposition~\ref{propE3}\ref{exp4}) if $h$ is of exponential type  and 
    $$\int_\R |h(x)|^2\exp(-2\varepsilon x)dx<\infty,$$ then
    $$h(z)=\exp(\varepsilon z)\int_{-\tau}^\tau \exp(iuz)g(u) du,$$
    for some $g\in L^2(-\tau,\tau)$, $\tau>0$ and $z\in \C$.

    Other classes of functions for which the proof of Theorem~\ref{th: moment formula holomorphic} works  is given by polynomials and the Fourier basis. In this case we indeed already know from the classical theory that the semigroup maps such functions to entire functions.
\item    As mentioned at the beginning of the section, for simplicity here we are dealing with processes with values on (a subset of) $\Rset$.
    Notice, however, that the above analysis could be extended to affine processes on more general state spaces. A first example in this direction concerns affine processes on the canonical state space $\Rset^m_+\times \Rset^n$, for some $m,n\in \N_0$, as introduced in \cite{DFS:03}. 
    \end{enumerate}
\end{remark}

\subsubsection{Beyond polynomial processes}

So far we have considered instances of polynomial processes, namely L\'evy and affine processes,  as examples of holomorphic ones, and extended the moment formula for polynomials to the holomorphic formula for classes of holomorphic functions. In this section, we go beyond polynomial processes. In particular,
we
exploit Lemma~5.11 in \cite{CST:23} to construct holomorphic jump-diffusions taking values on $S:=[0,1]$,  for which the conditions of Theorem \ref{th: moment formula holomorphic} are satisfied. Let $\Acal:C^2(\R;\C)\to M(\R;\C)$ be the extended generator of a $[0,1]$ valued $\mathcal{V}$-holomorphic process $X$, for $\mathcal{V}:=\Ocal(S)$. Let $(b,a,K)$ be the corresponding characteristics with respect to the truncation function $\chi(\xi)=\xi$.

We report here an adaptation of the statement of the lemma for the reader's convenience. In order to simplify the notation, set
$$ \mathcal B_C:=\{{\bf \mu}=(\mu_k)_{k \in \N_0}\colon {\bf \mu}_k\geq0,\ \sum_{k=0}^\infty{\bf \mu}_k\leq C\},\qquad \mathcal B:=\bigcup_{C>0}\mathcal B_C.$$
\begin{lemma}\label{lem4new}
Fix $T >0$,  $\lambda=(\lambda_k)_{k \in \N_0}$, and  $\mu\in \mathcal B$ such that
$$(\lambda_0\mu_0,\lambda_1\mu_1,\lambda_2\mu_2,\ldots)\in \SS.$$
Assume that
$$\Acal (\lambda_k\frac{(\cdot)^k}{k!})(x)=\sum_{j=0}^\infty \beta_{kj}(\lambda_j\frac{x^j}{j!}-\lambda_k\frac{x^k}{k!})+\beta_k\lambda_k\frac {x^k}{k!}$$
for some $\beta_{kj}\in \R_+$ and $\beta_k\in \R$ such that $\sup_{k\geq 0}\beta_k^+<\infty$ and $\lim_{j\to\infty}\beta_{kj}=0$ for each $k\in \N_0$. 
If for all $x\in (-\e,1+\e)$ the sequence 
\begin{equation}\label{eqn6Z}
\Big(\lambda_k\frac{x^k}{k!},\sum_{j=0}^\infty \beta_{kj}(\lambda_j\frac{x^j}{j!}-\lambda_k\frac{x^k}{k!})\Big)_{k\in\N_0}
\end{equation}
lies in the bounded pointwise closure of 
\begin{equation*}
\Big\{\Big(\a_k,\sum_{j=0}^\infty \beta_{kj}(\a_j-\a_k)\Big)_{k\in \N_0}\colon \a=(\a_k)_{k=0}^N,\quad N\in \N\Big\},
\end{equation*}
then the linear ODE given by \eqref{eq:ODE sequence space holomorphic} admits a solution $({\bf c}(t))_{t\in[0,T]}$ satisfying \eqref{eq: ODE pointwise} for each $x\in S$ of the form  ${\bf c}_k(t)=\lambda_k\mu_k(t)$ for some $\mu_k(t)\geq0$. If $\beta_k\leq 0$ for each $k$, then  $\sum_{k=0}^\infty\mu_k(t)\leq \sum_{k=0}^\infty\mu_k$.
\end{lemma}
\begin{proof}
The proof follows the proof of Lemma~5.11. In the last part, where the form of $\Acal$ plays a role, note that since 
$h_N:=\sum_{k=0}^N \mu_k(t) (\cdot)^k/{k!}$
 converges to $\sum_{k=0}^\infty \mu_k(t) (\cdot)^k/{k!}=h_{{\bf c}(t)}$  uniformly on $S$ and the same holds for the corresponding derivatives. 
 Since for each $x\in [0,1]$,
\begin{align*}
\int_{\Rset}h_{\c(t)}&(x+\xi)-h_{\c(t)}(x)-{h'}_{\c(t)}\xi K(x,d\xi)\\
&\leq 
   \sup_N\sup_{x\in [0,1]}(|h_N(x)|+|h_N'(x)|+|h_N''(x)|)  \int_{\Rset}|\xi|\land |\xi|^2 K(x,d\xi)<\infty,
\end{align*}
 an application of the dominated convergence theorem yields that 
   $ \Acal h_{{\bf c}(t)}
=\sum_{k=0}^\infty \mu_k(t)\Acal (\lambda_k\frac{(\cdot)^k}{k!}).
$
\end{proof}

Consider the jump-diffusion on $[0,1]$ whose coefficients $(b,a,K)$ with respect to the truncation function $\chi(\xi)=\xi$ are given by
$$b(x)=0,\qquad
a(x)= x(1-x)(1-x/2),\qquad K(x,\cdot)=\frac{(1-x)(1-x/2)}x1_{\{x\neq 0\}} \delta_{-x}(\cdot),$$ where $\delta_{-x}$ denotes the Dirac measure in $(-x)$. Since its generator $\Acal$ is given by
$$\Acal f(x)
= \frac 1 2 a(x)f''(x)
+\frac{(1-x)(1-x/2)}x\big(f(0)-f(x)+xf'(x)\big),$$
we can see that it is also a $[0,1]$-valued $\Vcal$-holomorphic process for $\Vcal=\Ocal(S)$. Observe also that this is a continuous martingale that can perform a jump to the state 0, where it will then get absorbed. Setting $f_k(x):=(x/2)^k$ we can see that the generator of this process satisfies
$$\Acal f_k(x)
=\frac {(k+2)(k-1)}4(f_{k-1}+2f_{k+1}-3f_k)1_{\{k\geq 2\}},$$
which implies that $\Acal$ is of the form described in Lemma~\ref{lem4new} for  $\lambda_k:=(1/2)^k$, $\beta_{0j}=\beta_{1j}=0$,
$$\beta_{kj}=\frac {(k+2)(k-1)}41_{\{j=k-1\}}+\Big(\frac {(k+2)(k-1)}2\Big)1_{\{j=k+1\}},\qquad k\geq 2,$$
and $\beta_k=0$. Following the examples in Section~6.2 of \cite{CST:23} we can then apply Lemma~\ref{lem4new} to conclude that there exists a solution $(\c(t))_{t\in [0,T]}$ of \eqref{eq:ODE sequence space holomorphic} satisfying \eqref{eq: ODE pointwise} for each initial condition $\u\in \Wcal^*$ where, recalling that $S:=[0,1]$,
$$\Wcal^*:=\{\u\in \S^*\colon\u_k=\mu_k \frac {k!}{2^k}\text{ for some $\mu_k\in \R_+$ such that $\sum_{k=0}^\infty \mu_k<\infty$}\}.$$
Since condition~\ref{moment formula ii} of Theorem~\ref{th: moment formula holomorphic} is always satisfied for bounded state spaces and condition~\ref{moment formula iii} of the same theorem can be verified using that $\Acal f_k(x)\geq0 $ for each $k\in \N_0$ and  $x\in [0,1]$ we can conclude that Theorem~\ref{th: moment formula holomorphic} can be applied yielding 
\begin{align}\label{momcomp}
    \E[h_\u(X_T)]=h_{\c(T)}(x_0).
\end{align}
for each $\u\in \Wcal^*$. This includes in particular
$\u=(1,u,u^2,\ldots)$
corresponding to $h_\u(x)=\exp(ux)$ for each $u\in \R_+$.

An inspection of the proof of Lemma~5.11 gives us also a constructive description of the process $(\c(t))_{t\in[0,T]}$  appearing in \eqref{momcomp}. Assuming for simplicity that $\sum_{i=0}^\infty \u_i2^i/i!=1$ we indeed get that
$$\c(T)_i:= \mathbb P(Z_T=i)i!/2^i,$$ where $(Z_t)_{t\in[0,T]}$ is the $\N_0$-valued process satisfying 
$$\mathbb P(Z_0=i)=\u_i\frac {2^i}{i!}$$
and jumping from state $k$ to state $j$ after an $\exp(\beta_{kj})$-distributed random time.

\section{Affine-holomorphic jump-diffusions}\label{sec: affine holomorphic}
Fix  $\S\subseteq\Rset^d$ and $T>0$.
\begin{definition}\label{def:affine-holomorphic process}
    Let $X=(X_t)_{t\in [0,T]}$ be an $S$-valued jump-diffusion with extended generator $\mathcal{A}$ and fix a  subset $\mathcal{V}\subseteq \mathcal{O}(S)$.  We say that $X$ is an $S$-valued $\mathcal{V}$-\textit{affine-holomorphic process} if there exists an operator $\mathcal{R}:\Vcal\rightarrow\mathcal{O}(S)$ such that for each  $f\in\Vcal$ it holds that $\exp(f)\in\Dcal(\Acal)$, 
    \begin{align*}
        \mathcal{A}\exp(f)=\exp(f)\mathcal{R}(f),
    \end{align*}
    and the process $N^{\exp(f)}$ introduced in equation \eqref{eq:localmart} defines a local martingale.
\end{definition}

Let $X=(X_t)_{t\in [0,T]}$ be a $S$-valued $\mathcal{V}$-affine-holomorphic process, for some subset $\mathcal{V}\subseteq \mathcal{O}(S)$, and let  $\mathcal{V}^*$ as in  \eqref{eq: V^*}.
Notice that by the properties of the functions in $\mathcal{O}(S)$ there always exists a  map $ R: \mathcal{V}^*\rightarrow \mathcal{O}(S)$
such that for all $\u\in\mathcal{V}^*$,$$\mathcal{A}\exp(h_\u)=\exp(h_{\u})R(h_\u)|_S.$$ Observe furthermore that also here such a map might be not unique.

\begin{remark}\label{rem : affine infinite dim}
    Note that every $S$-valued affine jump diffusion in the sense of Definition 2 in \cite{FL:20} is an $S$-valued $\Vcal$-affine-holomorphic process, for $\Vcal=   \{h:\Rset^d \rightarrow \C \ \colon \ h(x):=iu^\top x, \ u\in \Rset^d\}$.  Moreover for an $S$-valued $\mathcal{V}$-affine-holomorphic process, the infinite-dimensional process defined in Equation \eqref{eq: boldX infinite} can be viewed as an $\mathcal{V}^*$-affine process in the extended tensor algebra of $\Rset^d$ in the sense of Definition 3.6 in \cite{CST:23}.
\end{remark}

\subsection{Characteristics of affine-holomorphic jump-diffusions}\label{sec: affine charact}
In this section, we establish some sufficient conditions for a jump-diffusion process to be affine-holomorphic. In alignment with the section on holomorphic processes, we do not elaborate on existence results here. Specifically, considering kernels with holomorphic jump sizes
 we investigate the criteria on its characteristics guaranteeing the affine-holomorphic property.

Recall that for $\u,\v\in\SS$, we denote by $\textbf{exp*}(\u)\in\SS$ some coefficients determining the power series representation of $\exp(h_\u)$ on $S$, namely $h_{\textbf{exp*}(\u)}|_S:=\exp(h_\u)|_S$, by $\u\circ^s\v$ some coefficients determining the power series representation of $h_\u(\cdot +h_{\underline{\v}}(\cdot))$ on $S$, whenever the latter is well defined (see Section~\ref{sec 214}), and that $\mathbf{1}:=(\mathbf{1}_{\alpha})_{\alpha\in \N^d_0}$ denotes the sequence such that $\mathbf{1}_{\alpha}=1$ if $\alpha=(0,\dots,0)$ and $\mathbf{1}_{\alpha}=0$ otherwise. Recall also the following notation introduced in Theorem~\ref{ps to ps}. We have
\begin{equation}\label{eqnDJ}
\begin{aligned}
Jh(z,y)&:=\lambda(z)(h(z+j_\vare(z,y))-h(z)-\nabla h(z)^\top j_\vare(z,y)),\\
    \Dcal(J)&:=\{h\in H(P_G^d(0))\colon Jh(z,\cdot)\in L^1(E,F)\text{ for each }z\in P^d_{R(S)+\vare}(0)\},
\end{aligned}
\end{equation}
for each $h\in H(P_G^d(0))$, $z\in P^d_{R(S)+\vare}(0)$, and $y\in E$.
The proof of the next theorem follows the proof of Theorem~\ref{ps to ps}.
\begin{theorem}\label{theorem: new}
Let $X=(X_t)_{t\in [0,T}$ be an $S$-valued jump-diffusion with characteristics $(b,a,K)$ and extended generator $\Acal$.  Assume that $b_j,a_{ij}\in\Ocal(S)$, $K$ is a kernel with holomorphic jump size, and let $F$ be the corresponding non-negative measure on the measurable space $E$. 
Let $G$ satisfy \eqref{eqnG} and
for each $h\in H(P_G^d(0))$, $z\in P^d_{R(S)+\vare}(0)$, and $y\in E$ set $\Dcal(J)$ and $Jh$ as in \eqref{eqnDJ}.
Set 
 \begin{align*}
     \mathcal{V}&\subseteq\{h\in H(P^d_G(0))\colon e^h\in \Dcal(J)
     \text{ and } P^d_{R(S)+\vare}\ni z\mapsto Je^h(z,\cdot)\in L^1(E,F) \text{ is continuous}\},
\\  \mathcal{V}^*&:=\{\u\in \SS \colon h_\u=h|_{{P}^d_{R(S)}(0)},  \ h\in \mathcal{V}\}.
 \end{align*}
Then $X$ is an $S$-valued $\mathcal{V}$-affine-holomorphic process and for all $\u\in\mathcal{V}^*$, $$\mathcal{A}\exp(h_\u)=\exp(h_{\u})R(h_\u)|_S,$$ where $R:\mathcal{V}^*\rightarrow \S^*$ is given by
\begin{align}\label{eq:R1}
    R(\u)&:= \sum_{|\beta|=1}\u^{(\beta)}\ast \underline{\b}^{\beta} 
    +\sum_{|\beta_1|,|\beta_2|=1}\frac{1}{(\beta_1+\beta_2)!}\underline{\underline{\a}}^{\beta_1+\beta_2}\ast ({\u^{(\beta_1)}}\ast\u^{(\beta_2)}+\u^{(\beta_1+\beta_2)})\\
    &+\la\ast \int_{E}\bm{\exp}^*\big(\u\circ^s\underline\j( y)-\u\big)-\mathbf{1}-\sum_{|\beta|=1}\u^{(\beta)}\ast \underline{\bm{j}}( y)^{\ast\beta}\ F(d y) .\nonumber 
\end{align}
\end{theorem}
In the next lemma, we show that with a slightly stronger assumption on $h\in \Vcal$ we obtain a nicer representation of the operator $R$ in \eqref{eq:R1}. The proof of the next result is given in Appendix~\ref{proof: theorem new}.
\begin{lemma}\label{lem:niceR}
    If $h_\u\in \Vcal$ satisfies $h_\u\in 
    \Dcal(J)$ and the map $P^d_{R(S)+\vare}\ni z\mapsto Jh(z,\cdot)\in L^1(E,F)$ is continuous, we also get the representation
    \begin{align}\label{eq:R1nice}
    R(\u):= L(\u)+&
\sum_{|\beta_1|,|\beta_2|=1}\frac{1}{(\beta_1+\beta_2)!}\underline{\underline{\a}}^{\beta_1+\beta_2}\ast ({\u^{(\beta_1)}}\ast\u^{(\beta_2)})\\
    +&\la\ast \int_{E}\bm{\exp}^*\big(\u\circ^s\underline\j( y)-\u\big)-\mathbf{1}-\big(\u\circ^s\underline\j( y)-\u\big)\ F(d y) ,\nonumber 
\end{align}
for $L$ as in equation \eqref{eq:OpLu}.
\end{lemma}

As for the holomorphic case, several corollaries can then be deduced by Theorem~\ref{theorem: new}.
The proof of the following result is given in Appendix~\ref{proof: theorem affine weight}.
\begin{corollary}\label{theorem: affine_weigh}
Let $X=(X_t)_{t\in[0,T]}$ be an $S$-valued jump-diffusion with characteristics $(b,a,K)$ and extended generator $\Acal$. Fix a weight function $v$ in the sense of Definition~\ref{def2}. Assume that $b_j,a_{ij}\in\Ocal(S)$, $K$ is a kernel with holomorphic jump size, and let $F$ be the corresponding non-negative measure on the measurable space $E$.  Fix $\vare>0$ and suppose that for each $z\in P^d_{R(S)+\vare}(0)$, there exists $\delta_z>0$ such that
 \begin{align}
         &\int_E  \ \sup_{w\in P^d_{\delta_z}(z)}  \|j_\vare(w, y)\|\land\|j_\vare(w, y)\|^2 \ F(d y)<\infty,\label{eq: affine weight}\\
  &\int_E  \sup_{w\in P^d_{\delta_z}(z)} 1_{\{\|j_\vare(w, y)\|> 1\}} \exp(m \ v(\|w+j_\vare(w, y)\|))\ F(d y)<\infty,\label{eq:M affine weight}
  \end{align}
  for some $m\in \R_+$. Let $G\in (0,\infty]^d$ satisfy \eqref{eqnG} and set
\begin{align*}
    \mathcal{V}:=H_v(P^d_G(0)) \qquad \text{and}\qquad 
     \mathcal{V}^*:=\{\u\in \SS \colon h_\u=h|_{{P}^d_{R(S)}(0)},  \ h\in \mathcal{V}\}.\nonumber 
\end{align*}
Then $X$ is an $S$-valued $\mathcal{V}$-affine-holomorphic process and for all $\u\in\mathcal{V}^*$, $$\mathcal{A}\exp(h_\u)=\exp(h_{\u})R(h_\u)|_S,$$ where $R:\mathcal{V}^*\rightarrow \S^*$ is given by
\eqref{eq:R1nice}.
\end{corollary} 
Also in this case similar considerations as in Remark~\ref{remark:nofinitemoments} apply. 
\begin{remark}
Suppose that jump sizes are locally uniformly bounded in the sense of \eqref{eqn_bdd}.
     In this case condition \eqref{eq:M affine weight}  is implied by \eqref{eq: affine weight} for each weight function $v$ and the result of Corollary~\ref{theorem: affine_weigh} holds for $\Vcal=H(P_G^d(0))$. This is of particular interest when $S$ is bounded and condition \eqref{eqn_bdd} is automatically satisfied.
\end{remark}

Finally, we analyze the class of jump-diffusion processes with jump sizes that do not depend on the current value of the process.  The proof of the following proposition can be found in Appendix~\ref{proof: prop_affine_Levy}. Here $\Im(z)$ denotes the imaginary part of $z$.
\begin{corollary}\label{prop:affine_Levy}
Let $X=(X_t)_{t\in[0,T]}$ be an $S$-valued jump-diffusion with characteristics $(b,a,K)$ and extended generator $\Acal$. Assume that $b,a,\lambda\in\Ocal(S)$ and $K(x,d\xi)=\lambda(x)F(d\xi)$. Let $G\in (0,\infty]^d$ satisfy \eqref{eqnG} and set
\begin{align*}
    \mathcal{V}&\subseteq\{h\in H(P_G^d(0)) \colon |h(z)|\leq  g_h(\Im(z)), 
   \text{ for $g_h:\Rset\rightarrow\Rset_+$ continuous, $z\in \C$}\},\nonumber \\
 \mathcal{V}^*&:=\{\u\in \SS \colon h_\u=h|_{{P}^d_{R(S)}(0)},  \ h\in \mathcal{V}\}. \nonumber
\end{align*}
     Then $X$ is an $S$-valued $\mathcal{V}$-affine-holomorphic process and for all $\u\in\mathcal{V}^*$, $$\mathcal{A}\exp(h_\u)=\exp(h_\u)R(h_{\u})|_S,$$ where $R:\mathcal{V}^*\rightarrow \S^*$ is given by \eqref{eq:R1nice}
for $\j(\xi)=\xi\mathbf{1}$. 
\end{corollary}
\begin{remark}\phantomsection\label{rem: affine_Levy_affine}
\begin{enumerate}
\item Observe that for $d=1$ the operator $R$ given by \eqref{eq:R1nice} reads
     \begin{align}\label{eq:R_1dim}
    R(\u)= L(\u)&+\frac{1}{2}(\a\ast \u^{(1)}\ast\u^{(1)})_\alpha\\
    &+\la \ast \int_{\Rset}\bm{\exp}^*\big(\u\circ^s\j(\xi)-\u\big)-\mathbf{1}-\big(\u\circ^s\j(\xi)-\u\big)\ F(d \xi),\nonumber 
\end{align}
where  $L$ denotes the operator given in Equation \eqref{eq:L_1dim}.
    \item Fix again  $d=1$. By Proposition~\ref{propE3}, the set  $B_\tau$  of  
 all entire functions of order not exceeding 1 and of type not exceeding $\tau\in (0,\infty)$, which are bounded on $\R$ is included in the set $\mathcal{V}$ specified in Corollary~\ref{prop:affine_Levy}.

    More generally, we can extend the above result beyond $B_\tau$ and consider for instance functions of the form $h(z):=\exp(-z^2)$ for every $z=(x+iy)\in \mathbb{C}$, for which a direct computation shows that $|\exp(-z^2)|\leq \exp(y^2)$ for every $z$.

     \item  \label{rem: affine_Lévy_affine ii}Observe that a slight adaptation of Corollary~\ref{prop:affine_Levy} applies in particular if for all $x\in S$  
        \begin{align*}
            K(x,d\xi)=F_0(d\xi)+xF_1(d\xi),
        \end{align*}
        for some signed measures $F_0(d\xi)$, $F_1(d\xi)$ on $\Rset^d$ for which $\int_{\Rset^d}\|\xi\|\wedge\|\xi\|^2 |F_i|(d\xi)<\infty$, for $i=0,1$.
        That is, $K$ is the compensator of the jump measure of a $S$-valued  affine jump-diffusion (see e.g. \cite{FL:20}).

        Additionally, kernels of the form $K(x,d\xi)=\sum_{j=0}^N \lambda_j(x)F_j(d\xi)$, for some positive entire functions $\lambda_j(x)$ and positive measures   $F_j(d\xi)$ and limits thereof could also be considered (see for instance Section~2.3 in \cite{CLS:18}).
\end{enumerate}

\end{remark}
\begin{remark}
    Note that often a $\Vcal$-holomorphic process is also $\Vcal$-affine-holomorphic. This is however not true in general as one can see considering $\Vcal$ consisting of linear maps only and $\Acal$ being the generator of a diffusion process.
    
    Vice versa, observe that each $\Vcal$-affine-holomorphic process is an $\Wcal$-holomorphic process for $\Wcal:=\{\exp(h)\colon h\in \Vcal\}$. 
\end{remark}

\subsection{The affine-holomorphic transform formula}
 
Fix $\Vcal\subseteq\mathcal{O}(S)$ and let $X=(X_t)_{t\in[0,T]}$ be an $S$-valued $\Vcal$-affine-holomorphic process. Denote by $\mathcal{A}$ its extended generator, and by $R:\mathcal{V}^*\rightarrow \SS$ a linear operator such that 
$\mathcal{A}\exp(h_\u)=\exp(h_\u)h_{R(\u)}|_S$, for all $\u\in \Vcal^*$, where $\Vcal^*$ denotes the set of coefficients determining some power series representation of the functions in $\Vcal$. 

Parallel to the discussion pertaining to holomorphic processes in Section~\ref{sec: the holomorphic formula}, we here exploit duality methods to compute the expected value of $\exp(h_\u(X_t))$, for $t\in [0,T]$ and $h_\u\in \mathcal{V}$.

Recall from Section~\ref{sec212} that integrals of sequences are defined componentwise.

\begin{theorem}\label{th: affine formula}
  Set $X_0=x_0\in S$, fix $\u\in \Vcal^*$, and suppose that the following conditions hold true.
\begin{enumerate}\phantomsection
    \item \label{affine formula i}The sequence-valued  ODE
    \begin{align}\label{eq:ODE sequence space affine}
        \ppsi(t)=\u+\int_0^tR(\ppsi(s))ds, \qquad t\in [0,T],
    \end{align}
    admits a $\Vcal^*$-valued solution $(\ppsi(t))_{t\in[0,T]}$ such that 
   \begin{align}\label{eq: ODE Riccati pointwise}
       h_{\ppsi(t)}(x)=h_\u(x)+\int_0^th_{R(\ppsi(s))}(x)ds,
   \end{align}
  for all $x\in S$, $t\in[0,T]$.
    \item \label{affine formula ii}The process $N^{\exp(h_{\ppsi(s)})}$ given in equation \eqref{eq:localmart} defines a true martingale for each $s\in [0,T].$
    \item \label{affine formula iii}$\int_0^T\int_0^T\E[|\mathcal{A}\exp(h_{\ppsi(s)})(X_t)|]dsdt<\infty$.
\end{enumerate}
Then it holds that 
\begin{align}\label{eq: affine holo formula}
    \E[\exp(h_\u(X_T))]=\exp(h_{\ppsi(T)})(x_0).
\end{align}
\end{theorem}
The proof of Theorem~\ref{th: affine formula} follows the proof of Theorem~\ref{th: moment formula holomorphic}. However, for the sake of completeness, we include it below.
\begin{proof}
    We verify the conditions of Lemma A.1 in \cite{CST:23}. Set
    \begin{align*}
           Y^1(s,t):=\exp(h_{\ppsi(s)})(X_t),\qquad Y^2(s,t):=\exp(h_{\ppsi(s)}) h_{R(\ppsi(s))}(X_t),
    \end{align*}
    By continuity of $Y^1(\cdot,t)(\omega)$ and measurability (on $[0,T]\times \Omega$) of $Y^1(s,\cdot)$ the two maps $Y^1,Y^2:[0,T]\times [0,T]\times \Omega\longrightarrow\C$ are measurable functions. 
    Next, observe that by the assumption~\ref{moment formula ii} it holds that for each $s,t\in[0,T]$, the process $(N^{\exp(h_{\c(s)})}_t)_{t\in[0,T]}$, whose explicit form reads 
       \begin{align}\label{eq: true martingale affine }
           \exp(h_{\ppsi(s)})(X_t)-\exp(h_{\ppsi(s)})(x_0)-\int_0^t\exp(h_{\ppsi(s)})h_{R(\ppsi(s))}(X_u)du, 
       \end{align}
       is a true martingale. Moreover,  by  assumption~\ref{affine formula i} 
       it holds that for every    $s,t\in [0,T]$,
    \begin{align}\label{eq:solution Riccati ode}
      h_{\ppsi(s)}(X_t)-h_{\ppsi(0)}(X_t)-\int_0^s  h_{R(\ppsi(u))}(X_t)du=0.
    \end{align}
   Finally, taking expectation in both equations \eqref{eq: true martingale affine }, \eqref{eq:solution Riccati ode} and by assumption~\ref{moment formula iii},  all the hypothesis of the above mentioned lemma are satisfied, and thus the claim follows.
\end{proof}

\begin{remark}
\begin{enumerate}
    \item  Also in this case, an inspection of the proof shows that \eqref{eq:ODE sequence space affine} is not necessary. Specifically, condition~\ref{moment formula i} can be replaced by the assumption of the existence
of a $\Vcal^*$-valued map $(\ppsi(t))_{t\in[0,T]}$  such that $$
       h_{\ppsi(t)}(x)=h_\u(x)+\int_0^th_{L(\ppsi(s))}(x)ds$$
  for all $x\in S$, $t\in[0,T]$.

  On the other hand if  $\int_0^T|R(\ppsi(s))|_xds<\infty$ for all $x\in S$, $t\in [0,T]$ then condition \eqref{eq:ODE sequence space affine} implies \eqref{eq: ODE Riccati pointwise}.

  \item Recall that the operator $R$ might not be unique. Again, this is not an issue as explained
in Remark~\ref{rem2}\ref{rem2_2}.
       \item Notice that the claim of Theorem~\ref{th: affine formula} for the $\mathcal{V}$-affine-holomorphic process $X$ coincides with the assertion of Theorem 3.9 in \cite{CST:23} for the (infinite-dimensional) $\mathcal{V}^*$-affine process $\mathbb{X}$ introduced in Equation \eqref{eq: boldX infinite} and discussed in Remark~\ref{rem : affine infinite dim}.
    
    \item Let $X$ be an $S$-valued affine jump-diffusion and recall from Remark~\ref{rem : affine infinite dim} that $X$ is an $S$-valued $\mathcal{V}$-affine-holomorphic process, for 
    $\Vcal=   \{h:\Rset^d \rightarrow \C \ \colon \ h(x):=iu^\top x, \ u\in \Rset^d\}$. In this case, formula \eqref{eq: holomorphic formula} coincides with the so-called affine transform formula in Theorem 2 in \cite{FL:20}. 
\end{enumerate}
\end{remark}

The remaining part of the section is dedicated to the study of sufficient conditions for the application of Theorem~\ref{th: affine formula}. In particular, we specify some conditions for the assumptions~\ref{affine formula ii} and~\ref{affine formula iii} in Theorem~\ref{th: affine formula} to be satisfied. 
Additionally to the assumption made at the beginning of the section, we suppose that for each $\u\in\Vcal^*$, there exists a $\Vcal^*$-valued solution of the linear ODE \eqref{eq:ODE sequence space affine}, with initial value $\u$, that we denote by $(\ppsi(t))_{t\in[0,T]}$.

The next lemma pertains to affine-holomorphic processes whose extended generator $\mathcal{A}$ acts between (the restriction on $S$ of) weighted spaces of entire functions (see Definition~\ref{def2}). The proof of the following result follows the proof of Lemma~\ref{lemma: lemma moment formula weights}.

\begin{lemma}
Assume that $\{\exp(h)\colon h\in\Vcal\}\subseteq H_v(\C^d)$ and  $\{ \Acal(\exp(h))\colon h\in \Vcal\}\subseteq H_w(\C^d) $, for some weight functions $v$ and $w$. Suppose moreover that  
 $ \E[\sup_{t\leq T}v(\|X_t\|)]$,  $\E[\sup_{t\leq T}w(\|X_t\|)],$ and $
       \int_0^T\|h_{\bm{\exp}^*(\ppsi(s))*\mathcal R(\ppsi(s))}\|_wds$ 
 are finite.
Then conditions~\ref{affine formula ii} and~\ref{affine formula iii} of Theorem~\ref{th: affine formula} are satisfied.
  
\end{lemma}

Finally, we discuss the case of affine-holomorphic processes with values in a bounded state space $S$. An example of such processes is given by an affine process with compact state space (see \cite{KL:2018}) for which Corollary~\ref{prop:affine_Levy} directly applies (see Remark~\ref{rem: affine_Levy_affine}\ref{rem: affine_Lévy_affine ii}). This result is the equivalent of Lemma~\ref{lemma: i ii moment formula bounded} and since the proof is analogous we will omit it.

\begin{lemma}\label{lemma: i ii affine formula bounded}
    Assume that $S$ is a bounded set and that 
    \begin{align}\label{eq: integral R s compact}
        \int_0^T | \bm{\exp}^*(\ppsi(s))*R(\ppsi(s))|_{R(S)}ds<\infty.
    \end{align} Then condition \eqref{eq:ODE sequence space affine} implies \eqref{eq: ODE Riccati pointwise} and 
    conditions~\ref{affine formula ii} and~\ref{affine formula iii} of Theorem~\ref{th: affine formula} are satisfied.
\end{lemma}
 \subsection{Applications}
Here we provide explicit examples illustrating how the affine holomorphic formula can be effectively applied.

\begin{example}
    As first illustration we consider again the finite state space $S:=\{x_1,\ldots,x_N\}\subseteq \R^d$ seen in Section~\ref{finspace}. Let $X=(X_t)_{t\in [0,T]}$ be  a continuous-time Markov chain with a finite-state space $S$ with $X_0=x_0\in S$ and  generator
        $$\Acal f(x_i)=\sum_{j=1}^N\lambda_{ij}(f(x_j)-f(x_i)),$$
        for some $\lambda_{ij}\geq0$.
Set $\Vcal:=\{h:S\to \R\}$ and note that since  $1_{\{\cdot =x_i\}}\in \Vcal$ we can fix $\v^i\in \mathcal{S}^*$ such that $h_{\v^i}(x)=1_{\{x=x_i\}}$. for   each $i\in\{1,\ldots, N\}$. Note that the linear operator $R$ given by \eqref{eq:R1} reads as 
    $$R(\u)=\sum_{i=1}^N \v^i\sum_{j=1}^N\lambda_{ij}(\exp(\u_j-\u_i)-1),$$
for each $\u$ such that $h_\u\in \Vcal$ and $\u_i:=h_\u(x_i)$.
    Solving $\ppsi(t)=\u+\int_0^tR(\ppsi(s))ds$
    is of course more involved than solving a  system of linear ODEs. As a first step observe that  setting $R(\u)_i:=h_{R(\u)}(x_i)$, we obtain the finite-dimensional system
    
    \begin{equation}\label{psifindim}
        \ppsi(t)_i=\u_i+\int_0^tR(\ppsi(s))_ids,\qquad i\in \{0,\ldots,N\}.
    \end{equation}
    For fixed $i\in\{1,\ldots, N\}$ and $\u\in \Vcal^*$, from Section~\ref{finspace} we know that setting 
    $$\c(t)_i:=
    (h_{\v^1}(x_i),\ldots,h_{\v^N}(x_i))\exp(t\widetilde L)(\exp(h_\u(x_1)),\ldots, \exp(h_\u(x_N)))^\top,$$
    for $\widetilde L\in \R^{N\times N}$  given by $\widetilde L_{ik}=\sum_{j=1}^N\lambda_{ij}(1_{\{k=j\}}-1_{\{k=i\}})$ and  $\exp(t\widetilde L)$ denoting the matrix exponential of $t\widetilde L$, we get that
    $$\c(t)_i=\exp(h_\u(x_i))+\int_0^th_{L(\c(s))}(x_i)ds$$
for  
$L(\u)=\sum_{i=1}^N \v^i\sum_{j=1}^N\lambda_{ij}(\u_j-\u_i)$
Note that since $\u_i\in \R$ and $\c(t)_i=\E[\exp(h_\u(X^{x_i}_t))]$, where $X^{x_i}$ is given by $X$ for $x_0=x_i$ we also have that $\c(t)_i>0$.

This in particular implies that $\ppsi(t)$ given by 
    $\ppsi(t)_i:=\log(\c(t)_i)$
      solves \eqref{psifindim}.
Since the other conditions of Theorem~\ref{th: affine formula}  are satisfied by finiteness of the state space, we can conclude that
$$\E[\exp(h_\u(X_T^{x_i}))]=\exp(h_{\ppsi(T)})(x_i)=\exp(\ppsi(T)_i).
$$
In some cases \eqref{psifindim} can also be solved explicitly. This is for instance the case for $N=2$. The system of ODEs in this case reads as
 \begin{equation*}
        \ppsi(t)_1=\u_1+\int_0^t
        \lambda_{12}(\exp(\ppsi(t)_2-\ppsi(t)_1)-1)
        ds,\qquad \ppsi(t)_2=\u_2+\int_0^t
        \lambda_{21}(\exp(\ppsi(t)_1-\ppsi(t)_2)-1)
        ds,
    \end{equation*}
whose solution is given by
\begin{align*}
    \ppsi(t)_i
&=
\log\Big(\frac {\lambda_{12}1_{\{i=1\}}-\lambda_{21}1_{\{i=2\}}} {\lambda_{12}+\lambda_{21}}(e^{\u_1}-e^{\u_2})e^{-(\lambda_{12}+\lambda_{21})t}+\frac {\lambda_{21}e^{\u_1}+\lambda_{12}e^{\u_2}} {\lambda_{12}+\lambda_{21}}\Big).
\end{align*}
Observe that writing $\ppsi(t)_i$ as $\u_i+\int_0^t \partial_s\ppsi(s)_i ds$ we get
\begin{align*}
    \ppsi(t)_i&=\u_i+\int_0^t
    \frac {-(\lambda_{12}1_{\{i=1\}}-\lambda_{21}1_{\{i=2\}})(\lambda_{12}+\lambda_{21})(e^{\u_1}-e^{\u_2})e^{-(\lambda_{12}+\lambda_{21})s}} { (\lambda_{12}1_{\{i=1\}}-\lambda_{21}1_{\{i=2\}}) (e^{\u_1}-e^{\u_2})e^{-(\lambda_{12}+\lambda_{21})s}+{\lambda_{21}e^{\u_1}+\lambda_{12}e^{\u_2}} }ds,
\end{align*}
which is well-defined for each $\u_1,\u_2\in \C$  such that the denominator is different from 0 for each $s\in [0,T]$ and $i\in \{1,2\}$. For all such $\u_1,\u_2\in \C$ the maps $\ppsi(t)$ solves \eqref{psifindim}. 
Note that to guarantee well-definiteness it  is sufficient to verify that $e^{\u_1}\neq \alpha e^{\u_2}$ for each $\alpha\in \R$.

\end{example}

To apply the affine-holomorphic formula, one key condition is proving the existence of a solution to the Riccati ODE \eqref{eq:ODE sequence space affine}. In general, the existence of such an equation can be deduced by the existence of a non-vanishing solution of the corresponding linear ODE.  This is derived from Proposition 4.36 in \cite{CST:23}, which we restate here for the reader's convenience.

\begin{proposition}\label{prop: existence riccati}
    
Fix $\u\in \Scal^*$ and $\v$ such that $\exp(h_\u)=h_\v$. Let $(\c(t))_{t\in [0,T]}$ be a solution of \eqref{eq:ODE sequence space holomorphic} for the initial condition $\c(0)=\v$. Assume that ${\bf c}(t)_{0}\neq0$ 
for each $t\in[0,T]$. Then there exists a solution of the Riccati equation \eqref{eq:ODE sequence space affine} with initial condition $\ppsi(0)=\u$ given by
\begin{align*}
\bm\psi(t)_{0}&=\u_{0}+\int_0^t\frac{L{\bf c}(s)_{0}}{{\bf c}(s)_{0}}ds,\qquad
\text{and}\qquad {\bm\psi(t)}_\alpha=\Big(\sum_{k=1}^\infty{(-1)^{k-1}(k-1)! }{{\bf d}(t)^{* k}}\Big)_\alpha,
\end{align*}
for ${\bf d}(t)_{0}=0$ and ${\bf d}(t)_\alpha:={{\bf c}(t)_\alpha}/{{\bf c}(t)_{0}}$, for $|\alpha|>0$. 
\end{proposition}

As a direct application of Proposition \ref{prop: existence riccati}, 
Corollary~\ref{coro: levy1 bounded support}, and Proposition \ref{prop: affine}, the next results follow directly.

\begin{corollary}\label{coro: levy affine}
     Fix $S\subseteq\Rset$ and let  $X=(X_t)_{t\in[0,T]}$ be an $S$-valued Lévy process with characteristics $(b,a,F)$ and initial condition $X_0=x_0$. Assume that $F$ has bounded support and fix $\u \in \SS$ such that $\exp(h_\u)$ satisfies the conditions of Corollary \ref{prop: affine}. Then, there exists a solution of the Riccati equation \eqref{eq:ODE sequence space affine} with initial condition $\ppsi(0)=\u$.
\end{corollary}

\begin{corollary}
    Fix $S\subseteq\Rset$ and let  $X=(X_t)_{t\in[0,T]}$ be an $S$-valued affine processes that satisfied the conditions of Proposition \ref{prop: affine}. If for some $\u \in \SS$  $\exp(h_\u)$ satisfies the conditions of Corollary \ref{prop: affine}, then there exists a solution of the Riccati equation \eqref{eq:ODE sequence space affine} with initial condition $\ppsi(0)=\u$.
\end{corollary}

\begin{remark}
\begin{enumerate}
\item  It is important to observe that for Theorem~\ref{th: affine formula} to holds we would still need to prove  that the state space is contained in the domain of convergence of $h_{\ppsi(t)}$ and that conditions \eqref{eq: ODE Riccati pointwise} and~\ref{affine formula iii} are satisfied.

    \item    By Remark~\ref{remarkcarino}\ref{remarkcarinoii}, Corollary \ref{coro: levy affine} applies to coefficients of the form  $$\u=(\u_0,0,\u_2,0,\ldots,0,\u_{2n}),$$ with $\u_{2k}\in \R_-$ and $n\in \Nset$ .

   \item The existence of a solution to the Riccati ODE \eqref{eq:ODE sequence space affine} has been proven by \cite{ALL:24} for 
\begin{align*}
    dX^{(1)}_t&=g_0(t)p(X_t^{(2)})(\rho dW_t+\sqrt{1-\rho^2}dW_t^\perp)
    -\frac 1 2 g_0(t)^2p(X_t^{(2)})^2 dt, 
    \\
    dX^{(2)}_t&=(a+bX_t^{(2)})dt+cdW_t,
\end{align*}
where $a,b,c\in \R$ with $c\neq 0$, $p$ is a power series satisfying some technical conditions and $g_0:[0,T]\to \R$. For $\alpha=0$, $\beta=\alpha \e^{-1}$ and $c=\e^\alpha$,
this covers in particular the case of the \emph{quintic OU volatility model}, where $p$ is a polynomial of degree 5 and $g_0(t)=\xi_0/\sqrt{\E[p(X_t)^2]}$, and the \emph{one-factor Bergomi model}, where 
$$p(x)=\exp\Big(\frac {\eta p(X_t)} 2\Big)\qquad \text{and}\qquad g_0(t)=\xi_0\exp\Big(-\frac{\eta^2\E[p(X_t)^2]}4\Big).$$
The existence has been proven (in particular) for initial conditions $\u$ such that $h_\u(x)=g_1x_1+g_2x_2$  for $g_1,g_2$ with vanishing and nonpositive real part, respectively. To cast it in the current setting we should choose $g_0$ constant. To be able to conclude the results of Theorem~\ref{th: affine formula} we would still need to prove conditions \eqref{eq: ODE Riccati pointwise},~\ref{affine formula ii} and~\ref{affine formula iii}.
\end{enumerate}
\end{remark}

\appendix
\section{A sufficient condition for interchanging summation and integration}

 The next lemma states that continuity of the map \eqref{eq:continu} is a sufficient condition for the term-by-term integration of a convergent power series. Recall that integrals of (vectors of) sequence-valued maps are computed componentwise (see Section~\ref{sec212}).
\begin{lemma}\label{prop:Morerageneral}
Let $F$ be a non-negative measure on a measurable space $E$, $R\in(0,\infty]^d$ and for some $\vare>0$, let
$ f: P^d_{R+\vare}(0)\times E\rightarrow\mathbb{C}^d$ be a map such that 
\begin{enumerate}
\item\label{it:continu} for all $z\in P^d_{R+\vare}(0)$,
    $f^j(z,\cdot)\in L^1(E,  F)$ and the map
    \begin{align}\label{eq:continu}
       P^d_{R+\vare}(0) \ni z\longmapsto f^j(z,\cdot)\in L^1(E,  F)
       \end{align}   is continuous for each $j$;

    \item\label{eq:holo} for all fixed $ y\in E$ it holds $f(\cdot, y)\in H(P^d_{R+\vare}(0),\C^d)$.
    \end{enumerate}
     Let $\underline{\mathbf{f}}( y)\in ({\mathcal{P}^d_R(0)}^*)^d$ denote the coefficients determining the power series representation of  $f(\cdot, y)$ on $P^d_R(0)$. 
     Then
     \begin{equation}\label{eqn3}
         \int_{E} f(\cdot, y) \ F(d y)\in H(P^d_{R+\vare}(0),\C^d)
     \end{equation} 
     and $\int_{E}\underline{\mathbf{f}}( y) \ F(d y) $ are the coefficients determining its power series representation.
\end{lemma}

\begin{proof}
Fix $i,j\in\{1,\ldots,d\}$, $a\in P^d_{R+\vare}(0)$, and $\delta>0$ such that $P^d_{\delta}(a)\subset P^d_{R+\vare}(0)$. Consider the functions $f^j_i:P^1_\delta (a_i)\rightarrow\C$ and $g_i: P^1_\delta (a_i)\rightarrow\C$ given by
    \begin{align*}
    f^j_i(z_i,y)&:=f^j(a_1,\dots,a_{i-1},z_i,a_{i+1},\dots,a_d, y),\\
        g_i^j(z_i)&:= \int_{E} f^j_i(z_i, y)  
 F(d y).
    \end{align*}

    By Hartogs' theorem (see Proposition~\ref{propclass}\ref{class:it6}) in order to show \eqref{eqn3} it suffices to show that $g_i^j\in H(P^1_{R_i+\vare}(0),\C)$. 
By Morera's theorem (see Proposition~\ref{propclass}\ref{class:it5}) this follows by showing that $g_i^j$ is continuous on $ P^1_\delta (a_i)$ and 
$$
    \int_{\triangle}g^j_i(z_i) dz_i =0,
$$
for every triangle $\triangle\subset P^1_\delta (a_i)$.
Observe that condition~\ref{it:continu} implies continuity of the map
$\int_{E} f^j(\cdot, y)   F(d y)$ on $P^d_{R+\vare}(0)$ and thus of $g^j_i$ on $ P^1_\delta (a_i)$.
    Next, fix a  triangle $\triangle\subset P^1_\delta (a_i)$ and note that by ~\ref{it:continu} we get
    \begin{align*}
        &\int_{\triangle}\int_{E}|f^j_i(z_i, y)|  F(d y) dz_i
        =\int_{\triangle}\|f^j_i(z_i, \cdot)\|_{L_1} dz_i<\infty.
    \end{align*}
    Therefore, by Fubini's theorem 
$
\int_{\triangle}\int_{E}f^j_i(z_i, y)  F(d y) dz_i
        =\int_{E}\int_{\triangle}f^j_i(z_i, y) dz_i   F(d y).
$
    Since for each $y\in E$ the map
$f^j_i(\cdot,y)$
    is holomorphic on $P^1_\delta (a_i)$ by~\ref{eq:holo}, by Goursat's theorem (see Proposition~\ref{propclass}\ref{class:it4}) we can conclude that
    \begin{align*}
 \int_{\triangle}g^j_i(z_i) dz_i=\int_{E}\int_{\triangle}f^j_i(z_i, y) dz_i   F(d y)=0.
    \end{align*} 
Next, by the Taylor expansion (see Proposition~\ref{propclass}\ref{class:it3}), for all $z\in P^d_{R}(0)$ it holds
    \begin{align*}
        \int_{E}f^j(z, y)\  F(d y)=\sum_{\alpha\in \N^d_0}D^\alpha \left(\int_{E}f^j(\cdot, y)\  F(d y)\right) (0)\ \frac{z^\alpha}{\alpha!}.
    \end{align*}
    The claimed representation of the coefficients of $\int_{E} f(\cdot, y)  F(d y)$ can thus be proven verifying that 
    \begin{equation}\label{eqn6}
        D^\alpha \left(\int_{E}f^j(\cdot, y)\  F(d y)\right) (0)=\int_E \underline{\mathbf{f}}(y)^j_\alpha F(dy).
    \end{equation}
    By the Cauchy integral formula (see Proposition~\ref{propclass}\ref{class:it1}), fixing $0<M<R$   
    \begin{align*}
     D^\alpha \left(\int_{E}f^j(\cdot, y)\  F(d y)\right)(0)=\frac{\alpha!}{(2\pi i)^d}\int_{T^d_{M}(0)} \int_{E}f^j(z, y)\ F(d y) \frac{1}{z^{\alpha +1}}  \ dz,
    \end{align*}
    where $T^d_{M}(0)$ denotes the polytorus centered at $0\in \C^d$ and with polyradius $M$ in the sense of \eqref{eqn5}.
    Since by condition~\ref{it:continu} 
    \begin{align*}
        \int_{T^d_{M}(0)}\int_{E}|f^j(z, y)|\ F(d y) \frac{1}{|z|^{\alpha +1}}  \ dz<\infty,
    \end{align*}
 by Fubini's theorem we get
       \begin{align*}
     D^\alpha \left(\int_{E}f^j(\cdot , y)\  F(d y)\right) (0)= \int_{E} \frac{\alpha!}{(2\pi i)^d}\int_{T^d_{M}(0)} f^j(z, y)\  \frac{1}{z^{\alpha +1}}  \ dz \ F(d y).
    \end{align*}
    By condition~\ref{eq:holo}, a further application of the Cauchy integral formula yields
    $$\frac{\alpha!}{(2\pi i)^d}\int_{T^d_{M}(0)} f^j(z, y)\  \frac{1}{z^{\alpha +1}}   dz=D^\alpha f^j(\cdot, y)(0)=\underline{\mathbf{f}}(y)^j_\alpha$$
    and \eqref{eqn6} follows.  
\end{proof}
    
\begin{remark}

     An inspection of the proof shows that the continuity assumption in~\ref{it:continu} could be replaced by requiring the map
$ \int_E f^j(\cdot, y)   F(d y)
$
      to be continuous on $P^d_{R+\vare}(0)$ for each $j$,  and the map 
      $
           \int_E \|f(\cdot, y)\|  F(d y)
      $
      to be in $L^1_{loc}(P^d_R(0))$.
\end{remark}

\section{Proofs of Section~\ref{sec:holomorphic_jd}}
Before providing the proofs of Section~\ref{sec: section 2.1}, we recall that by Definition~\ref{def:holomorphic process}, to prove that $X$ is an $S$-valued $\Vcal$-holomorphic process we need to show that for each  $f\in\Vcal$ it holds $f\in\Dcal(\Acal)$,  $\mathcal{A}f\in \mathcal{O}(S)$, and the process $N^f$ introduced in Equation \eqref{eq:localmart} is a local martingale.

\subsection{Proof of Theorem~\ref{prop: holo generalkernel}}\label{proofpropgkernel}
Fix $h\in \Vcal$, $z\in {S_\vare}$, and $\xi\in \textup{supp}(K_\vare(z,\cdot))$. Observe that since  $G_j>|z_j|+|\xi_j|$, there exists $R>0$ such that $z+\xi \in P^d_R(z)$ and $\overline {P^d_R(z)}\subset P^d_G(0)$. Thus, by Proposition~\ref{propclass}\ref{class:it3} (see also Remark~\ref{remarkholo}),
$$h(z+\xi)-h(z)-\nabla h(z)^\top\xi=\sum_{|\beta|\geq 2}\frac{1}{\beta!}h^{(\beta)}(z) \xi ^\beta.$$
Furthermore, by definition of $\mathcal{V}$, we have that $C(z):=\sup_{|\beta|\geq 2}|h^{(\beta)}(z)|<\infty$ and hence
$$\frac{1}{\beta!}|h^{(\beta)}(z) \xi ^\beta| \leq C(z)\frac{1}{\beta!} |\xi|^\beta.
$$
Since  $\sum_{|\beta|\geq 0}\frac{1}{\beta!} |\xi| ^\beta=\exp(|\xi_1|+\ldots+|\xi_d|)$ we get the bound

\begin{align}\label{eqn4}
         &\int_{\xi\in \C^d}|h(z+\xi)-h(z)-\nabla h(z)^\top\xi| K_\vare(z,d\xi) \nonumber \\
         &\qquad
         \leq \int_{\C^d}\sum_{|\beta|\geq 2}C(z)\frac{1}{\beta!}|\xi|^\beta K_\vare(z,d\xi) \nonumber \\
&\qquad\leq 
C(z)\int_{\|\xi\|\leq 1}\|\xi\|^2 K_\vare(z,d\xi)
+
C(z)\int_{\|\xi\|> 1}\exp(|\xi_1|+\ldots+|\xi_d|) K_\vare(z,d\xi),
\end{align}
which is finite  by condition \eqref{eqn1} and  condition~\ref{iitheorem1}. This in particular implies that $h\in\mathcal{D}(\mathcal{A})$.
Next, by \eqref{eqn4} and the dominated convergence theorem we get that
$$
\int_{\C^d}\sum_{|\beta|\geq 2}\frac{1}{\beta!}h^{(\beta)}(z) \xi ^\beta K_\vare(z,d\xi)
=
\sum_{|\beta|\geq 2}\frac{1}{\beta!}h^{(\beta)}(z)h_{ \m^\beta}(z).
$$
The map $\Acal h$ can thus be written as the following pointwise limit of functions in $\Ocal(S)$:
    \begin{align}\label{eq:genps}
        \mathcal{A}h(x)= \lim_{n\to\infty}\bigg(\sum_{|\beta|=1}h^{(\beta)}(x)h_{\underline{\b}^{\beta}}(x)+\sum_{|\beta|=2}\frac{1}{\beta!}h^{(\beta)}(x)h_{\underline{\underline{\a}}^{\beta}}(x) 
        +\sum_{|\beta|=2}^n\frac{1}{\beta!}h^{(\beta)}(x)h_{ \m^\beta}(x)\bigg),
    \end{align}
    for each $x\in S$. Moreover, the limit on the RHS converges also for each $z\in S_\vare$.
   Since the sequence $(h^{(\beta)})_{|\beta|\geq2}$ is locally uniformly bounded on $P_G^d(0)$ by assumption, the sequence in \eqref{eq:genps} is locally uniformly bounded on $P^d_{(R(S)+\vare)\land G}(0)$ by condition~\ref{iiitheorem1}. Since $S_\vare\subseteq P^d_{(R(S)+\vare)\land G}(0)$, the Vitali-Porter theorem (see Proposition~\ref{propclass}\ref{class:it2}) yields then that the limit in \eqref{eq:genps} is well defined for each $z\in P^d_{(R(S)+\vare)\land G}(0)$, belongs to $H(P^d_{(R(S)+\vare)\land G}(0))$, and the convergence holds locally uniformly. This in particular implies that $\Acal h\in \Ocal(S)$ and  for all $z\in P_{R(S)}^d(0)$  by  Proposition~\ref{propclass}\ref{class:it3} it holds
    \begin{align*}
        \mathcal{A}h(z)=\sum_{|\alpha|\geq 0}D^\alpha\left(\mathcal{A}h\right)(0) \frac{z^\alpha}{\alpha!}.
    \end{align*}
  Let $\u\in \Vcal^*$ be such that $h_\u=h|_{P_{R(S)}^d(0)}$. By Weierstrass' theorem (see Proposition~\ref{propclass}\ref{class:it7})  we get 
    \begin{align*}
         D^\alpha (\mathcal{A}h_{\u})
         (0)=&\sum_{|\beta|=1}D^\alpha\big(h_{\u}^{(\beta)}h_{\underline{\b}^{\beta}}\big)(0)
         +\sum_{|\beta|=2}\frac{1}{\beta!}D^\alpha\big(h_{\u}^{(\beta)}h_{\underline{\underline{\a}}^{\beta}}\big)(0)
         +\sum_{|\beta|\geq 2}\frac{1}{\beta!}D^\alpha\big(h_{\u}^{(\beta)}h_{ \m^\beta}\big)(0)
       \\
         =&\sum_{|\beta|=1}(\u^{(\beta)}\ast \underline{\b}^{\beta})_\alpha+\sum_{|\beta|=2}\frac{1}{\beta!}(\u^{(\beta)}\ast (\underline{\underline{\a}}^{\beta}+\m^\beta))_\alpha+\sum_{|\beta|\geq 3}\frac{1}{\beta!}(\u^{(\beta)}\ast \m^\beta)_\alpha,
    \end{align*}
for all $|\alpha|\geq 0$. Since this expression corresponds to $L(\u)_\alpha$ for $L$  as in \eqref{eq:Lgeneralkernel}, this implies that
    $\Acal h_\u=h_{L(\u)}|_S$. 
    Finally, we need to show that the process $N^h$ introduced in Equation \eqref{eq:localmart} is a local martingale. By Theorem II.1.8 in \cite{JS:87} it then suffices to show that for all $h\in\mathcal{V}$ 
\begin{align}\label{localintvar1}
    \int_0^T\int_{\Rset^d} |h(X_{s^-}+\xi)-h(X_{s^-})-\nabla h(X_{s^-})^\top\xi| \ K(X_{s^-},d\xi) \ ds<\infty \ a.s.
\end{align}
Notice first that the process $(X_s)_{s\in[0,T]}$ is a.s.~\cadlag. Thus, it suffices to show that the RHS of \eqref{eqn4} is bounded on compact subsets of $S$. Observe that since the sequence $(h^{(\beta)})_{|\beta|\geq2}$ is locally uniformly bounded on $P_G^d(0)$ we  know that $C(z)$ is bounded on compact subsets of $S$.
By condition \eqref{eqn1} the same holds for the first term of \eqref{eqn4}. For the second term instead, note that  setting  $F(\xi):=\prod_{i=1}^d(\exp(-\xi_i)+\exp(\xi_i))$,  we get
$$
\exp(|\xi_1|+\dots+|\xi_d|)1_{\{\|\xi\|>1\}}\leq  F(\xi)1_{\{\|\xi\|>1\}}
\leq F(\xi)-1-\nabla F(0)^\top \xi+C'\|\xi\|^2,
$$
for some $C'>0$. Since $F\in \Vcal$, applying the Vitali Porter theorem as before, we get that
$$\int_{\R^d} F(\xi)-1-\nabla F(0)^\top\xi \ K(x,d\xi)$$
lies in $\Ocal (S)$. Using again that 
$\int_{ \R^d}\|\xi\|^2K(\cdot,d\xi)\in \Ocal(S)$   the claim follows.

\begin{remark}\phantomsection\label{rem1}
\begin{enumerate}
    \item \label{rem1i}
 In Theorem~\ref{prop: holo generalkernel},   condition
\eqref{eqGj}
 is needed for representing $h$ as a  of power series centered at every $z\in S_\varepsilon $ and  evaluable at $\xi$ for all $\xi\in \text{supp}(K_\varepsilon(z,\cdot))$. Assuming instead $$G_j>\sup_{z\in {S_\vare}}\max\{|z_j|,\sup_{\xi\in \textup{supp}(K_\vare(z,\cdot))}|z_j+\xi_j|\}$$ is not sufficient for  deriving these representations (see Remark \ref{remarkholo}).

This manipulation permits to write $\Acal h$ as a sequence of 
holomorphic functions converging pointwise on $S$. From the Stone-Weierstrass theorem we know that  uniform convergence on $S$ is not sufficient to conclude that $\Acal h$ is holomorphic on $S$. We thus need to resource to the Vitali-Porter theorem (see Proposition~\ref{propclass}\ref{class:it2}). Condition \eqref{eqGj} is applied once again to verify the respective assumptions.

    \item \label{rem1ii}The proof of Corollary~\ref{coro1dimensional1} is analogous to the one of Theorem~\ref{prop: holo generalkernel}. We stress that when $d=1$, the Vitali-Porter theorem guarantees that a locally bounded sequence of holomorphic functions converges uniformly on compact subsets of some domain if pointwise convergence holds on a set containing an accumulation point in the considered domain (see Proposition~\ref{propclass}\ref{class:it2} for more details). For this reason, instead of checking the pointwise convergence of \eqref{eq:genps} for each $z\in S_\vare$ it thus suffices to check it for each $x\in S$.
  
    This explains why in this particular case, instead of assuming $K\in \mathcal{O}(S)$, we simply require that condition \eqref{eq:Vitali_1dim_moments} is satisfied. Similarly, condition~\ref{iitheorem1} in Theorem~\ref{prop: holo generalkernel} is then replaced by the weaker one which concerns only the set $S$ which is by assumption a set of accumulation points of $P^1_{R(S)+\vare}(0)$.
    \end{enumerate}
\end{remark}

\subsection{Proof of Theorem~\ref{ps to ps}}\label{sec312}
First observe that  $\Vcal\subseteq \Dcal(J)$ directly implies that $\Vcal\subseteq \Dcal(\Acal)$. Next, fix $h_{\u}\in \Vcal$. Recall from Section~\ref{sec 214} that the assumptions on $G$ guarantee that $h_{\u}(\cdot+j_\vare(\cdot,y))\in H(P_{R(S)+\vare}^d(0))$ and the corresponding coefficients are given by $\u\circ^s\underline{\j}( y)$.
    By Lemma~\ref{prop:Morerageneral} we thus get that the map $z\mapsto\int_E Jh_\u(z,y)F(dy)$ lies in $H(P_{R(S)+\vare}^d(0))$ and the corresponding coefficients are given by
    $$\la\ast \int_{E}\u\circ^s \underline{\bm{j}} ( y)-\u-\sum_{|\beta|=1}\u^{(\beta)}\ast \underline{\bm{j}}( y)^{\ast\beta} \ F(d y).$$
This implies that $\Acal(\Vcal)\subseteq \Ocal(S)$. As the continuity of the map $z\to Jh_\u(z,\cdot)$ implies the continuity of the map
$$z\mapsto \int_E|Jh_\u(z,y)|F(dy),$$
 the claim follows as in the proof of Theorem~\ref{prop: holo generalkernel}.  
\endproof

\begin{remark}\label{rem: G Morera} 
    Observe that the condition \eqref{eqnG} on the polyradius $G$ is needed to derive the representation \eqref{eq:OpLu} in terms of the operation on the coefficients $\circ^s$ introduced in Section~\ref{sec 214} (see Remark \ref{remarkholo}). 
\end{remark}

\subsection{Proof of Corollary~\ref{prop:unboundedjs}}\label{proofprof:unboundedjs}
We verify the conditions of Theorem~\ref{ps to ps}.
 Fix $h\in  \Vcal$ and note that by definition of $G$ there is an $\eta>0$ such that setting $R_z:=\eta+\sup_{y\in \text{supp}(F)}|j_\vare(z,y)|\land 1$ we get
 $$\overline{P_{R_z}^d(z)}\subseteq P_G^d(0)$$
 for each $z\in P^d_{R(S)+\vare}(0)$. Since $|j_\vare(z,y)|1_{\{\|j_\vare(z,y)\|\leq 1\}}$ is bounded away from $R_z$, we get by Proposition~\ref{propclass}\ref{class:it3} 
 $$|h(z+j_\vare(z,y))-h(z)-\nabla h(z)^\top j_\vare(z,y)|1_{\{\|j_\vare(z,y)\|\leq 1\}}\leq C \|f|_{T_{R_z}^d(z)}\|_\infty\|j_\vare(z,y)\|^21_{\{\|j_\vare(z,y)\|\leq 1\}},$$
 for some $C$ not depending on $z$ and $y$. Observe also that by
 definition of $\Vcal$ it holds
 $|h(w+\xi)|\leq C' v(\|w+\xi\|)$ and 
 for $\|\xi\|>1$ we get
\begin{align*}
    |h(z+\xi)-h(z)-\nabla h(z)^\top \xi|
    &\leq C'|v(\|z+\xi\|)|
 +(|h(z)|+\|\nabla h(z)\|) \|\xi\|,
\end{align*}
 for some $C'>0$. By continuity of $h$ and its derivatives we thus obtain
 $$\sup_{w\in P_{\delta_z}^d(z)}|Jh(w,y)|
 \leq C''\sup_{w\in P_{\delta_z}^d(z)}  \Big(v(\|w+j_\vare(w,y)\|)1_{\{\|j_\vare(w,y)\|>1\}}
 +\|j_\vare(w,y)\|\land \|j_\vare(w,y)\|^2
 \Big),$$
 for some $C''>0$.
The claim now follows by the dominated convergence theorem.\qed

\begin{remark}\phantomsection\label{rem: G Morera corollary}
\begin{enumerate}
    \item \label{rem: G Morera corollary i} Notice that here the condition on the polyradius $G$ is also necessary to apply Proposition~\ref{propclass}, which provides an estimate of the Taylor approximation of a complex-valued holomorphic function. 
    \item An inspection of the proof of Corollary~\ref{prop:unboundedjs} shows that the claim follows also by considering the slightly larger set of holomorphic functions given by 
\begin{align*}
    \{h\in H(P^d_G(0)) \colon
     \sup_{z\in P^d_{R(S)+\vare}(0)}|h(z)|v(\|z\|)^{-1}<m \}.
\end{align*}
\end{enumerate}
    
\end{remark}

The proof of Theorem~\ref{proof: levy processes 1} and Corollary~\ref{coro: levy1 bounded support} strongly relate to the Lévy-Kinchine formula for the moment generating function (Theorem~25.17 in \cite{Sato:13}), which states that, under some integrability conditions, for each $|\tau|\leq \alpha$ it holds
\begin{equation}\label{expmom}
    P_t \exp(\tau(\cdot))(x)=\exp(\tau x+t \psi(\tau)),
\end{equation}
for $\psi(\tau)=b\tau+\frac 1 2 a\tau^2+\int_E \exp(\tau y)-1-\tau y F(dy)$,
and $P_t$ denoting the semigroup $(P_t)_{t\in[0,T]}$ given by $P_th(x)=\E[h(X_t+x)]$,  for each $t\in [0,T]$, $x\in\Rset$ and measurable map $h$ for which $\E[|h(X_t+x)|]<\infty$.

\subsection{Proof of Theorem~\ref{prop: levy processes 1}}\label{proof: levy processes 1}
Note that 
\begin{equation}\label{firsteq}
    \Vcal = \{ h \in H(\C)\colon |h(z)|\leq C \exp(\tau |z|) \text{ for }\tau<\alpha, \ C>0\}.
\end{equation}
 ~\ref{item: levy1 i}: Since the jump size $j_\vare$ is constant in its first argument, condition \eqref{eq:theorem_unboundejs} for $v(t):=\exp(\alpha t)$ is verified. The statement then follows by Corollary~\ref{prop:unboundedjs}. 

\ref{item: levy1 ii}:  Observe that $\mathcal A$ is $g$-cyclical for $g(x):=\exp(\alpha  x)+\exp(-\alpha  x)$ in the sense of Definition~\ref{def:g-cyclical}. Moreover, by definition of $\Vcal$ for each $h\in \Vcal$ there is a $p>1$ such that  $\|h^p\|_g<\infty$. Next, by Proposition~\ref{propE3} we know that $h',h''\in \Vcal$. Noting that
$$|h(x+y)-h(x)-h'(x)y|
\leq \sup_{|y|\leq 1}|h''(x+y)||y|^2+
(|h(x+y)|+|h(x)|+|h'(x)||y|)1_{|y|>1},$$
we also get that $\Acal h\in \Vcal$ and hence $\|(\Acal h)^p\|_g<\infty$ for some (possibly different) $p>1$. Proceeding as in the proof of Lemma~\ref{lemma: lemma moment formula weights}\ref{it ii weight} the claim follows.

\ref{item: levy1 iii}:  
Consider the semigroup $(P_t)_{t\in[0,T]}$ given by $P_th(x)=\E[h(X_t+x)]$, for $h\in \mathcal{V}$, $t\in [0,T]$, $x\in\Rset$. Notice that it is well-defined since
$P_t|h|(x)=\E[|h(X_t+x)|]<\infty,$ by Theorem 25.17 in \cite{Sato:13} and definition of $\Vcal$. Fix $h_\u\in \mathcal{V}$ and recall that we already showed that $\mathcal{A}h_\u\in \mathcal{V}$ and thus $|\mathcal{A}h_\u|\leq C\exp(\alpha |\cdot|)$. To show that $\int_0^tP_s|\Acal h_\u|(x)ds<\infty$ it suffices to note that 
\begin{align}\label{eq: PsAh bounded levy}
\exp(\alpha|x|)\leq \exp(\alpha x)+\exp(-\alpha x)
\end{align}
and \eqref{expmom} can be applied.
Next, we prove that for all $t\in [0,T]$, $P_th_\u\in\Vcal$. Fix $t\in [0,T]$ and for each $z\in \C$ set
$$P_th_\u(z)=\E[h_\u(X_t+z)].$$ 
Observe that by \eqref{firsteq} there is a $\tau<\alpha$ such that 
\begin{equation}\label{finitemom}
  \E[|h_\u(X_t+z)|]  \leq  C\E[\exp(\tau |X_t+z|)]
  \leq C\exp(\tau |z|) \E[\exp(\tau |X_t|)],
\end{equation}
for each $z\in \C$. Therefore by the dominated convergence theorem the map $z\mapsto \E[h_\u(X_t+z)]  $ is continuous and
$\int_\triangle \E[|h_\u(X_t+z)|] dz $
is finite for every triangle $\triangle\subset \C$. By Fubini's and Goursat's theorem (see Proposition~\ref{propclass}\ref{class:it4}), we then get $\int_\triangle \E[h_\u(X_t+z)] dz = \E[\int_\triangle h_\u(X_t+z)dz]=0.$ Finally, by Morera's theorem (see Proposition~\ref{propclass}\ref{class:it5}) $P_th_\u\in H(\C)$ and by \eqref{finitemom} and \eqref{firsteq} $P_th_\u\in \Vcal$. Finally, notice that since $\mathcal{A}h_\u\in \mathcal{V}$, by the previous reasoning $P_t\mathcal{A}h_\u\in \mathcal{V}$, implying in particular that $S\ni x\mapsto P_t\mathcal{A}h_\u(x)$ is continuous, concluding the proof.

Finally, let us denote by $\c(t)_{t\in [0,T]}$ such solution and recall from Lemma~\ref{lemma:AP=PA} that for all $x\in S$, $h_{\c(t)}(x)=P_th_\u(x)$ and $\mathcal{A}P_th_\u(x)=P_t\mathcal{A}h_\u(x)$. Since $\c(t)\in \Vcal^*$, by~\ref{item: levy1 ii} condition~\ref{moment formula ii} of Theorem~\ref{th: moment formula holomorphic} is verified. Next, notice that again by Lemma~\ref{lemma:AP=PA} and \eqref{expmom}
\begin{align*}
    \int_0^T\int_0^T\E[|\mathcal{A}h_{\c(s)}(X_t)|]dsdt=\int_0^T\int_0^T\E[|P_{s}\mathcal{A}h_\u(X_t)|]dsdt<\infty.
\end{align*}
Since all the conditions of Theorem~\ref{th: moment formula holomorphic} are verified, the claim follows.

\subsection{Proof of Corollary~\ref{coro: levy1 bounded support}}\label{proofcor}
Set  $$\mathcal{B}=\{h\in H(\C)\colon |h|_\R|\leq C\exp(a|\cdot|)\text{ for some }a\in \Rset, C>0\}.$$
and notice that $\mathcal{B}\subseteq\mathcal{V}$.

Condition~\ref{item: levy2 i} follows by Remark~\ref{remark:nofinitemoments}\ref{bounded jump size} and~\ref{item: levy2 ii} follows as in Proposition~\ref{prop: levy processes 1}.   We highlight in particular that $\mathcal{A}(\mathcal{V})\subseteq \mathcal{B}$.

       ~\ref{item: levy2 iii}:  We proceed as in the proof of Theorem~\ref{prop: levy processes 1} showing that the conditions of Lemma~\ref{lemma:AP=PA} are satisfied. Consider the semigroup $(P_t)_{t\in[0,T]}$ given by $P_th(x)=\E[h(X_t+x)]$, for $h\in \mathcal B$, $t\in [0,T]$, $x\in\Rset$.
Since for each $h\in \Vcal$ it holds $\mathcal{A}h\in \mathcal{B}$, following the proof of Theorem~\ref{prop: levy processes 1}\ref{item: levy1 iii} we get that $\int_0^TP_s|\mathcal{A}h|(x)ds<\infty$ for all $x\in S$ and $N^h$ is a true martingale.

Fix $h_\u\in \mathcal{V}$, for $\u$ as in~\ref{item: levy2 iii}. 
Observe that by Proposition~\ref{propfur} we get 
        \begin{align}\label{eq: inversion Fourier}
            h_\u(x)=\int_\Rset \hat h_\u(u)\exp(iux)du,
        \end{align}
        for a.e.~$x\in \Rset$. Since  $h_\u$ is continuous on $\R$ by assumption, and the right-hand side of \eqref{eq: inversion Fourier} is continuous by \eqref{eq:condition Fourier} and dominated converge theorem, the equality in \eqref{eq: inversion Fourier} holds for every $x\in \R$. 
        Finally, for $t\in [0,T]$ and $x\in \Rset$, 
        by Fubini's theorem we have that
           \begin{align*}
    P_t h_\u(x)=\int_{\Rset}\E[\exp(iuX_t)]  \exp(iux)  \hat h_\u(u) du.
    \end{align*}
    Since the expectation on the right hand side is bounded by 1, by \eqref{eq:condition Fourier} we can extend this map to $\C$ 
      and use Morera's theorem to deduce that $P_th_\u$ is an entire map bounded on $\Rset$. Moreover, by \eqref{expmom} and Leibniz integral rule we can conclude that $P_th_\u\in \mathcal{V}$. Since $\Acal h_\u\in \Bcal$ by \eqref{expmom}  $P_t\Acal h_\u$ is continuous on $S$. The last part of the proof follows as in Theorem~\ref{prop: levy processes 1}.
    \qed

\subsection{Proof of Proposition~\ref{prop: affine}}\label{proof: prop affine}
\ref{itemaffini i}: This follows by Remark~\ref{remark:nofinitemoments}\ref{bounded jump size}. 

\ref{itemaffini ii}:  Observe that for $h\in \Vcal$ by the characterization of the characteristic of an affine process (see Lemma 2 in \cite{FL:20})
we have that
$|\Acal h(x)|\leq C(1+|x|).$
Since $X$ is also a polynomial process we know that
$$\E[\sup_{t\in [0,T]} (1+|X_t|)]\leq C\E[(1+X_T^2)]<\infty,$$
showing that $N^h$ is a uniformly integrable local martingale and thus a true martingale.

Next fix $\u$ as in~\ref{itemaffini iii} and consider the semigroup given by $P_th(x):=\E[h(X_t)|X_0=x]$ for each measurable $h$ such that $\E[|h(X_t)||X_0=x]<\infty$. We verify that the remaining conditions of Lemma~\ref{lemma:AP=PA} are satisfied. Observe that by Fubini we have that
$$P_th_\u(x)=\E\Big[\int_{-\tau}^\tau \exp((\varepsilon+iu)X_t)g(u) du\Big|X_0=x\Big]
=\int_{-\tau}^\tau \exp(\phi(t,\varepsilon+iu)+\psi(t,\varepsilon+iu)x)g(u) du.$$
Since by the dominated convergence theorem we know that $u\mapsto \phi(t,\varepsilon+iu)$ and $u\mapsto \psi(t,\varepsilon+iu)$ are continuous for each $t$ and $\e$, the map 
$$x\mapsto \exp(\phi(t,\varepsilon+iu)+\psi(t,\varepsilon+iu)x)g(u)$$
can be extended to $\C$ and since its module is bounded on compacts an application of Fubini, Goursat and Morera (see Proposition~\ref{propclass}\ref{class:it4} and~\ref{class:it5}) yield that $P_th_\u\in H(\C)$.

Next, note that since affine processes are Feller (see Theorem 2.7 in \cite{DFS:03}) we know that $P_t$ is mapping the set of  bounded continuous functions into itself. This in particular implies that 
$$P_t\Acal h_\u(x)=\lim_{N\to \infty}\E[\Acal h_\u(X_t)\land N|X_0=x],$$
which is a sequence of continuous functions in $x$. Observe that
$$|P_t\Acal h_\u(x)-\E[\Acal h_\u(X_t)\land N|X_0=x]|
\leq 
\E[|\Acal h_\u(X_t)|1_{\{\Acal h_\u(X_t)>N\}} |X_0=x],$$
which by Cauchy-Schwartz and Markov's inequality can be bounded by
\begin{align*}
    &\E[|\Acal h_\u(X_t)|^2 |X_0=x]^{1/2}\mathbb P(\{\Acal h_\u(X_t)>N\} |X_0=x)^{1/2}\\
&\qquad\leq 
\E[|\Acal h_\u(X_t)|^2 |X_0=x]^{1/2}
\frac{\E[|\Acal h_\u(X_t)| |X_0=x]^{1/2}}{N^{1/2}}.
\end{align*}
By the polynomial property of $X$, proceeding as in~\ref{itemaffini ii} we get that the convergence is uniform on compacts and thus $P_t\Acal h_\u$ is continuous on $S$. Similarly, again by the polynomial property of $X$ we get that for each $x\in S$, $P_t|\Acal h_\u|(x)$ is bounded by a continuous function in $t$ and thus $\int_0^TP_t|\Acal h_\u|(x) dt<\infty$ for each $x\in S$.  Observe furthermore that same argument implies also that $\int_0^T\int_0^TP_{t+s}|\Acal h_\u|(x)dsdt<\infty$. The claim follows by Remark~\ref{rem: lemma AP=PA}.
\qed

\section{Proofs of Section~\ref{sec: affine holomorphic}}
Before presenting the proofs of Section~\ref{sec: affine charact}, recall that by Definition~\ref{def:affine-holomorphic process}, in order to prove that $X$ is an $S$-valued $\Vcal$-affine-holomorphic process we need to show that for each  $f\in\Vcal$ it holds $\exp(f)\in\Dcal(\Acal)$, there exists a map $\mathcal{R}:\Vcal\rightarrow\mathcal{O}(S)$ such that for all $f\in\Vcal$,
$\mathcal{A}\exp(f)=\exp(f)\mathcal{R}(f)$ and the process $N^{\exp(f)}$ introduced in equation \eqref{eq:localmart} defines a local martingale. 
\subsection{Proof of Lemma~\ref{lem:niceR}}\label{proof: theorem new}
First observe that  from $\{e^h\colon h\in \Vcal\}\cup\Vcal\subseteq \Dcal(J)$ it follows  $\{e^h\colon h\in \Vcal\}\cup\Vcal\in \Dcal(\Acal)$. This in particular implies that
 the map $P^d_{R(S)+\vare}(0)\ni z\mapsto Je^h(z,\cdot)-e^{h(z)}Jh(z,\cdot)\in L^1(E,F)$ is well defined and continuous for each $h\in \Vcal$. 
 This integrability then yields that
 \begin{align*}
\mathcal{A}\exp(h)(z)&=\exp(h(z))\left(\nabla h(z)^\top b(z)+\frac{1}{2}\text{Tr}\big(a(z)\left(\nabla ^2h(z)+\nabla h(z)^\top\nabla h(z) \right)\big)\right)\\
    &\qquad +\int_EJe^h(z,y) F(dy)\nonumber\\
&=\exp(h(z))\left(\mathcal{A}h(z)+\frac{1}{2}\text{Tr}\big(a(z)\left(\nabla ^2h(z)+\nabla h(z)^\top\nabla h(z) \right)\big)\right.\\
       & \qquad \left.+\lambda(z)\int_{E}\exp\Big(h(z+j_\vare(z,\xi))-h(z)\Big)-1-\big(h(z+j(z,\xi))-h(z)\big) \ F(d\xi)\right),
\end{align*}
for each $h\in \Vcal$, $z\in P^d_{R(S)+\vare}(0)$. The claim now follows as in the proof of Theorem~\ref{ps to ps}.
 \qed

\subsection{Proof of Corollary~\ref{theorem: affine_weigh}}\label{proof: theorem affine weight}
We prove that the conditions of Theorem~\ref{theorem: new} and Lemma~\ref{lem:niceR} are satisfied.
Proceeding as in the proof of Corollary~\ref{prop:unboundedjs}, for each $z\in P_{R(S)+\vare}^d(0)$ and $y\in E$ both quantities $\sup_{w\in P_{\delta_z}^d(z)}|Jh(w,y)|$ and $\sup_{w\in P_{\delta_z}^d(z)}|Je^h(w,y)|$ can be bounded by
$$ C\sup_{w\in P_{\delta_z}^d(z)}  \Big(\exp(m \ v(\|w+j_\vare(w,y)\|))1_{\{\|j_\vare(w,y)\|>1\}}
 +\|j_\vare(w,y)\|\land \|j_\vare(w,y)\|^2
 \Big).$$
 Observe that the constant $m$ enters in the bound due to the second quantity, of which  we know how to bound the exponent and not the whole function.

By conditions  \eqref{eq: affine weight}, \eqref{eq:M affine weight} this implies that $h, e^h\in \Dcal(J)$. 
Finally, observe that the maps
\begin{align*}
  z\longmapsto F^1_z(\cdot):=Je^h(z,\cdot)\nonumber\qquad\text{and}\qquad z\longmapsto F^2_z(\cdot):=Jh(z,\cdot)
\nonumber,
\end{align*} 
are both continuous by  the dominated convergence theorem (see the proof of Corollary~\ref{prop:unboundedjs} for the detailed argument). \qed

\subsection{Proof of Corollary~\ref{prop:affine_Levy}}\label{proof: prop_affine_Levy}
First observe that since $j(\cdot,\xi)=\xi$ is a real valued constant on $S$, the same holds for its extension $j_\vare(\cdot,\xi)=\xi$ on $P_{R(S)+\vare}^d(0)$. This in particular implies that
$$\int_{\R^d}\sup_{w\in P_{R(S)+\vare}^1(0)}\|j_\vare(w,\xi)\|\land\|j_\vare(w,\xi)\|^2 F(d\xi)=\int_{\R^d}\|\xi\|\land\|\xi\|^2 F(d\xi)<\infty.$$
Moreover, using that $\xi\in \R^d$, by definition of $\Vcal$ we get that $|h(z+\xi)|1_{\{\|\xi\|>1\}}\leq g_h(\Im(z))1_{\{\|\xi\|>1\}}$.
Since for each $z\in P_{R(S)+\vare}^d(0)$ and $\delta_z\in (0,\infty)$ such that $P_{\delta_z}^d(z)\subseteq P_{R(S)+\vare}^d(0)$ by continuity of $g_h$ it holds
$$\int_{\R^d}\sup_{z\in P_{\delta_z}^1(z)}\exp(g_h(\Im(z)))1_{\{\|\xi\|>1\}} F(d\xi)<\infty$$
the claim follows as for Corollary~\ref{theorem: affine_weigh}.\qed

\section{A primer on holomorphic functions}\label{appendix:holofunctions}

The goal of this section is to provide accurate statements and precise references for the needed results from complex analysis.
We will use the notation relative to polydiscs and polytorus introduced in Section~\ref{sec213}, as well as the multi-index notation introduced in Section~\ref{sec: multi index notation}.

\begin{definition}\label{def: holo functions}(\textit{Definition 1.2.1 in} \cite{SC:05})
    Let $U\subseteq\mathbb C^d$ be an open set and $m\in \N$.
    \begin{enumerate}
        \item {\textbf{Complex differentiable functions}} 
        
        A function $f:U\rightarrow\C^m$ is called \textit{complex differentiable} at $z_0\in U$ if for every $\varepsilon>0$ there exists a $\delta>0$ and a $\C$-linear map
    \begin{align*}
        Lf(z_0): \C^d\rightarrow\C^m
    \end{align*}
    such that  for all $z\in U$ with $\|z-z_0\|\leq \delta$, the inequality 
    \begin{align*}
        \|f(z)-f(z_0)-Lf(z_0)(z-z_0)\|\leq \vare\|z-z_0\|
    \end{align*}
    holds.
     \item {\textbf{Holomorphic functions}}
     
     A function $f:U\rightarrow\C^m$ is called \textit{holomorphic} on $U$ if it is complex differentiable at all $z_0\in U$. A holomorphic function on the whole $\C^d$ is called \textit{entire}.
    \end{enumerate}
   
\end{definition}

\begin{proposition}\label{propclass}
Let $U\subseteq\mathbb C^d$ be an open set.
\begin{enumerate}
    \item\label{class:it1}{\textbf{Cauchy integral formula}} 
 (Theorem 1.3.3 in \cite{SC:05})

     Let $f:U\to \C$ be a holomorphic function and fix $z_0\in U$ and $R\in (0,\infty)^d$ such that   $\overline{P_R^d(z_0)}\subseteq U$. Then 
    \begin{align*}
     D^\alpha f(z_0)=\frac{\alpha!}{(2\pi i)^d}\int_{T^d_{R}(z_0)} \frac{f(w)} {(z_0-w)^{\alpha +1}}   dw,
    \end{align*}
    where $\alpha+1:=(\alpha_1+1,\ldots,\alpha_d+1)$, and 
    $$|D^\alpha f(z_0)|\leq \frac {\alpha!}{R^\alpha}\|f|_{T_R^d(z_0)}\|_\infty.$$

    \item\label{class:it3}{\textbf{Taylor expansion}} (Corollary 1.5.9 in \cite{SC:05})
    
    Let $f:U\to \R$ be a holomorphic functions and fix $z_0\in U$ and $R\in (0,\infty]^d$ such that   $\overline{P_R^d(z_0)}\subseteq U$. Then 
\begin{equation}\label{eqPSR}
        f(z)=\sum_{\alpha\in \N^d_0}\frac 1 {\alpha!}D^\alpha f(z_0)(z-z_0)^\alpha,
    \qquad \text{for all }z\in P_R^d(z_0).
\end{equation}
By~\ref{class:it1} and Example~1.5.7 in \cite{SC:05}, for $R\in (0,\infty)^d$  the reminder's term can be bounded as follows
\begin{align*}
\Big|\sum_{|\alpha|\geq k+1}\frac 1 {\alpha!}D^\alpha f(z_0)(z-z_0)^\alpha\Big|
\leq\|f|_{T_R^d(z_0)}\|_\infty\prod_{i=1}^d\frac 1 {1-\frac{|z_i-z_{0i}|}{R_i}}\sum_{|\alpha|= k+1}\frac{|z-z_0|^\alpha}{R^\alpha}.
        \end{align*}
 \item\label{class:it identity}{\textbf{Identity theorem}} (Conclusion 1.2.12.2 in \cite{SC:05}) Let $f:U\rightarrow\C$ be a holomorphic function. If $U$ is connected and $f=0$ on some open set $E\subseteq U$, then $f(z)=0$ for all $z\in U$. For $d=1$ the set $E$ is just required to have an accumulation point in $U$.
    \item \label{class:it2}{\textbf{Vitali-Porter theorem}} (Exercise 1.4.37 in \cite{SC:05}) 
    
    Fix an open subset $E\subseteq U$ and consider a sequence $\{f_n\}_n$ of holomorphic functions $f_n:U\to \C$.
    Suppose that $\{f_n\}_n$ is a locally bounded sequence (uniformly bounded on every compact set of the domain of definition) and  
    \begin{equation}\label{eqn7}
    f(z):=\lim_{n\to\infty}f_n(z)
    \end{equation} exists for all $z\in E$. Then the limit in \eqref{eqn7} is well defined for each $z\in U$, the resulting map $f:U\to\C$ is holomorphic, and the convergence holds locally uniformly. For $d=1$ the set $E$ is just required to have an accumulation point in $U$ (see p.44 \cite{S:93}).

    \item\label{class:it4}{\textbf{Goursat's theorem}} (Theorem 1.1 in \cite{SS:10})
    
    Fix $d=1$ and let $\triangle\subseteq U$ be a triangle whose interior is contained in $U$. If $f:U\to \C$ is holomorphic then
    $$\int_\triangle f(z) dz=0.$$
    \item\label{class:it5}{\textbf{Morera's theorem}} (Theorem 5.1 in \cite{SS:10})
    
    Fix $d=1,R\in (0,\infty)$,   $a\in \C$, and let $f:P^1_R(a)\to \C$ be a continuous map. If for any triangle $\triangle\subset P^1_R(a)$ it holds
    $$\int_\triangle f(z) dz=0,$$
    then $f$ is holomorphic.
    \item\label{class:it6}{\textbf{Hartogs' theorem}} (see p.28 in \cite{S:92})
    
    Let $f:U\to \C$ be a partially holomorphic function, meaning that it is holomorphic in each variable  while the other variables are held constant. Then $f$ is holomorphic.
    \item \label{class:it7}{\textbf{Weierstrass' theorem}} (Theorem 1.4.20 in \cite{SC:05}).

     $H(U,\C)$ is a closed subspace of $C(U,\C)$ with respect to the locally uniform convergence. Moreover, for every $\alpha\in \N_0^d$ the linear operator
$D^\alpha:H(U,\C)\to H(U,\C)$ 
is continuous.
\end{enumerate}
\end{proposition}

\begin{remark}\label{remarkholo}
 If $U=P^d_\infty(0)$, and  $f$ thus entire, the representation in \eqref{eqPSR} holds for all $z,z_0\in \mathbb{C}$. 
 If this is not the case, then requiring that $z,z_0\in U$ is not sufficient to guarantee that $\overline{P_{|z-z_0|}^d(z_0)}\subseteq U$ and thus to derive a power series representation of $f(z)$ centered in $z_0$ and evaluated in $z-z_0$, for each $z\in U$ as in  \eqref{eqPSR}.  This is possible if $f\in H(D)$, for an open set $D$ such that $U\subseteq D\subseteq \C^d$ and   for all $z_0,z
 \in U$, $\overline{P_{|z-z_0|}^d(z_0)}\subseteq D$.
\end{remark}

We also consider a few results concerning holomorphic functions of finite order and type (see \eqref{eqn8} and \eqref{eqn10}). 
\begin{proposition}\label{propE3} 
(Theorem 2.4.1, Theorem 6.2.14, Theorem 6.2.4 and Theorem 11.1.2 in \cite{B:11} and  Theorem 19.3 in \cite{Rudin:86})
\begin{enumerate}
    \item Order and type of an entire function do not change under differentiation. 
    \item  Non-constant entire functions of order strictly smaller than 1 are unbounded on lines.
    \item\label{item entire functions bounded R iii} Let $\tau\in (0,\infty)$ and let
$B_\tau$ be the set of  
 all entire functions of exponential type  $\tau$, which are bounded on the real axis.  Then  $h\in B_\tau$ implies 
 $$|h(z)|\leq \sup_{x\in \R}|h(x)| \exp(\tau |\Im(z)|).$$
 Moreover, in this case we also have that $h'\in B_\tau$, 
$\sup_{x\in \R}|h'(x)|\leq \sup_{x\in \R}|h(x)|\tau$
for all $z\in \C$, and thus in particular 
 $$|h'(z)|\leq \sup_{x\in \R}|h(x)| \tau \exp(\tau |\Im(z)|).$$
 \item\label{exp4} {\textbf{Paley--Wiener theorem}} Let $h\in H(\C)$ be an entire function of exponential type $\tau$ such that 
 $$\int_{\R}|h(x)|^2dx<\infty.$$
 Then there exists $g\in L^2(-\tau,\tau)$ such that
 $h(z)=\int_{-\tau}^\tau g(t)e^{itz}dt$ for each $z\in \C$.
\end{enumerate}
\end{proposition}
In the spirit of the Paley--Wiener theorem but for general holomorphic functions we have the following result (see Theorem 9.11 and Theorem 9.14 in \cite{Rudin:86}). 
\begin{proposition}
    [Inversion theorems]\label{propfur} Fix $h \in H(\C)$ such that
    $$\int_\Rset |h(x)|dx<\infty \quad\text{ or }\quad \int_\Rset |h(x)|^2 dx<\infty$$ 
    and suppose that the map $\hat h:\R\to\C$ given by \eqref{eqnfur}
    satisfies
    $\int_\Rset |\hat h(u)|du<\infty$. Then
$
            h(x)=\int_\Rset \hat h(t)\exp(itx)dt,
$
for a.e.~$x\in\R$.
\end{proposition}

\bibliographystyle{abbrvnat}

\end{document}